%% file: High-d-spherical-arXiv.tex
\documentclass{article}

\usepackage{a4}
\usepackage{graphicx}
\usepackage{subfigure}
\RequirePackage{amsthm,amsmath,amssymb,amsfonts,enumerate}

\usepackage{makeidx}         
\usepackage{graphicx}        
\usepackage{multicol}        

\usepackage{amsmath}
\usepackage{amsfonts}
\usepackage{amsbsy}
\usepackage{graphics}
\usepackage{epsfig}
\usepackage{amssymb}
\usepackage{mathrsfs}
\usepackage{dsfont}

\usepackage{url}

\usepackage{color}
\usepackage[dvipsnames, table]{xcolor}
\usepackage[textsize=scriptsize]{todonotes}
\usepackage{comment}
\usepackage{subfiles}

\def\Px{\mathsf{P}}
\def\mgbb{{\boldsymbol{\gamma}}}

\def\eqd{\stackrel{{d}}{=}}

\def\supp{\mathrm{Supp}}
\def\Arg{\mathrm{Arg}}
\def\TT{^\top}
\def\Pr{\mathsf{Prob}}
\def\RR{\mathds{R}}
\def\QQ{\mathds{Q}}

\def\SX{{\mathscr X}}
\def\Ib{\mathbf{I}}
\def\xb{\mathbf{x}}
\def\Xb{\mathbf{X}}
\def\Ex{\mathsf{E}}
\def\var{\mathsf{var}}
\def\simd{\stackrel{d}{\sim}}
\def\me{\epsilon}
\def\ml{\lambda}
\def\ms{\sigma}
\newcommand{\vsp} {\vspace{0.3cm}}
\def\PP{\mathbb{P}}
\def\SB{{\mathscr B}}
\def\SE{{\mathscr E}}
\def\SJ{{\mathcal J}}
\def\SS{{\mathcal S}}
\def\0b{\mathbf{0}}
\newcommand{\bea}{\begin{eqnarray*}}
\newcommand{\eea}{\end{eqnarray*}}
\newcommand{\be}{\begin{eqnarray}}
\newcommand{\ee}{\end{eqnarray}}
\def\Rb{\mathbf{R}}
\def\ub{\mathbf{u}}
\def\zb{\mathbf{z}}
\def\CR{\mathsf{CR}}
\def\mg{\gamma}
\def\dd{\mbox{\rm d}}
\def\rad{\stackrel{\rm d}{\ra}}
\def\ra{\rightarrow}
\def\SN{{\mathcal N}}
\def\ma{\alpha}
\def\mo{\omega}
\def\e1{\mathrm{e}}
\newcommand{\fin} {\mbox{}~\hfill{\lower-0.3ex\hbox{$\triangleleft$}}}
\def\Ind{\mathds{1}}
\def\SO{{\mathcal O}}

\def\SV{\mathscr{V}}
\def\bb{\mathbf{b}}

\newtheorem{theorem}{Theorem}[section]
\newtheorem{proposition}{Proposition}[section]
\newtheorem{lemma}{Lemma}[section]
\newtheorem{corollary}{Corollary}[section]
\newtheorem{definition}{Definition}[section]
\newtheorem{remark}{Remark}[section]

\begin{document}

\title{Optimal random quantisers for spherically symmetric distributions}

\author{Luc Pronzato\footnote{(corresponding author) Laboratoire I3S, CNRS-Universit\'e C\^ote d'Azur, Sophia Antipolis, France, {\tt pronzato@i3s.unice.fr}} \ and Anatoly Zhigljavsky\footnote{School of Mathematics, Cardiff University, UK, {\tt ZhigljavskyAA@cardiff.ac.uk}}}

\date{\today}

\maketitle

\begin{abstract}
Zador's celebrated theorem is a cornerstone of optimal quantisation: it establishes both the weak limit of the empirical distribution of an optimal $n$-point quantiser in $\RR^d$ and the decay rate of the associated $L_s$-mean quantisation error. In large dimension, however, observing this asymptotic behaviour requires an astronomically large sample size. We prove that, for spherically symmetric target distributions, optimisation over all spherically symmetric distributions is a convex problem and derive an equivalence theorem that both characterises global optimality and yields a constructive algorithm. We show that, for moderate $n$, random quantisers uniformly distributed on a sphere of suitably chosen radius $R$ perform exceptionally well and, over a broad range of values of $n$, are numerically certified to be optimal among all random quantisers. Their expected distortion has an explicit integral representation that can be evaluated to arbitrary precision, and we prove concentration across random quantisers: the distortion variance tends to zero as $n\to\infty$ for fixed $d$. For $s=2$, both the optimal radius and the associated minimum expected distortion admit exact expressions. For general $s$, the optimal radius can be determined efficiently, and extreme-value theory provides useful approximations when $n$ grows with $d$. Depending on this growth rate, $R$ either converges to zero or approaches a positive limit that is independent of $s$.
\end{abstract}

{\bf keywords:} quantisation, distortion, spherically symmetric distribution, Zador's theorem, extreme-value theory


\section{Introduction}
\label{sec:intr}

For $\mu$ a measure on $\SX\subseteq\RR^d$ (or a $d$-dimensional Riemannian manifold),
with density $\varphi$ and finite moment of order $s>0$, the $(\mu,s)$-distortion on an $n$-point set $\Xb_n=\{\xb_1,\ldots,\xb_n\}\in\RR^{n\times d}$ is $D_{\mu,s}(\Xb_n)=\Ex_\mu\{ \min_{i=1,\ldots,n} \|U-\xb_i\|^s \}$, where $U\simd\mu$. Zador's celebrated theorem \cite{Zador82} states that, when $\mu$ has finite moment of order $s'>s$, the empirical distribution of an $n$-point optimal quantiser $\Xb_n^*$ minimising $D_{\mu,s}(\Xb_n)$ converges weakly, as $n\to\infty$, to the probability measure with density proportional to $\varphi^{d/(d+s)}$. It also characterises the asymptotic minimum of the normalised distortion $n^{s/d}D_{\mu,s}(\Xb_n^*)$, or equivalently of the normalised quantisation error $n^{1/d}E_{\mu,s}(\Xb_n^*)$. These asymptotic considerations are informative only when $n^{1/d}$ is sufficiently large; in large dimension, this requires astronomically large values of $n$.

The purpose of this paper is to show that for $\mu$ spherically symmetric, random quantisers with a suitable spherically symmetric distribution can achieve exceptional performance. When $n$ is large but $n^{1/d}$ is small, the best random quantisers we obtain are for designs $\Xb_n$ uniform on a sphere $\SS_{d-1}(R)$ with suitable radius $R$, and the variability of $D_{\mu,s}(\Xb_n)$ vanishes as $n$ increases (Theorem~\ref{Th:distortion-concentration}). In the particular case $s=2$, the optimal $R$ is proportional to $\Ex_\mu\{\|U\|\}$ with a constant depending only on $n$ and $d$, and the best expected $(\mu,2)$-distortion is $\Ex_\mu\{\|U\|^2\}-R^2$; see Corollary~\ref{Coro:astar-s2-exact}. For general $s$, using extreme-value theory, we derive approximations for $R$ and for the best expected $(\mu,s)$-distortion.  We then identify different asymptotic regimes where $d$ increases and $n$ grows with $d$. When $\mu$ possesses a norm-concentration property and concentrates on a sphere of radius $r$ as $d \to \infty$, the optimal radius $R$ of the random quantiser exhibits distinct behaviours depending on the growth rate of $n$ relative to $d$:
for a sub-exponential growth, $R \to 0$; for a super-exponential growth, $R \to r$; for an exponential growth $n = \ml^d(1 + o(1))$, $R \to r\sqrt{1 - 1/\ml^2}$; see Theorem~\ref{Th:joint-exact-distortion-limit}.
Overall, the paper provides a framework for generating nearly optimal quantisers in the large $n$ and $d$ situation for spherically symmetric distributions.
Sequences of nested designs with low quantisation error can thus be efficiently generated, offering a compelling alternative to the classical greedy-packing algorithm for space-filling design (see, e.g., \cite{PZ2022}). While this paper focuses exclusively on spherically symmetric distributions, the conclusion highlights that extending this approach to the more common case of quantising the uniform measure in the $d$-dimensional hypercube yields promising numerical results.

The paper is organised as follows. Section~\ref{S:distancecdt-etc} introduces the key concepts used throughout the paper. It also derives an explicit integral expression for the expected $(\mu,s)$-distortion of a random quantiser whose $n$ points are independent and identically distributed (i.i.d.) according to a spherically symmetric distribution $\PP$.
In Section~\ref{S:all-radial-optimisation}, we exploit the convexity of the expected $(\mu,s)$-distortion with respect to $\PP$ to derive a necessary and sufficient condition for optimality of a quantiser distribution $\PP^*$ and propose an optimisation algorithm. Three emblematic cases are presented: (\textit{i}) $\mu$ is uniform on the unit sphere, (\textit{ii}) $\mu$ is uniform in the unit ball, and (\textit{iii}) $\mu$ follows a multivariate spherically symmetric normal distribution. In the latter two scenarios, in addition to the case where $\PP$ is uniform on a sphere, we also analyse the cases where $\PP$ is uniform in a ball and where $\PP$ is a multivariate spherically symmetric normal distribution. Section~\ref{S:asymptotic} applies extreme-value theory to derive asymptotic results ($n \to \infty$) for random quantisers uniform on a sphere.
Section~\ref{S:conclusion} provides a brief conclusion to the paper. Several illustrative examples are presented, using the Matlab scripts available at \url{https://github.com/lpronzato-prog/Spherically-symmetric-random-designs}. Additional numerical examples are presented in the supplementary material.

\vsp
We denote by $\SB_d(\xb,r)$ the $d$-dimensional (closed) Euclidean ball with centre $\xb$ and radius $r$; we write $\SB_d(r)=\SB_d(\0b_d,r)$, and $\SS_{d-1}(r)$ denotes its $(d-1)$-dimensional boundary sphere; $\|\cdot\|$ is the $\ell_2$-norm; $\ub\TT$ is the transpose of the column vector $\ub$; $\Ib_d$ is the $d$-dimensional identity matrix.

\section{Distance c.d.f., quantisation error and distortion}\label{S:distancecdt-etc}

\subsection{Definitions and basic properties}

For any $n$-point set $\Xb_n$ we denote by $d(\cdot,\Xb_n)$ the distance function defined by $\xb\in\RR^d \mapsto d(\xb,\Xb_n)=\min_{\xb_i\in\Xb_n}\|\xb-\xb_i\|$; $\nu$ denotes a probability measure on $\RR^d$.

\begin{definition}\label{D:distance-cdf}
For $\Xb_n$ a fixed $n$-point set in $\RR^d$, the \emph{distance c.d.f.} $F(\cdot\,;\Xb_n,\nu)$ is the c.d.f.\ of the random variable $d(U,\Xb_n)$ when $U$ is distributed with $\nu$:
\bea 
F(t;\Xb_n,\nu) = \nu\left\{U\in \bigcup_{i=1}^n \SB_d(\xb_i,t)\right\}  = \nu\{d(U,\Xb_n) \leq t\}\,.
\eea
When $\Xb_n=\Rb_n$ is a random $n$-point set generated with some probability measure $\Px_n$ on $\RR^{n\times d}$, we define the \emph{mean distance c.d.f.} $F(\cdot\,; \Px_n,\nu)$ of the random variable $d(U,\Rb_n)$ by $F(t; \Px_n,\nu) = \Ex_{\Px_n}\{F(t;\Rb_n,\nu) \}$.
\end{definition}

If $\nu$ is the delta measure $\delta_\ub$ concentrated at $\ub\in\RR^d$, the mean distance c.d.f.\ is simply $F(t; \Px_n,\delta_\ub)=\Px_n \left\{ d(\ub,\Rb_n)\leq t  \right\}$.

When $\SX$ is a compact subset of $\RR^d$, we denote by $\CR(\Xb_n)= \max_{\xb\in\SX} d(\xb,\Xb_n)$ the covering radius of $\Xb_n$. If $\nu$ is equivalent to the Lebesgue measure on $\SX$, $F(\cdot\,;\Xb_n,\nu)$ has a density on $[0,\CR(\Xb_n)]$, which we denote by $f(\cdot\,;\Xb_n,\nu)$, and the essential supremum of the random variable $d(U,\Xb_n)$ equals $\CR(\Xb_n)$.
When $\nu$ is the uniform measure on $\SX$, we simply denote the distance c.d.f.\ by $F(\cdot\,;\Xb_n)$: the random variable $d(U,\Xb_n)$ is supported on $[0,\CR(\Xb_n)]$ and $F(\cdot\,;\Xb_n)$ contains all information about the filling of $\SX$ by $\Xb_n$. In particular, the $\mg$-quantiles $q_\mg(\Xb_n)$ of $F(\cdot\,;\Xb_n)$ provide useful space-filling characteristics. When $\SX$ is connected and coincides with the closure of its interior, the density $f(\cdot\,;\Xb_n)$ is strictly positive on $(0,\CR(\Xb_n))$ and $q_\mg(\Xb_n)$ is defined by
\bea 
F[q_\mg(\Xb_n);\Xb_n]=\mg \ \mbox{ for any }\mg\in(0,1)\,,
\eea
with $q_{0}(\Xb_n)=0$ and $q_{1}(\Xb_n)=\CR(\Xb_n)$.

\begin{definition}\label{D:quantisation-error}
The \emph{$L_s$-mean quantisation error} of a general probability measure $\mu$ on $\RR^d$ with finite $s$-th moment $\Ex\{\|X\|^s\}$ for a given $n$-point set $\Xb_n$ is the $L_s(\mu)$-mean (and, when $s\geq1$, the $L_s(\mu)$-norm) of the distance function $d(\cdot,\Xb_n)$:
\bea 
E_{\mu,s}(\Xb_n) = \|d(\cdot,\Xb_n)\|_{L_s} = \left[\int_\SX d^s(\xb,\Xb_n)\,\mu(\dd\xb)\right]^{1/s}\,, \quad s> 0\,.
\eea
The quantity $D_{\mu,s}(\Xb_n)=E_{\mu,s}^s(\Xb_n)$ is called the $(\mu,s)$-distortion related to $\Xb_n$ {\rm \cite{Pages97}}. When $\Xb_n=\Rb_n$ is random, we define the expected $(\mu,s)$-distortion as $D_{\mu,s}(\Px_n)=\Ex_{\Px_n}\{E_{\mu,s}^s(\Rb_n)\}$.
\end{definition}

For any $s>0$, the $(\mu,s)$-distortion satisfies (see, e.g., \cite[p.~150]{Feller71-2}):
\be
D_{\mu,s}(\Xb_n) = \Ex_\mu\{d^s(U,\Xb_n)\} = s\, \int_0^\infty t^{s-1}\,[1-F(t;\Xb_n,\mu)] \,\dd t \,. \label{moment-ds-cdf}
\ee
The result is easily obtained when $\mu$ is such that $F(\cdot\,;\Xb_n,\mu)$ has a density $f(\cdot\,;\Xb_n,\mu)$. Indeed, by swapping the order of integration we get
\bea
\Ex_\mu\{d(U,\Xb_n)\} &=& \int_0^\infty \tau\,f(\tau;\Xb_n,\mu)\,\dd \tau = \int_0^\infty \left(\int_0^\tau \dd t\right) \,f(\tau;\Xb_n,\mu)\,\dd \tau \\
&& \hspace{-1cm} = \, \int_0^\infty \left(\int_t^\infty  f(\tau;\Xb_n,\mu)\,\dd \tau\right) \,\dd t = \int_0^\infty [1-F(t;\Xb_n,\mu)] \,\dd t \,.
\eea
Now, for any $s>0$, $G^{(s)}(\cdot\,;\Xb_n,\mu)$, defined by $G^{(s)}(t\,;\Xb_n,\mu)=F(t^{1/s};\Xb_n,\mu)$ for any $t\geq 0$, is the c.d.f.\ of $d^s(U,\Xb_n)$, and we obtain
\bea
\Ex_\mu\{d^s(U,\Xb_n)\} = \int_0^\infty [1-F(t^{1/s};\Xb_n,\mu)] \,\dd t = s\, \int_0^\infty t^{s-1}\,[1-F(t;\Xb_n,\mu)] \,\dd t \,. 
\eea

\begin{remark}\label{R:distortion-evaluation}
For any probability measure $\mu$ such that $\Ex\{\|X\|^{2s}\}<\infty$, the $(\mu,s)$-distortion of an arbitrary $n$-point set $\Xb_n$ can be approximated by the Monte-Carlo estimator $D_{\mu_N,s}(\Xb_n)=(1/N)\sum_{i=1}^N d^s(U_i,\Xb_n)$, where the $U_i$ are i.i.d.\ with distribution $\mu$. This estimator satisfies the classical central limit theorem
\bea
\sqrt{N} \left[D_{\mu_N,s}(\Xb_n)-D_{\mu,s}(\Xb_n)\right] \rad \SN(0,V(\Xb_n))\,, \quad N\to\infty\,,
\eea
where $V(\Xb_n)=D_{\mu,2s}(\Xb_n)-D_{\mu,s}^2(\Xb_n)$. When $\mu$ is supported on $\SX$ compact, $\CR(\Xb_n)<\infty$ and a direct application of Hoeffding's inequality gives:
\bea
\mbox{for all } \ma\in(0,1)\,, \ \Pr\left\{|D_{\mu_N,s}(\Xb_n)-D_{\mu,s}(\Xb_n)|>\ma\, \CR^s(\Xb_n) \right\} < 2\,\e1^{-2\,N\ma^2} \,,
\eea
showing the exponentially fast concentration of $D_{\mu_N,s}(\Xb_n)$ to its mean $D_{\mu,s}(\Xb_n)$ as $N$ increases. When $\mu$ has unbounded support, depending on its tail properties, other concentration inequalities can also be derived.
\fin
\end{remark}

\subsection{Quantisation error of random quantisers}

Let $\Rb_n=\{\xb_1,\ldots,\xb_n\}$ denote a random $n$-point set generated according to a probability measure $\Px_n$ on $\RR^{n\times d}$. Our objective is to choose $\Px_n$ to minimise the expected $(\mu,s)$-distortion $D_{\mu,s}(\Px_n)$ defined in Definition~\ref{D:quantisation-error}.

From \eqref{moment-ds-cdf}, we can write
\be\label{Ls-mean-meandcdf}
D_{\mu,s}(\Px_n) =  s \int_{t\geq 0} t^{s-1}[1-F(t;\Px_n,\mu)]\,\dd t
\ee
where $F(\cdot\,; \Px_n,\mu) = \Ex_{\Px_n}\{F(\cdot\,;\Rb_n,\mu) \}$ is the mean distance c.d.f.\ of Definition~\ref{D:distance-cdf}; that is, the c.d.f.\ of the random variable $d(U,\Rb_n)$ where both $U$ and $\Rb_n$ are random. The expected $(\mu,s)$-distortion $D_{\mu,s}(\Px_n)$ is the $(\mu,s)$-distortion for the c.d.f.\ $F(\cdot\,; \Px_n,\mu)$. In this paper we focus on distortion and quantisation error, but we could have also considered the $\mg$-quantiles of $F(\cdot\,; \Px_n,\mu)$; see Remark~\ref{R:quantiles}.
Note that Jensen inequality yields the following upper bound on the expected quantisation error: $\Ex_{\Px_n}\{ E_{\mu,s}(\Rb_n) \}\leq [D_{\mu,s}(\Px_n)]^{1/s}$ for $s \geq 1$.

By swapping the order of expectations, we get
\be
F(t; \Px_n,\mu)  &=& \Ex_{\Px_n} \left\{\Ex_\mu \left\{\Ind_{\{\ub\in\RR^d:\,d(\ub,\Rb_n)\leq t\}}(U) \right\}\right\} \nonumber \\
&=&  \Ex_\mu \left\{\Ex_{\Px_n}\left\{\Ind_{\{\ub\in\RR^d:\,d(\ub,\Rb_n)\leq t\}}(U) \right\}\right\} \nonumber \\
&=& \label{eq:prod_partial_cdf} \Ex_\mu  \left\{F(t; \Px_n,\delta_U) \right\}\,,
\ee
where $F(t; \Px_n,\delta_\ub) = \Px_n  \left\{ d(\ub,\Rb_n)\leq t  \right\}$.

In the following, we call random quantiser an $n$-point set $\Rb_n$ with i.i.d.\ points having some distribution $\mathbb{P}$, so that $\Px_n(\dd\xb_1,\ldots,\dd\xb_n)=\prod_{i=1}^n \PP(\dd\xb_i)$; we shall denote this distribution $\Px_n=\PP^{[n]}$.
In Sections~\ref{S:PP-distribution-d2-Xspherical} and \ref{S:F-distribution-d2-Xspherical} we show that $F(t; \PP^{[n]},\mu)$ can be expressed in the form of a double integral when $\PP$ is spherically symmetric. This will rely on the following property which is true for any $\PP$.

For any fixed $\ub  \in \RR^d$, any $t \geq 0$ and any $\Rb_n$ with i.i.d.\ points having the arbitrary distribution $\PP$, we have
\be
F(t; \PP^{[n]},\delta_\ub) &=& 1-\prod_{j=1}^n \PP\left\{ \ub  \notin \SB_d({\xb}_j,t)  \right\} \nonumber \\
&& \hspace{-3cm} = \, 1-\prod_{j=1}^n\left(1- \PP  \left\{ \ub  \in \SB_d({\xb}_j,t)  \right\} \right) = 1-\bigg(1-\PP \left\{ \|X - \ub\| \leq t \right\} \bigg)^n\,,
\label{eq:prod_partial}
\ee
where $X \simd \PP$.

The expression \eqref{eq:prod_partial} indicates that the calculations of the mean distance c.d.f.\ $F(t; \PP^{[n]},\mu)$ and then of the expected $(\mu,s)$-distortion require the calculation of the probability
\bea 
\PP \left\{ \|X - \ub\| \leq t \right\} = \PP \left\{\ub    \in \SB_d(X,t)  \right\}\,,
\eea
which depends on $\PP$, $\ub$, and $t$. Exchanging the order of integration in \eqref{eq:prod_partial_cdf} thus leads us to consider another distance c.d.f., in which $\ub$ is fixed and plays the role of a one-point quantiser, whereas $X$ is random. This symmetry disappears for random $n$-point sets, whose $n$ i.i.d.\ points are accounted for explicitly in \eqref{eq:prod_partial}. It is noticeable that this is the only place where the size $n$ appears.

\subsection{$\PP \left\{ \|X- \ub\| \leq t \right\}$ when $\PP$ is spherically symmetric}\label{S:PP-distribution-d2-Xspherical}

By definition, the random vector $X=(X_1, \ldots,X_d)$ is spherically symmetric if it can be represented as $X \eqd  R \cdot Z^{(d)} $, where $R=\|X\|$, $Z^{(d)}=(Z_1,\ldots,Z_d)$ is uniformly distributed on the unit sphere $\SS_{d-1}(1)$, and $R$ and $Z^{(d)}$ are independent.

When $\PP$ is spherically symmetric, the distribution of $\|X  - \ub \|$ for $X\simd \PP$ only depends on $\|\ub\|$ and $d$. This distribution is specified in Proposition~\ref{th:dist} below. In the following, $\beta_{a,b}$ denotes a random variable with the Beta-density
\bea
\varphi_{a,b}(t)=t^{a-1}(1-t)^{b-1}/B(a,b) \ \mbox{ for } 0\leq t \leq 1\,,
\eea
where $B(a,b)=\Gamma(a)\Gamma(b)/\Gamma(a+b)$ is the Beta-function, $a,b>0$. The c.d.f.\ of $\beta_{a,b}$ is $\Pr\{\beta_{a,b} \leq t\}=I_t(a,b)$, where $I_t(a,b)$ is the regularised incomplete Beta-function, for which we use the convention
\be\label{convention}
I_t(\cdot,\cdot)=   \left \{\begin{array}{ll}
    0 &\;\;  {\rm for}\;\; t\leq 0 \\
  1 & \;\; {\rm for}\;\; t \geq 1 \, .\\
  \end{array}
\right.
\ee
The following lemma (see, e.g., \cite[Sect.~2.2]{FangKN90}) will be used several times.

\begin{lemma}\label{L:projections-beta}
Let $X=(X_1,\ldots,X_d) \eqd  R \cdot Z^{(d)}$ be a spherically symmetric random vector in $\RR^d$. Then, for any $m\in\{1,\ldots,d-1\}$, we have
\bea
X^{(m)}=(X_1,\ldots,X_m) \eqd  R \cdot \sqrt{\beta_{m/2,(d-m)/2}}\cdot Z^{(m)} \,,
\eea
where $R=\|X\|$, $\beta_{m/2,(d-m)/2}$ and $Z^{(m)}$ are independent.  Moreover, the joint density of $V^{(m)}=X^{(m)}/R$ is
\bea
\varphi_m(v_1,\ldots,v_m)= \frac{\Gamma(d/2)}{\pi^{m/2}\,\Gamma((d-m)/2)} \, \left(1-\sum_{i=1}^m v_i^2 \right)^{(d-m)/2-1} \ \mbox{ for } \sum_{i=1}^m v_i^2\leq 1\,.
\eea
\end{lemma}

The next proposition is the key ingredient in the derivation of exact and asymptotic results for the mean distance c.d.f.\ and the $(\mu,s)$-distortion.

\begin{proposition} \label{th:dist}
Assume that $d\geq 2$ and $X\simd\PP$ with $\PP$ spherically symmetric. Then, for any fixed $\ub \in \RR^d$ we have
\be
 \|X  - \ub \|^2 \eqd (\|\ub\|-R)^2 +4\,\|\ub\|\,R \, \beta_{\delta,\delta}\, , \label{eq:quantile2}
\ee
where $\delta=(d-1)/2$ and the random variables $R=\|X \|$ and $\beta_{\delta,\delta}$ are independent.

Moreover, when $\Rb_n\sim \Px_n=\PP_a^{[n]}$ with $\PP_a$ uniform on $\SS_{d-1}(a)$, $a\geq 0$, we have
\be\label{d2URn}
d^2(\ub,\Rb_n)\eqd (\|\ub\|-a)^2+4\,a\,\|\ub\|\,\zeta(n,d)\,,
\ee
where $\zeta(n,d)=\min_{i=1,\ldots,n} \zeta_i$ and the $\zeta_i\eqd \beta_{\delta,\delta}$ are i.i.d.
\end{proposition}

\begin{proof}
As $X$ is spherically symmetric, the distribution of $\|X  - \ub \| $ only depends on $\ub $ through $r=\|\ub\|$, and without any loss of generality we can assume that $\ub=(r,0, \ldots, 0)$. We thus have
\bea 
 \|X \!-\! \ub \|^2
\eqd  (r\!-\!X_1)^2\!+\! \sum_{j=2}^d X_j^2 =  r^2\! -\! 2\,r X_1\!  +\!R^2= (R\!-\!r)^2 \!+\!2\,rR(1\!-\!Z_1)\,, \nonumber
\eea
where we have denoted $Z_1=X_1/R$. From Lemma~\ref{L:projections-beta}, $R$ and $Z_1$ are independent and the expression of $\varphi_1(\cdot)$ gives  $(1-Z_1)/2 \eqd  \beta_{\delta,\delta}$, which yields \eqref{eq:quantile2}.
Since the $n$ points of $\Rb_n$ are i.i.d.\ with $\PP_a$, the representation \eqref{eq:quantile2} yields \eqref{d2URn}.
\end{proof}

From \eqref{eq:quantile2},  for all $ t\geq 0$ we have
\bea
\PP \left\{ \| X \!-\!\ub \|\! \leq \! { t } \right\}\!&=&  \! \Pr  \left\{(R-r)^2 +4\,Rr \, \beta_{\delta,\delta}\,  \leq  { t ^2} \right\}  \,,
\eea
with $r=\|\ub\| > 0$, and thus
\bea
 \PP  \left\{ \| X \!-\!\ub \|\! \leq \! { t } \right\}&=&\Pr  \left\{\, \beta_{\delta,\delta}\,  \leq  \frac{t^2 -(R-r)^2 }{4\,Rr}  \right\}\,. \;\;
\eea
By conditioning on the random variable $R= \|X \|$ and denoting $\Phi(\cdot )$ the c.d.f.\ of $R=\|X \|$, we obtain
\be\label{beta_approx_for_cube}
\PP  \left\{ \| X \!-\!\ub \|\! \leq \! { t } \right\}\!&=& \int_{\rho\geq 0}
\Pr  \left\{\, \beta_{\delta,\delta}\,  \leq  \frac{t^2 -(\rho-r)^2 }{4\,\rho r}  \right\} \dd\Phi(\rho )  \nonumber \\
&=& \int_{\rho\geq 0} I_\upsilon(\delta,\delta) \,\dd\Phi(\rho ) \,,
\ee
where $\delta=(d-1)/2$,
\be\label{v}
\upsilon=\upsilon(t,\rho,r)=[t^2 -(\rho-r)^2]/(4 \rho r)
\ee
and $I_\upsilon(\cdot,\cdot )$ is the regularised incomplete beta-function with the added convention~\eqref{convention}.
The integration in \eqref{beta_approx_for_cube} is over the support of the distribution of $R=\|X \|$.
If $r=\|\ub\| = 0$, then \eqref{beta_approx_for_cube} reduces to
\bea
    \PP  \left\{ \| X\!-\!\ub \|\! \leq \! { t } \right\} = \PP  \left\{\|X\| \leq t \right\} = \int_{0}^{{t}} \dd\Phi(\rho) \,.
\eea

When $\PP=\PP_a$, the uniform distribution on the sphere $\SS_{d-1}(a)$, the expression \eqref{beta_approx_for_cube} with the convention \eqref{convention} can be simplified as follows:
\be\label{PPaa}
\PP_a\left\{ \| X-\ub \| \leq t \right\} =
\left\{\begin{array}{ll}
0                      & \mbox{if } t\leq |r-a| \,,\\
I_\upsilon(\delta,\delta) & \mbox{if } |r-a| \leq t \leq a+r \,, \\
1                      & \mbox{if } a+r \leq t \,,
\end{array} \right.
\ee
where $r=\|\ub\|$ and $\upsilon=\upsilon(t,a,r)$ is given by \eqref{v}.

\subsection{$F(t; \PP^{[n]},\mu)$ and $D_{\mu,s}(\PP^{[n]})$ when $\PP$ is spherically symmetric}\label{S:F-distribution-d2-Xspherical}

From \eqref{eq:prod_partial_cdf} and \eqref{eq:prod_partial}, the mean distance c.d.f.\ $F(\cdot\,; \PP^{[n]},\mu)$ can be calculated explicitly as
\be \label{eq:prod_partial_cdf4}
F(t; \PP^{[n]},\mu) = 1- \Ex_\mu \left\{ \bigg(1-\PP \left\{ \| X-U \| \leq t \right\} \bigg)^n \right\}\,,
\ee
where $U\simd \mu$, $X \simd \PP$, and $X$ and $U$ are independent.
When $\PP$ is spherically symmetric, $\PP \left\{ \|X- \ub \| \leq t \right\} $ is given by \eqref{beta_approx_for_cube} which only depends on $\ub$ through $r=\|\ub\|$, and the expectation with respect to $U$ in \eqref{eq:prod_partial_cdf4} is reduced to an expectation with respect to the distribution of $\|U\|$. The calculation of $F(t; \PP^{[n]},\mu)$ thus amounts to a double integration, with respect to the distributions of $\|X\|$ and $\|U\|$.
From \eqref{Ls-mean-meandcdf}, the calculation of $D_{\mu,s}(\PP^{[n]})$ requires an additional integration with respect to $t$,
\be
D_{\mu,s}(\PP^{[n]}) &=&  s \int_{t\geq 0} t^{s-1} \Ex_\mu \left\{ \bigg(1-\PP \left\{ \| X-U \| \leq t \right\} \bigg)^n \right\} \,\dd t \,, \nonumber \\
&=& s \int_{t\geq 0} t^{s-1} \int_{r\geq 0} H(r,t;\Phi,n)\, \dd\Psi(r)\,\dd t \,, \label{Ls-mean-final}
\ee
where
\bea 
H(r,t;\Phi,n) =\bigg(1-\int_{\rho\geq 0} I_\upsilon(\delta,\delta) \,\dd\Phi(\rho ) \bigg)^n \,,
\eea
with $\delta=(d-1)/2$, $\upsilon=\upsilon(t,\rho,r)$ given by \eqref{v}, $\Psi(\cdot)$ the c.d.f.\ of $\|U\|$ for $U\simd\mu$ and $\Phi(\cdot)$ the c.d.f.\ of $\|X\|$ for $X\simd\PP$.

For given $\Psi(\cdot)$ and $\Phi(\cdot)$, the integrals required to compute $D_{\mu,s}(\PP^{[n]})$ can be evaluated with arbitrary precision. 
The measure $\mu$ does not need to be spherically symmetric; however, $\Psi(\cdot)$ can be challenging to derive in general cases. When $\mu$ is spherically symmetric, a distribution with minimum expected distortion is also spherically symmetric; see Remark~\ref{R:rotational-invariance}. In the next section, we demonstrate that determining such a distribution forms a convex problem, for which we derive a necessary and sufficient optimality condition and propose an algorithm inspired by optimal design theory.

\begin{remark}
The expression \eqref{Ls-mean-final} gives the expected $(\mu,s)$-distortion for random quantisers. Therefore, according to the paradigm of the so-called ``probabilistic method'' (see, e.g., \cite{AlonS2000}), for any $\PP$ there exists at least one non-random $n$-point set $\Xb_n$ with $(\mu,s)$-distortion $D_{\mu,s}(\Xb_n)\leq D_{\mu,s}(\PP^{[n]})$. See Section~1.3 of the supplementary material for an illustration involving full factorial designs.
\fin
\end{remark}

\subsection{An exact result for $s=2$ when $\PP$ is uniform on $\SS_{d-1}(a)$}\label{S:s=2:exact}

When $s=2$ and $\PP$ is the uniform measure $\PP_a$ on a sphere $\SS_{d-1}(a)$, Proposition~\ref{th:dist} can be directly applied to express both the expected distortion $D_{\mu,2}(\PP_a^{[n]})$ and the optimal radius $a^*(d,n,2)$.

\begin{corollary}
\label{Coro:astar-s2-exact}
Assume that $d\geq 2$, that $\mu$ is any probability measure on $\RR^d$ with finite second moment, and that $\Rb_n\simd\PP_a^{[n]}$ with $\PP_a$ uniform on $\SS_{d-1}(a)$, $a\geq0$. Then the expected $(\mu,2)$-distortion is exactly
\be\label{exact-D2-Pa}
D_{\mu,2}(\PP_a^{[n]})
=\Ex_\mu\{\|U\|^2\}+a^2
-2a\,\Ex_\mu\{\|U\|\}\,(1-2m_{n,d})\,,
\ee
where
\be \label{exact-mnd}
m_{n,d}=\int_0^1\left[1-I_z(\delta,\delta)\right]^n\,\dd z\,,
\qquad \delta=\frac{d-1}{2}\,.
\ee
The unique minimiser of $D_{\mu,2}(\PP_a^{[n]})$ with respect to $a\geq0$ is
\be\label{exact-a2-Pa}
a^*(d,n,2)
=\Ex_\mu\{\|U\|\}\,(1-2m_{n,d})
\ee
and the corresponding minimum expected distortion is
\be\label{exact-min-D2-Pa}
D_{\mu,2}(\PP_{a^*}^{[n]})
=\Ex_\mu\{\|U\|^2\}-[a^*(d,n,2)]^2\,.
\ee
\end{corollary}

\begin{proof}
Conditioning on $U=\ub$ in \eqref{d2URn} and then averaging with respect to $U$ give
\bea
D_{\mu,2}(\PP_a^{[n]})
&=&\Ex_\mu\left\{(\|U\|-a)^2
+4a\,\|U\|\,m_{n,d}\right\}\,,
\eea
with $m_{n,d}= \Ex\{\zeta(n,d)\}=\int_0^1\left[1-I_z(\delta,\delta)\right]^n\,\dd z$, which yields \eqref{exact-D2-Pa}. Moreover,
$m_{n,d}\leq\Ex\{\zeta_1\}=1/2$, so that the minimiser of this strictly convex quadratic function of $a$ is nonnegative and is given by \eqref{exact-a2-Pa}. Substitution gives \eqref{exact-min-D2-Pa}. Note that the dependence of the optimal radius on $\mu$ is entirely through the mean radius $\Ex_\mu\{\|U\|\}$.
\end{proof}

\section{Optimisation over all spherically symmetric distributions}\label{S:all-radial-optimisation}

\subsection{A necessary and sufficient condition for optimality}\label{S:NSC}

For $\PP$ a spherically symmetric distribution with c.d.f.\ $\Phi(\cdot)$, denote
\bea 
A(\Phi;r,t) =\int_{\rho\geq 0} I_\upsilon(\delta,\delta) \,\dd\Phi(\rho )
\eea
with $\delta=(d-1)/2$ and $\upsilon=\upsilon(t,\rho,r)$ given by \eqref{v}, so that the expected distortion $D_{\mu,s}(\PP^{[n]})$ equals
\be\label{eq:radial-functional}
\SJ_{\mu,s}(\Phi;n) = s \int_{t\geq 0} t^{s-1} \int_{r\geq 0} \left[1-A(\Phi;r,t)\right]^n\, \dd\Psi(r)\,\dd t
\ee
with $\Psi(\cdot)$ the c.d.f.\ of $\|U\|$ for $U\simd\mu$, see \eqref{Ls-mean-final}. We denote by $\SE_{\mu,s}(\Phi,n)$ the expected-distortion efficiency
\be\label{ED-efficiency}
\SE_{\mu,s}(\Phi;n)=\frac{\SJ_{\mu,s}(\Phi^*;n)}{\SJ_{\mu,s}(\Phi;n)}\,,
\ee
where $\Phi^*$ is an optimal c.d.f.\ minimising $\SJ_{\mu,s}(\Phi;n)$, with $\SE_{\mu,s}(\Phi,n)\in[0,1]$.

\begin{definition}
    For any $n\geq 1$ and $s>0$, and for $\Phi(\cdot)$ the c.d.f.\ of $\|X\|$ with $X\sim \PP$ spherically symmetric, we define the {\em sensitivity function} associated with the distortion $D_{\mu,s}(\PP^{[n]})$ by
    \be
S_{\mu,s}(\rho;\Phi,n) = ns \int_{t\geq 0} t^{s-1}\int_{r\geq 0}
\left[1-A(\Phi;r,t)\right]^{n-1}\,I_\upsilon(\delta,\delta)\,\dd\Psi(r)\,\dd t \,,
\label{eq:radial-sensitivity}
\ee
where $\delta=(d-1)/2$, $\upsilon=\upsilon(t,\rho,r)$ given by \eqref{v} and $\Psi(\cdot)$ is the c.d.f.\ of $\|U\|$ with $U\simd\mu$.
\end{definition}

The functional $\SJ_{\mu,s}(\cdot\,;n)$ and the associated sensitivity function $S_{\mu,s}(\cdot\,;\Phi,n)$ satisfy the following properties (the denomination ``Equivalence Theorem'' is a tribute to the result of Kiefer and Wolfowitz \cite{KieferW60} in optimal design of experiments).

\begin{theorem}[Equivalence Theorem]\label{Th:radial-equivalence}
For every $n\geq1$ and $s>0$, the functional $\Phi\mapsto \SJ_{\mu,s}(\Phi;n)$ is convex. If $\mu$ is supported in $\SB_d(r_{\max})$, then the infimum distortion over all spherically symmetric distributions is attained by a probability measure such that $\Phi$ is supported on $[0,r_{\max}]$. A spherically symmetric distribution $\PP^*$ such that $\Phi^*$ is supported on this interval is globally optimal among all spherically symmetric distributions if and only if
\be
S_{\mu,s}(\rho;\Phi^*,n) \leq \overline{S}_{\mu,s}(\Phi^*,n)
\quad\hbox{for every }\rho\in[0,r_{\max}]\,,
\label{eq:radial-equivalence-condition}
\ee
where
\bea
\overline{S}_{\mu,s}(\Phi,n) =
\int_{\rho\geq 0} S_{\mu,s}(\rho;\Phi,n) \,\dd\Phi(\rho) \,.
\eea
In \eqref{eq:radial-equivalence-condition}, equality holds for $\Phi^*$-almost every $\rho$ and the computable quantity
\be
g_{\mu,s}(\Phi;n)=\max_{0\leq\rho\leq r_{\max}} S_{\mu,s}(\rho;\Phi,n)
-\overline{S}_{\mu,s}(\Phi,n)
\label{eq:radial-gap}
\ee
provides a global optimality certificate in terms of expected-distortion efficiency \eqref{ED-efficiency}:
\be
1 \geq \SE_{\mu,s}(\Phi;n) \geq 1 - \frac{g_{\mu,s}(\Phi;n)}{\SJ_{\mu,s}(\Phi;n)} \,.
\label{eq:radial-gap-bound}
\ee
\end{theorem}

\begin{proof}
The map $\Phi\mapsto A(\Phi;r,t)$ is linear, while $x\mapsto(1-x)^n$ is convex on $[0,1]$. Integration in \eqref{eq:radial-functional} preserves convexity. If $\rho>r_{\max}$, replacing a design point $\xb=\rho \zb$ with $\|\zb\|=1$ by $\xb'=r_{\max}\zb$ cannot increase its distance from any $\ub\in\SB_d(r_{\max})$, since
\bea
\|\ub-\rho\, \zb\|^2-\|\ub-r_{\max}\,\zb\|^2
=(\rho-r_{\max})[\rho+r_{\max}-2\, \ub\TT \zb]\geq0.
\eea
Hence every spherically symmetric distribution can be projected onto $[0,r_{\max}]$ without increasing the distortion. The set of probability measures on this compact interval is weakly compact, and the distortion is continuous under weak convergence, a minimiser thus exists.

For two c.d.f.\ $\Phi$ and $\Phi'$, differentiation along
$\Phi_\epsilon=(1-\epsilon)\Phi+\epsilon\Phi'$ gives
\be
\frac{\dd}{\dd\epsilon} \SJ_{\mu,s}(\Phi_\epsilon;n)
\bigg|_{\epsilon=0+}
=\overline{S}_{\mu,s}(\Phi,n) -
\int_{\rho\geq 0} S_{\mu,s}(\rho;\Phi,n)\,\dd\Phi'(\rho) \,.
\label{eq:radial-directional-derivative}
\ee
Convexity implies that $\Phi^*$ is optimal if and only if this derivative is non-negative for every $\Phi'$. Taking $\Phi'$ to be a point mass at $\rho$ shows that the condition \eqref{eq:radial-equivalence-condition} is necessary. Conversely, the linearity of \eqref{eq:radial-directional-derivative} in $\Phi-\Phi'$ implies that when \eqref{eq:radial-equivalence-condition} is satisfied, then $\dd\SJ_{\mu,s}[(1-\epsilon)\Phi^*+\epsilon\Phi';n]/\dd\epsilon
\big|_{\epsilon=0+} \geq 0$ for any $\Phi'$, so that $\Phi^*$ is optimal. The condition \eqref{eq:radial-equivalence-condition} is thus sufficient. Since its left-hand side has $\Phi^*$-average $\overline{S}_{\mu,s}(\Phi^*,n)$, equality must hold $\Phi^*$-almost everywhere.
Finally, the optimality of $\Phi^*$ and convexity imply, for any $\Phi$,
\bea
\SJ_{\mu,s}(\Phi;n) \geq \SJ_{\mu,s}(\Phi^*;n) \!\!\!\! &\geq& \!\!\!\!\SJ_{\mu,s}(\Phi;n) + \dd\SJ_{\mu,s}[(1-\epsilon)\Phi+\epsilon\Phi^*;n]/\dd\epsilon
\big|_{\epsilon=0+} \\
\!\!\!\! &\geq& \!\!\!\! \SJ_{\mu,s}(\Phi;n) + \min_{0\leq\rho\leq r_{\max}} \left[\overline{S}_{\mu,s}(\Phi,n) - S_{\mu,s}(\rho;\Phi,n)
\right]\,,
\eea
that is, \eqref{eq:radial-gap-bound}.
\end{proof}

\begin{remark}\label{R:rotational-invariance} Spherical symmetry of $\PP$ is not required for convexity. Indeed, the expressions in \eqref{Ls-mean-meandcdf}, \eqref{eq:prod_partial_cdf}, and \eqref{eq:prod_partial} show that the functional
$\PP\mapsto D_{\mu,s}(\PP^{[n]})$ is convex for any measure $\mu$ on $\SX$, any $n\geq 1$ and $s>0$. 
However, exploiting convexity to derive optimal distributions is significantly simplified under the spherical symmetry condition. This is not restrictive when
$\mu$ is spherically symmetric. Indeed, when $\mu$ is invariant under rotations, since $D_{\mu,s}(\PP^{[n]})$ is convex in $\PP$ and invariant under rotations of $\PP$, by a standard symmetrisation argument there exists an optimal $\PP^*$ which is rotationally invariant.
\fin
\end{remark}

As stated in the next corollary, \eqref{eq:radial-equivalence-condition} provides a direct numerical test of whether optimisation over a single sphere has lost anything relative to optimisation over arbitrary radial mixtures.

\begin{corollary}\label{Coro:one-sphere-radial-optimality}
The distribution concentrated at one radius $a\in[0,r_{\max}]$ is globally optimal among all radial distributions if and only if
\bea
S_{\mu,s}(\rho;\delta_a,n)\leq
S_{\mu,s}(a;\delta_a,n)
\quad\hbox{for all }\rho\in[0,r_{\max}].
\eea
\end{corollary}

\vsp
For unbounded target distributions, the  equivalence condition of Theorem~\ref{Th:radial-equivalence} remains valid whenever the differentiations above are justified by the relevant moment bounds. Once an optimum $\Phi^*$ has been determined on $[0,r_{\max}]$ for a large enough $r_{\max}$, optimality over $\RR^+$ is guaranteed when \eqref{eq:radial-equivalence-condition} is satisfied for all $\rho\geq r_{\max}$.

\subsection{An optimisation algorithm}\label{S:NSC-algorithm}

Theorem~\ref{Th:radial-equivalence} yields a constructive procedure. Starting from a finitely supported measure with c.d.f.\ $\Phi_k$, calculate
\be\label{rho_k}
\rho_k\in\Arg\max_{0\leq\rho\leq r_{\max}}
S_{\mu,s}(\rho;\Phi_k,n) \,.
\ee
If $g_{\mu,s}(\Phi_k;n)$ is below a prescribed tolerance, \eqref{eq:radial-gap-bound} certifies near-optimality. Otherwise set
\be
\Phi_{k+1}=(1-\gamma_k)\Phi_k+\gamma_k\delta_{\rho_k},
\qquad
\gamma_k\in\Arg\min_{0\leq\gamma\leq1}
\SJ_{\mu,s}[(1-\gamma)\Phi_k+\gamma\delta_{\rho_k};n].
\label{eq:radial-Frank-Wolfe}
\ee
This is the classical conditional-gradient, or vertex-direction, algorithm on the space of probability measures used in optimal design \cite{Fedorov72, Wynn70}. Each iteration requires only a one-dimensional maximisation and a one-dimensional line search, preserves finite support, decreases the distortion, and supplies the rigorous stopping certificate $g_{\mu,s}(\Phi_k;n)$. It can therefore determine whether the best radial law is a single sphere or a mixture of a small number of spheres, without imposing that structure in advance. Several variants of the algorithm have been proposed in the literature; see, e.g., \cite[Sect.~9.1]{PP2013}. In the Matlab script {\tt Optimal\_Sphere\_Mixture.m} available at
\url{https://github.com/lpronzato-prog/Spherically-symmetric-random-designs},
we exploit the linearity in $\Phi$ of the map $\Phi\mapsto A(\Phi;r,t)$ to update
$\SJ_{\mu,s}(\Phi_k;n)$ and $S_{\mu,s}(\rho;\Phi_k,n)$,
and alternate between vertex-direction iterations \eqref{eq:radial-Frank-Wolfe} and vertex-exchange iterations defined by:
\bea
\Phi_{k+1} = \Phi_k + \gamma_k (\delta_{\rho_k} - \delta_{\rho_k^-})\,,
\eea
where $\gamma_k \in \Arg\min_{0 \leq \gamma \leq w_k^-} \SJ_{\mu,s}[\Phi_k + \gamma (\delta_{\rho_k} - \delta_{\rho_k^-}); n]$.
Here, $\rho_k$ is given by \eqref{rho_k}, $\rho_k^- \in \Arg\min_{\rho \in \supp(\Phi_k)} S_{\mu,s}(\rho; \Phi_k, n)$, and $w_k^-$ is the weight of $\rho_k^-$ in $\Phi_k$. Note that when $\gamma = w_k^-$, the support point $\rho_k$ replaces $\rho_k^-$.
Typically, the search domain for including new support points (here, the interval $[0, r_{\max}]$) is discretized. In \texttt{Optimal\_Sphere\_Mixture.m}, we start with a finite set of candidates for $\rho_k$. However, if two support points $\rho_i^-$ and $\rho_j^-$ of $\Phi_k$ are separated by less than some $\me$, we periodically consider replacing $\Phi_k$ with
$\Phi_k - w_i^- \delta_{\rho_i^-} - w_j^- \delta_{\rho_j^-} + (w_i^- + w_j^-) \delta_\rho$,
where $\rho$ is optimally chosen in the interval $[\rho_i^-, \rho_j^-]$ (and is included in the set of candidates for subsequent iterations). We also periodically remove all support points of $\Phi_k$ that receive negligible weights. We stop the algorithm when $g_{\mu,s}(\Phi_k;n)/\SJ_{\mu,s}(\Phi_k;n)<\me$ for a small $\me$; by \eqref{eq:radial-gap-bound}, this certifies that $\SE_{\mu,s}(\Phi_k;n)>1-\me$.

\subsection{Three examples}\label{S:NSC-examples}

\subsubsection{$\mu$ is uniform on the unit sphere $\SS_{d-1}(1)$}\label{S:NSC-examples-sphere}

Here $\mu$ is the spherical measure (uniform on $\SS_{d-1}(1)$), and the distribution of $\|U\|$ for $U\simd\mu$ is the delta measure at $r=1$. When $n=1$ and $\Xb_n=\{\0b_d\}$, the distance c.d.f.\ $F(t;\Xb_n,\mu)$ equals 0 for $t<1$ and 1 for $t\geq 1$. 

Using the algorithm described above, we obtain that, given $s$ and $d$, for very small values of $n>1$ the optimal distribution is a mixture of the delta measure at the origin and the uniform measure on $\SS_{d-1}(a)$ for some $a<1$. Figure~\ref{F:nmin_sphere} shows $n_*(d,s)$, the smallest values of $n$ such that the distribution $\PP_{a^*}$ uniform on a single sphere $\SS_{d-1}(a^*)$ is optimal, as a function of $d=2,\ldots,15$ for different $s$ (for $n=1$, the degenerate case $\PP=\delta_{\0b_d}$ is optimal). For all $n\geq n_*(d,s)$, the application of Corollary~\ref{Coro:one-sphere-radial-optimality} shows that the optimal distribution is concentrated (and uniform) on a single sphere $\SS_{d-1}(a^*)$: we simply need to plot $S_{\mu,s}(\rho;\delta_{a^*},n)$ for $\rho\in[0,1]$. This is to be contrasted with the situation met in Sections~\ref{S:NSC-examples-ball} and \ref{S:NSC-examples-normal}. In the special case $s=2$ and $d=3$, the optimality of $\PP_{a^*}$ for any $n\geq 3$ can be proved formally (as then $\delta=1$ and $I_v(1,1)=v$ for $v\in[0,1]$).

\begin{figure}[ht!]
\begin{center}
\includegraphics[width=.5\linewidth]{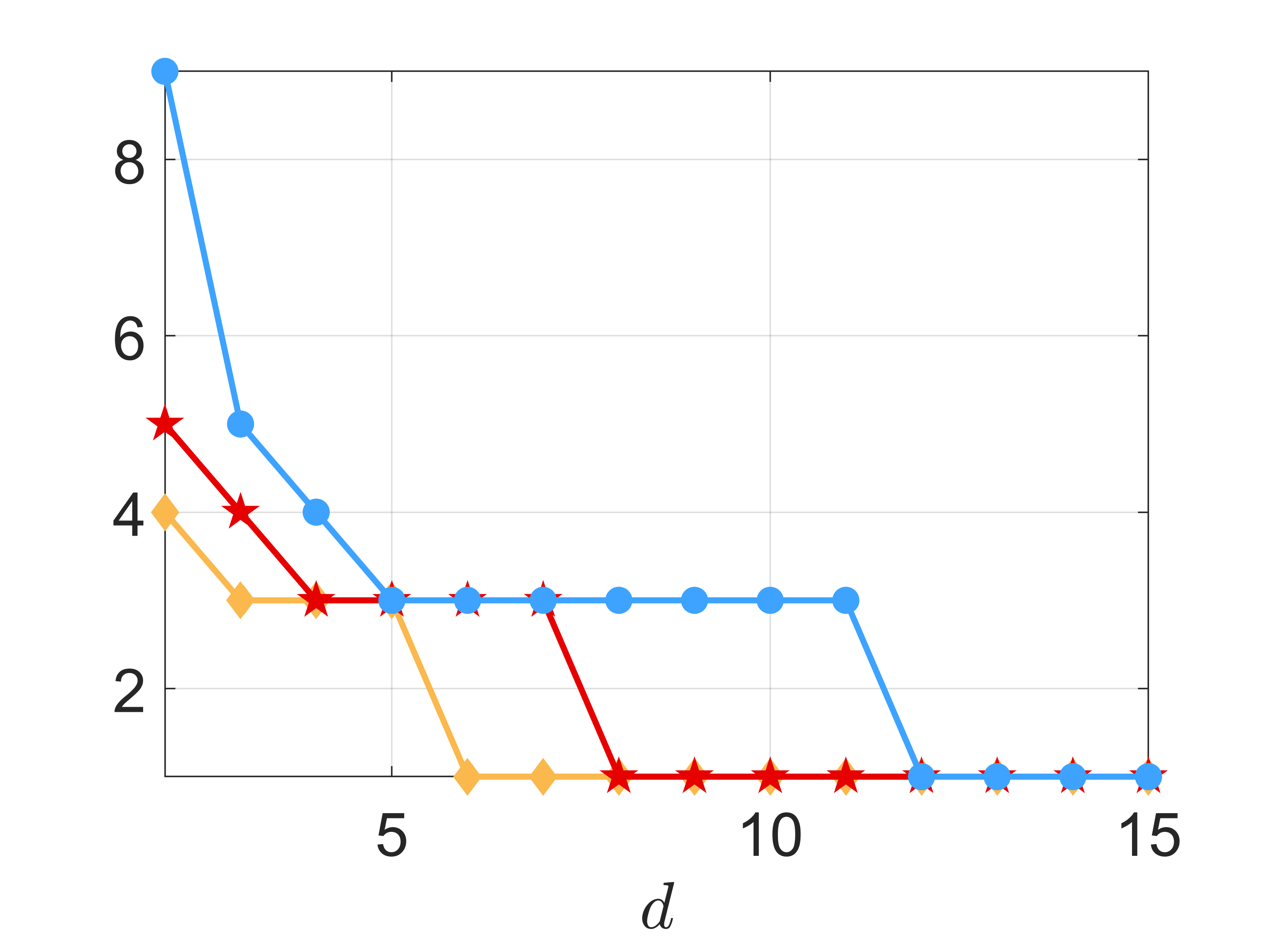}
\end{center}
\caption{\small Sphere. Smallest $n=n(d,s)$ such that $\PP$ uniform on $\SS_{d-1}(a^*)$ is optimal, for $s=2$ ({\color{Dandelion} $\blacklozenge$}), $s=4$ ({\color{red} $\bigstar$}) and $s=10$ ({\color{Cerulean} $\bullet$}).}
\label{F:nmin_sphere}
\end{figure}

The value $a^*=a^*(d,n,s)$ that minimises $D_{\mu,s}(\PP^{[n]}_a)$ with respect to $a$ can be obtained numerically with arbitrary precision for any given $d,n$ and $s$. The expression \eqref{PPaa} can be simplified as follows:
\bea 
\PP_a\left\{ \| X-\ub \| \leq t \right\} =
\left\{\begin{array}{ll}
0                      & \mbox{if } t\leq 1-a \,,\\
I_\upsilon(\delta,\delta) & \mbox{if } 1-a \leq t \leq 1+a \,, \\
1                      & \mbox{if } 1+a \leq t \,,
\end{array} \right.
\eea
where $\upsilon=[t^2 -(1-a)^2]/(4 a)$. Formula~\eqref{Ls-mean-final} then gives
\be\label{Ds-sphere}
D_{\mu,s}(\PP^{[n]}_a) =  (1-a)^s + s \int_{1-a}^{1+a} t^{s-1} \left[1-I_\upsilon(\delta,\delta)\right]^n\,\dd t \,.
\ee
In the special case $d=3$, $\delta=1$ and $I_v(1,1)=v$ for $v\in[0,1]$, so that $D_{\mu,s}(\PP^{[n]}_a)$ can be calculated explicitly for $s$ even. It is given by a polynomial in $a$ of degree $s$, with for example $D_{\mu,2}(\PP^{[n]}_a)=(1-a)^2+4\,a/(n+1)$ and $D_{\mu,4}(\PP^{[n]}_a)=(1-a)^4+8\,a[n(1-a)^2+2+2\,a^2]/[(n+1)(n+2)]$, which gives $a^*(3,n,2)=(n-1)/(n+1)$ and $a^*(3,n,4)$ as the root of a cubic equation, satisfying $a^*(3,n,4)=1-4/n+16/n^2+\SO(1/n^3)$ when $n\to\infty$. We thus have $a^*(3,3,2)=1/2$, and $\PP_{1/2}$ is globally optimal among all distributions on $\RR^3$, whereas $a^*(3,2,2)=1/3$ with $\PP_{1/3}$ being only optimal among all distributions concentrated on a sphere (as $n_*(3,2)=3$, see Figure~\ref{F:nmin_sphere}).

The left column of Figure~\ref{F:sphere} presents $a^*(d,n,s)$ as a function of $d$ for different $n$ and $s=2$ (top) and $s=10$ (bottom). For fixed $s$ and $d$, $a^*$ increases with $n$, and asymptotically, when $n$ tends to infinity, $a^*$ tends to 1 (see, e.g., \cite[Chap.~9]{GrafL2000}). Intuitively the choice $a=1$ seems natural for all $n$ and $s$. However, Figure~\ref{F:sphere} indicates that $a=1$ can be a poor choice, especially in high dimension. For any fixed $d$ and $n$, $a^*(d,n,s)$ is a decreasing function of $s$, with $a^*(d,n,\infty)=0$ if $n$ is small enough. Indeed, for $n \leq d$, $\CR(\Xb_n)\geq 1$ for any $\Xb_n$, and the inequality is strict if $\0b_d\notin\Xb_n$. To see this, choose a unit vector orthogonal to the linear span of the at most $d$ points in $\Xb_n$. This vector, viewed as a point of $\SS_{d-1}(1)$, is at distance at least one from every point of $\Xb_n$, and at distance strictly greater than one if none of those points is the origin. This implies that $a^*(d,n,s)=0$ for large enough $s$ when $n\leq d$. Also, for fixed $n$ and $s$, $a^*(d,n,s)$ decreases with $d$. The evolution of $D_{\mu,s}(\PP^{[n]}_{a^*})$ as a function of $d$ for different $n$ and $s=1,10$ is presented in the supplementary material.

The right column of Figure~\ref{F:sphere} presents the efficiency of the naive quantiser with $n$ points i.i.d.\ on $\SS_{d-1}(1)$ relative to optimised random quantisers with $n$ points on $\SS_{d-1}(a^*)$. It shows that the benefit of using an appropriate $a$ is significant when $d$ is large or $n$ is small.

\begin{figure}[ht!]
\begin{center}
\includegraphics[width=.49\linewidth]{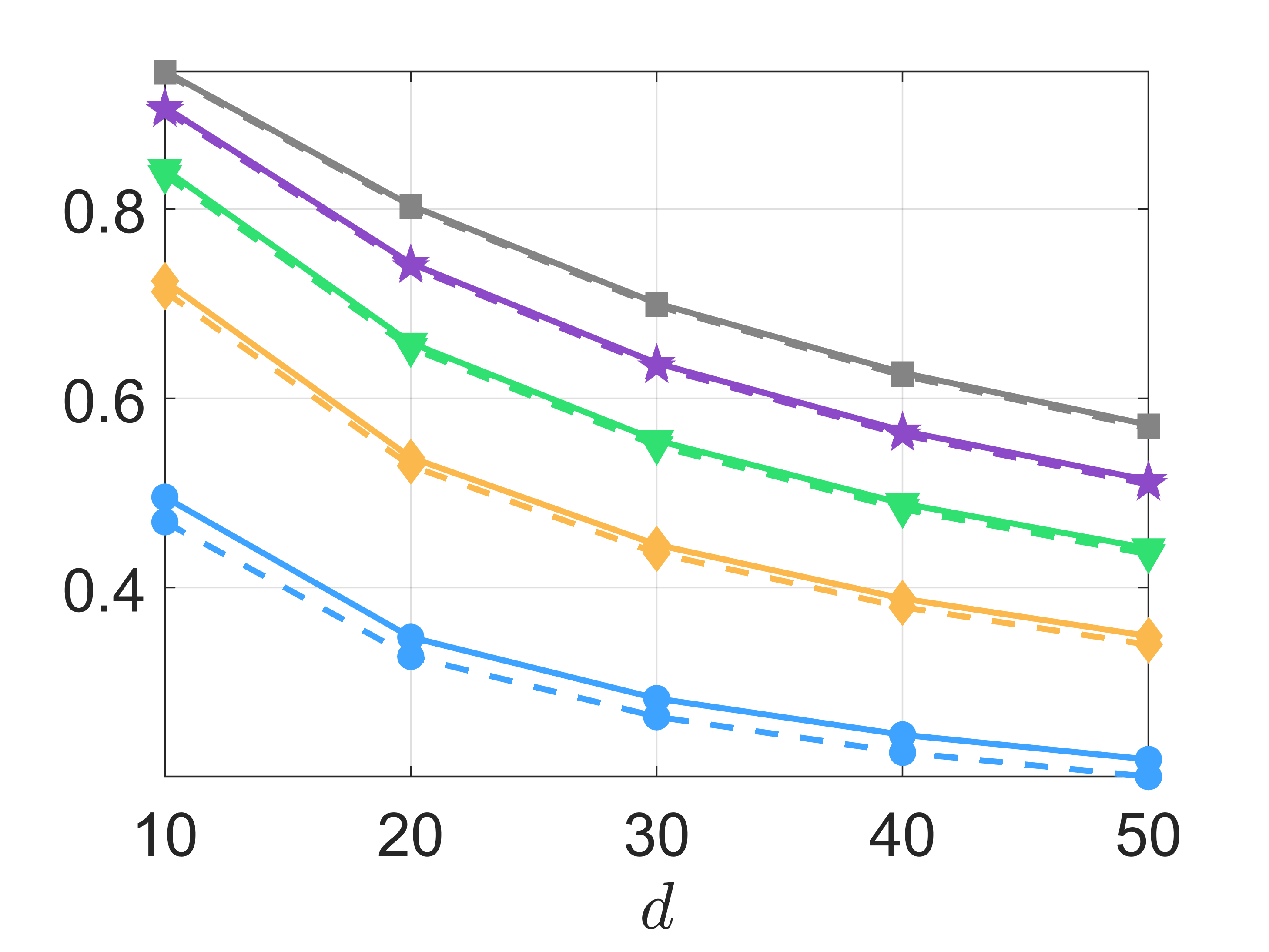}
\includegraphics[width=.49\linewidth]{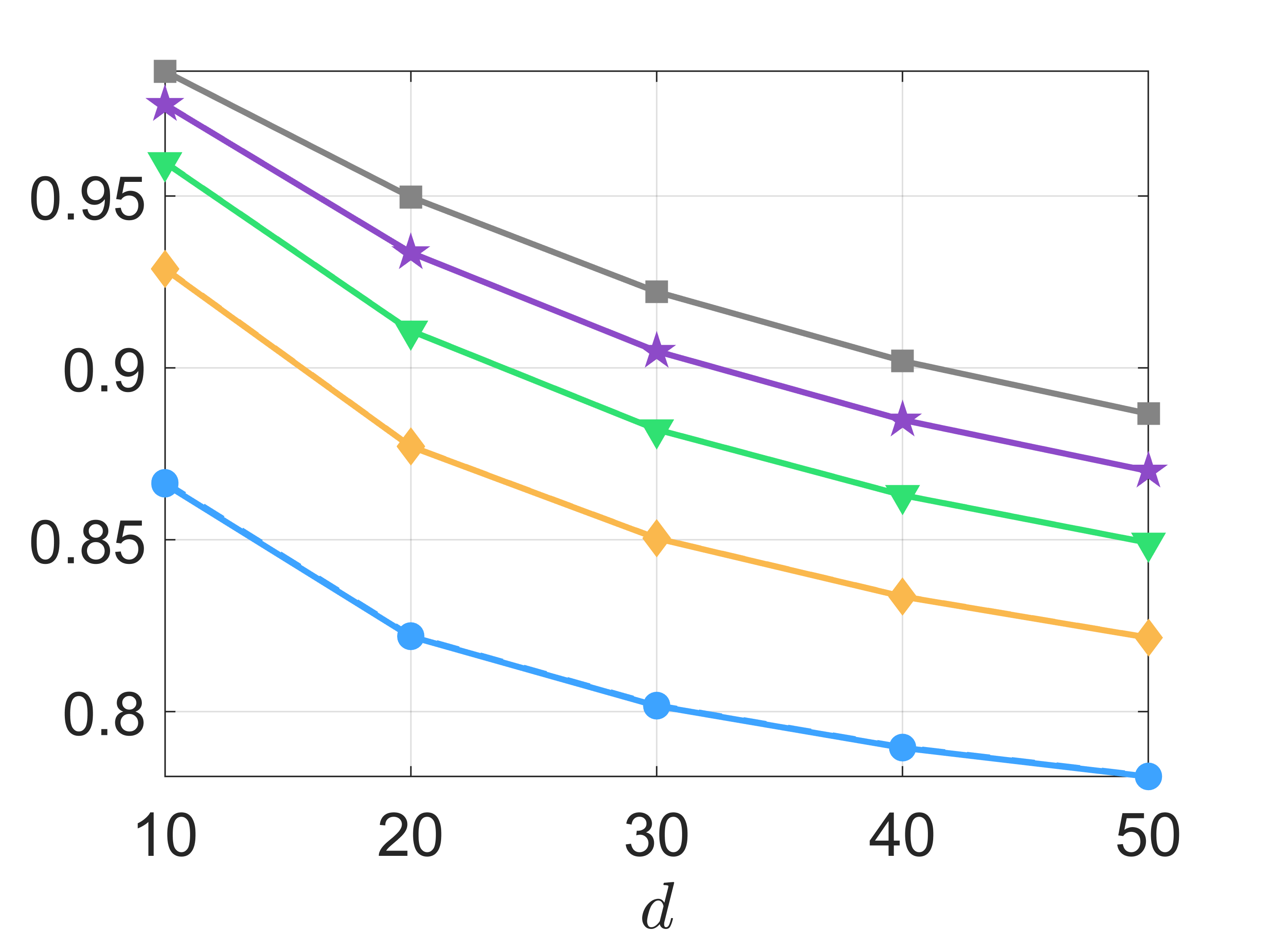} \\
\includegraphics[width=.49\linewidth]{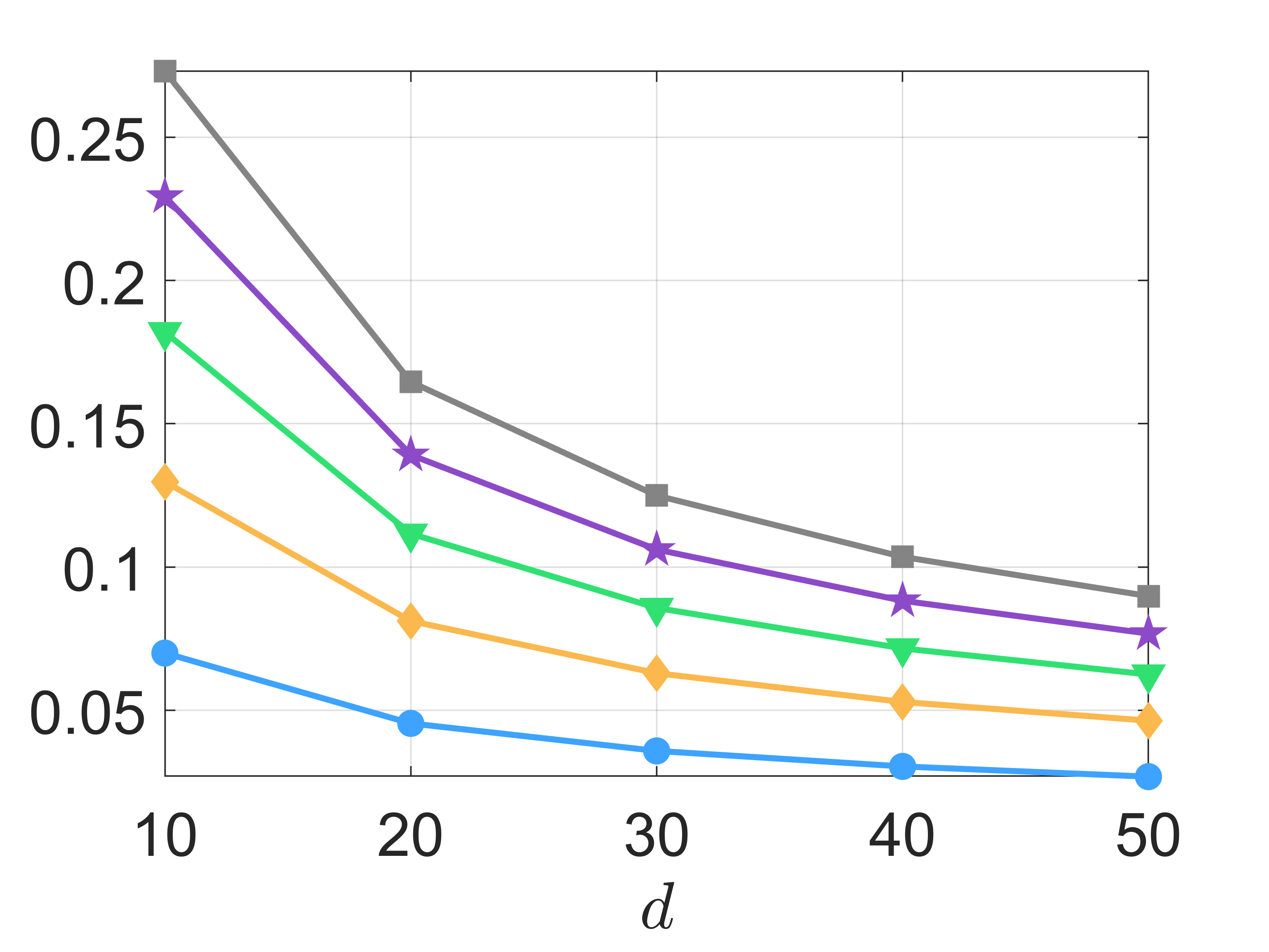}
\includegraphics[width=.49\linewidth]{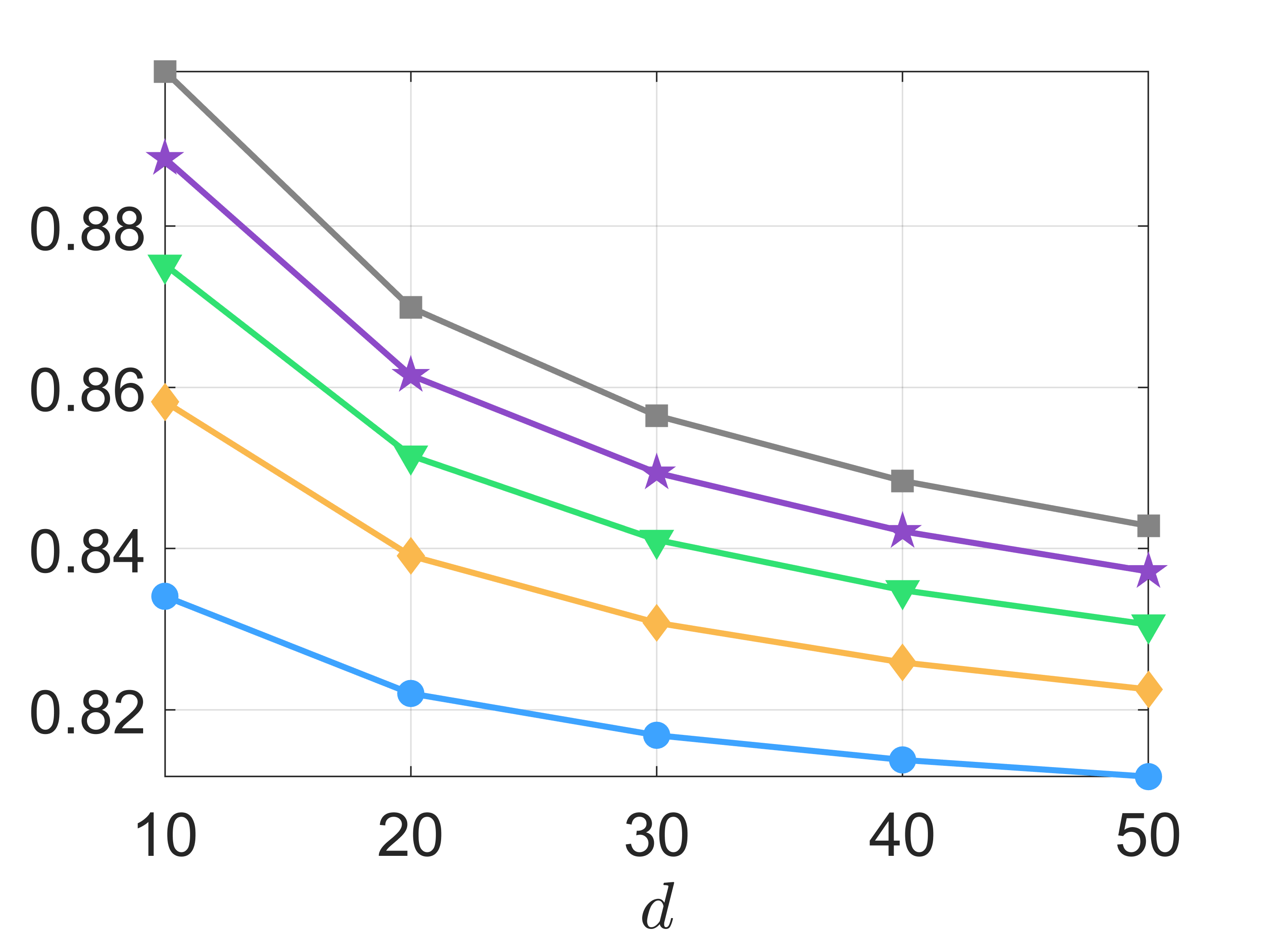}
\end{center}
\caption{\small Sphere. Value $a^*$ minimising $D_{\mu,s}(\PP^{[n]}_a)$ (left column) 
and ratio $D_{\mu,s}^{1/s}(\PP^{[n]}_{a^*})/D_{\mu,s}^{1/s}(\PP^{[n]}_1)$ (right column) as functions of $d$ for $s=2$ (top row) and $s=10$ (bottom row) and different values of $n$: $n=10$ ({\color{Cerulean} $\bullet$}), $n=10^2$ ({\color{Dandelion} $\blacklozenge$}), $n=10^3$ ({\color{green} $\blacktriangledown$}), $n=10^4$ ({\color{Orchid} $\bigstar$}), and $n=10^5$ ({\color{gray} {\tiny $\blacksquare$}}).
On the top-left panel, the dashed lines correspond to the extreme-value approximation \eqref{hat_a_2} of Corollary~\ref{Coro:astar-s2-Pa}.}
\label{F:sphere}
\end{figure}

\subsubsection{$\mu$ is uniform in the unit ball $\SB_d(1)$}\label{S:NSC-examples-ball}

For $\mu$ uniform in $\SB_d(1)$, the density of $r=\|U\|$ for $U\simd\mu$ is $\psi(\tau)=d\,\tau^{d-1}$. When $\PP=\PP_a$ uniform on $\SS_{d-1}(a)$, \eqref{Ls-mean-final} gives
\be\label{DS-Pa}
D_{\mu,s}(\PP_a^{[n]}) =  s \int_{t\geq 0} t^{s-1} \int_{r\geq 0}  H(r,t;\Phi_a,n) \,\psi(r)\,\dd r \,\dd t \,,
\ee
where $H(r,t;\Phi_a,n) =\bigg(1-\PP_a \left\{ \| X-\ub \| \leq t \right\} \bigg)^n$ with $\Phi_a(\cdot)$ the c.d.f.\ of the delta measure $\delta_a$ and $r=\|\ub\|$, so that $\PP_a\{ \| X-\ub \| \leq t \}$ is given by \eqref{PPaa}.
We also consider random quantisers whose $n$ points are i.i.d.\ with $\PP=\PP_{0,b}$ uniform in $\SB_d(b)$. The c.d.f.\ $\Phi(\cdot)$ of $R=\|X\|$ has then the density $\phi_{0,b}(\rho)=d\,\rho^{d-1}/b^d$ for $\rho\in[0,b]$.
Plots of $b^*$ minimising $D_{\mu,s}(\PP^{[n]}_{0,b})$ and of $D_{\mu,s}(\PP^{[n]}_{0,b^*})$ as functions of $d$ for different $n$ and $s=1,10$ are presented in the supplementary material.

For $n$ small enough, a distribution $\PP_{a^*}$ with $a^*=a^*(d,n,s)$, uniform on a single sphere, is optimal. However, Zador's theorem indicates that $\PP=\mu$ is asymptotically optimal. Consequently, distributions concentrated on a single sphere cease to be optimal for sufficiently large $n$. Below we show that this $n$ grows exponentially fast with $d$.

The algorithm described in Section~\ref{S:NSC-algorithm} produces optimal $\PP^*$ that are mixtures of uniform distributions on concentric spheres with radii $\rho_i$. The number of spheres increases (slowly) with $n$, and the weight $w_i$ allocated to each $\rho_i$ increases with $\rho_i$ (as expected, since $\PP^*$ should increasingly resemble $\mu$ as $n$ grows); see Figure~\ref{F:cdf-ball} for an illustration.

We first consider the case $s=2$ and $d=8$. The algorithm is stopped when
$\SE_{\mu,s}(\Phi_k;n) > 0.9999$, see \eqref{eq:radial-gap-bound}; we therefore identify the computed distribution with the optimum $\PP^*$ in the discussion below (for comparison, the expected-distortion efficiency for $\PP=\mu$ is only about 0.9296 for $n=1\,000$ and 0.9631 for $n=100\,000$).
Figure~\ref{F:sensitivity-ball} shows the sensitivity function $S_{\mu,s}(\rho;\Phi^*,n)$ as a function of $\rho\in[0,1]$ when $n=1\,000$ (two concentric spheres) and $n=100\,000$ (four spheres). The vertical dashed lines indicate the positions of the support points $\rho_i$ of $\Phi^*$ (for which $S_{\mu,s}(\rho_i;\Phi^*,n) = \overline{S}_{\mu,s}(\Phi^*,n)$; see Theorem~\ref{Th:radial-equivalence}). Figure~\ref{F:cdf-ball} presents the optimal c.d.f.\ $\Phi^*$ for both design sizes.

\begin{figure}[ht!]
\begin{center}
\includegraphics[width=.49\linewidth]{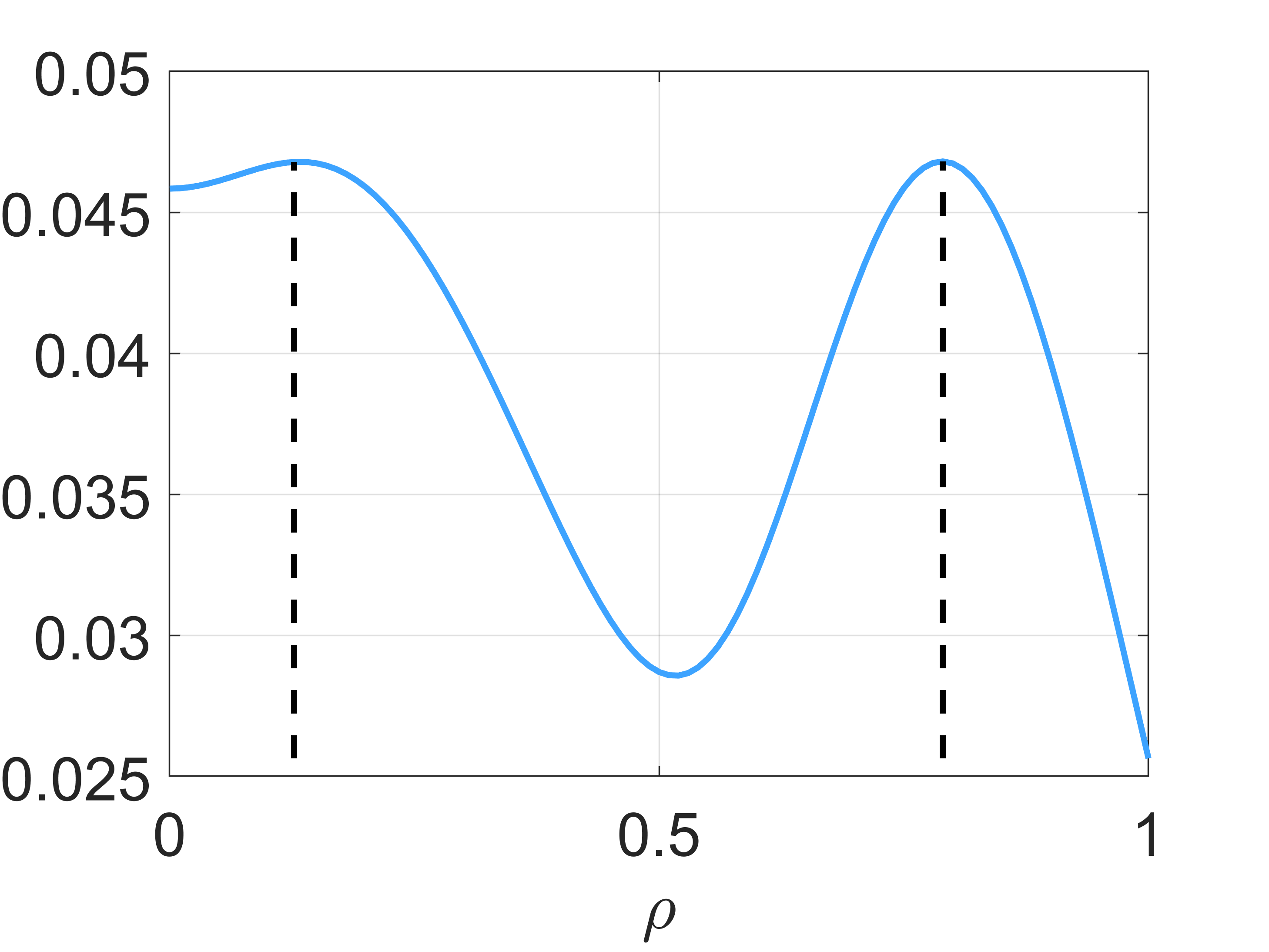}
\includegraphics[width=.49\linewidth]{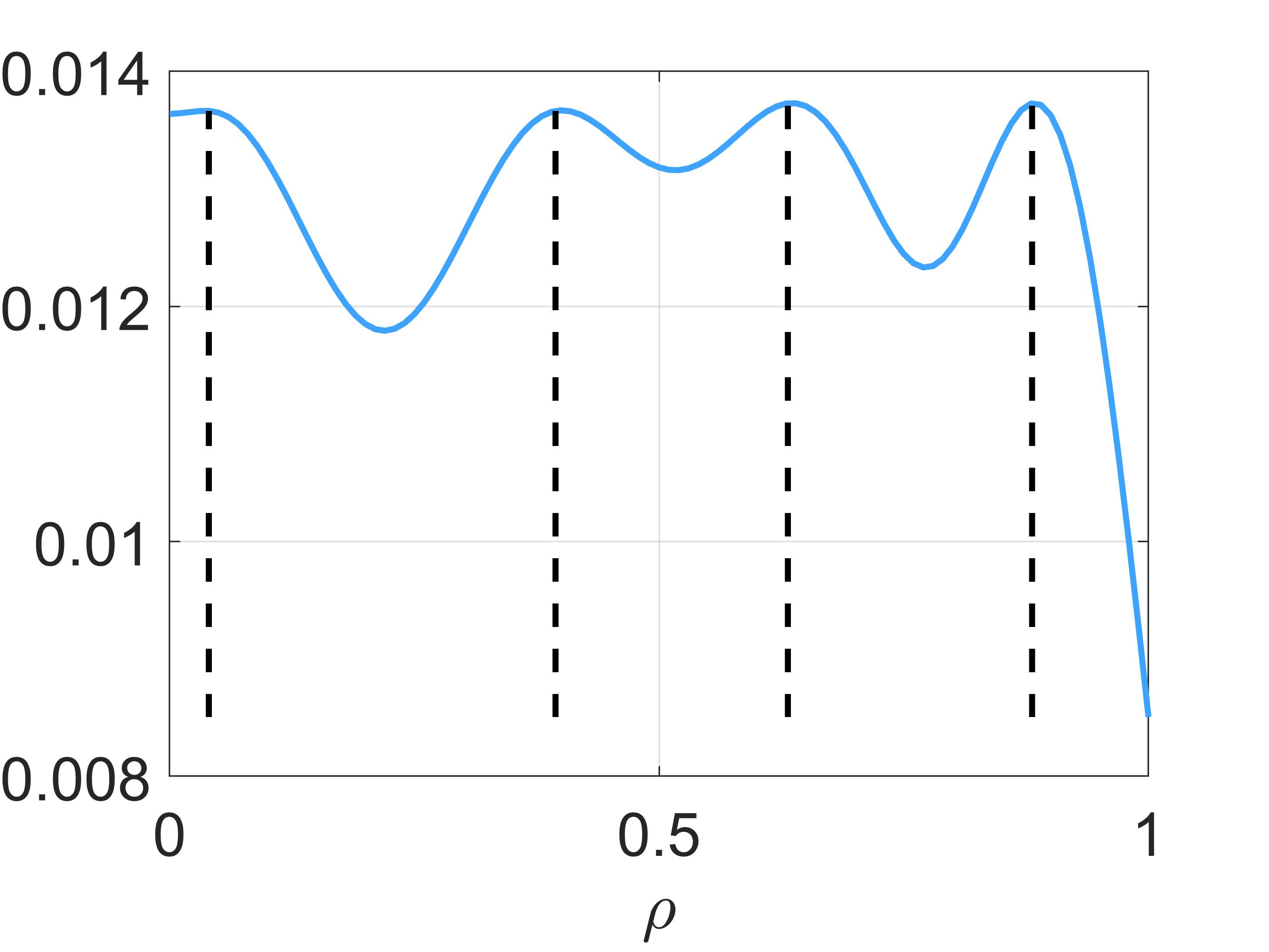}
\end{center}
\caption{\small Ball, $d=8$. Sensitivity function $S_{\mu,2}(\rho;\Phi^*,n)$ when $n=1\,000$ (left) and $n=100\,000$ (right).}
\label{F:sensitivity-ball}
\end{figure}

\begin{figure}[ht!]
\begin{center}
\includegraphics[width=.49\linewidth]{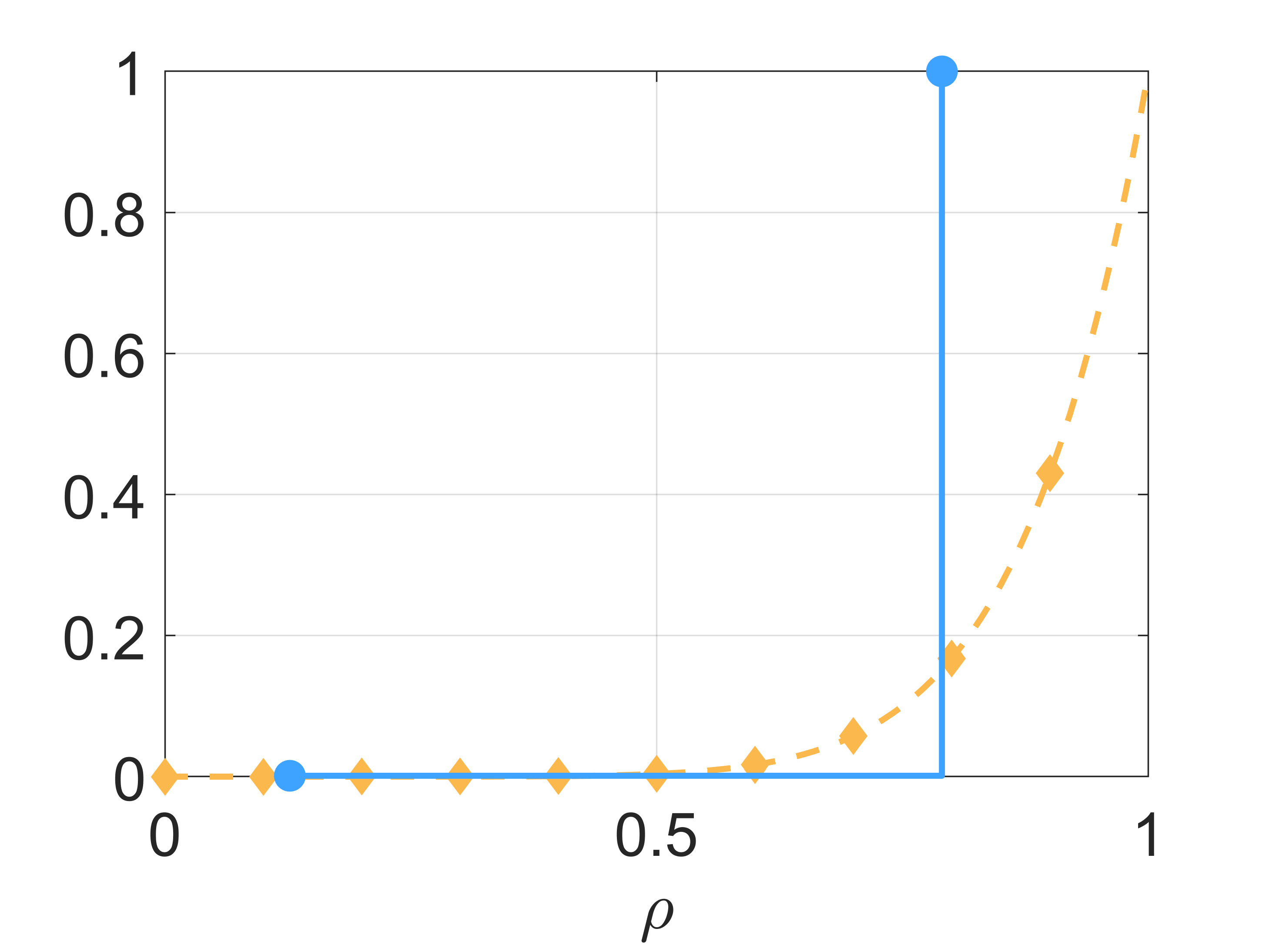}
\includegraphics[width=.49\linewidth]{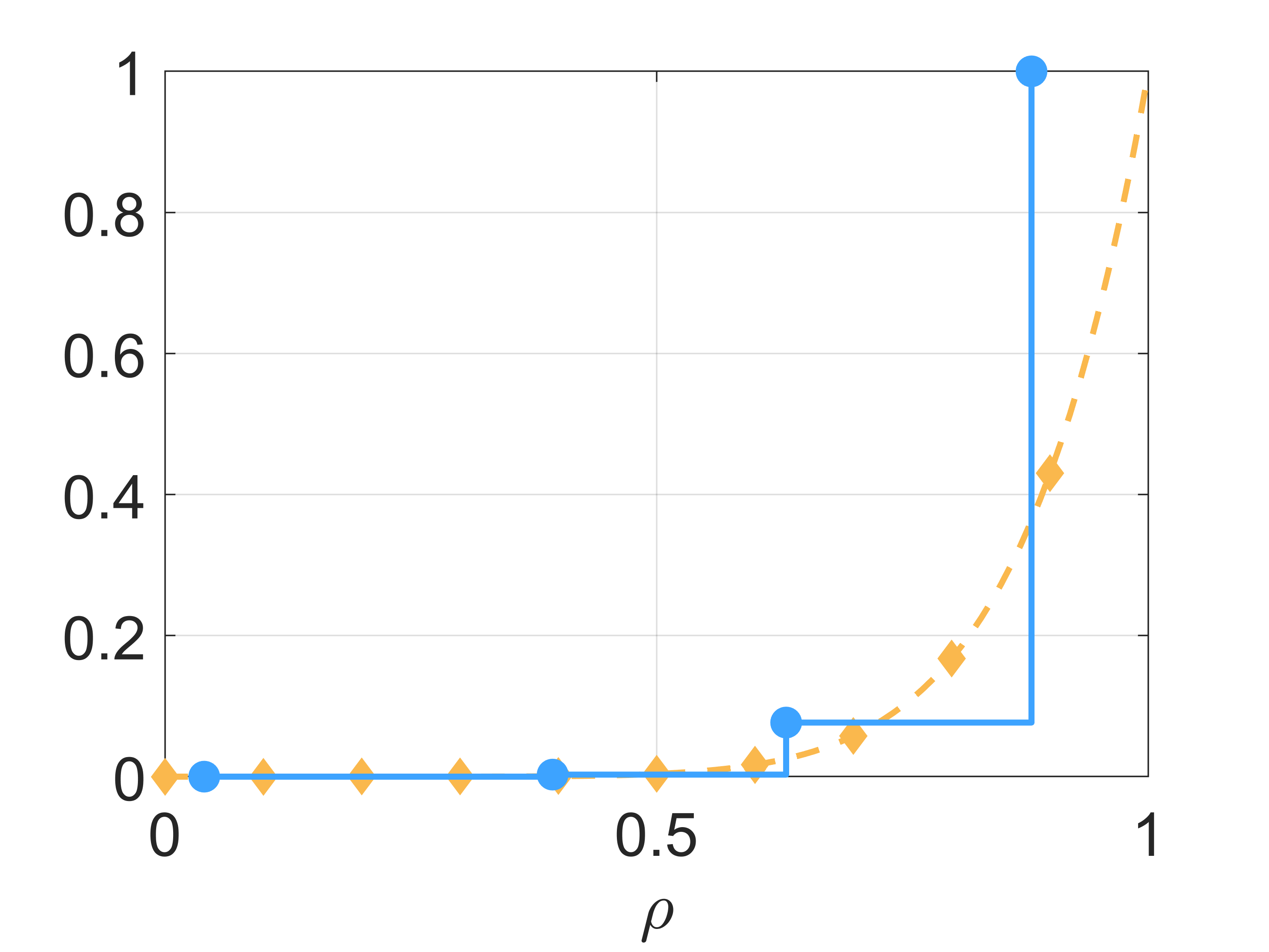}
\end{center}
\caption{\small Ball, $s=2$, $d=8$. $\Phi^*$ ({\color{Cerulean} {\tiny ---}$\bullet$}) and c.d.f.\ for $\PP=\mu$ ({\color{Dandelion} {\small - - -}$\blacklozenge$}) when $n=1\,000$ (left) and $n=100\,000$ (right). The locations of the support points of $\Phi^*$ are indicated by {\color{Cerulean} $\bullet$}.}
\label{F:cdf-ball}
\end{figure}

Below we address the following questions:
\begin{itemize}
    \item for which $n=n_1(d,s)$ does $\PP_{a^*}$, the uniform distribution on a sphere with optimal radius $a^*(d,n,s)$, cease being optimal?
    \item for which $n=n_2(d,s)$ does $\PP_{a^*}$ cease to perform better than the uniform distribution $\PP_{0,b^*}$ on a ball with optimised radius $b^*(d,n,s)$?
    \item for which $n=n_3(d,s)$ does $\PP_{a^*}$ cease to perform better than the uniform distribution $\PP_{0,1}=\mu$ (which is asymptotically optimal according to Zador's theorem)?
\end{itemize}

The left panel of Figure~\ref{F:asymptotics-ball_s2} presents the evolution of $n_1(d,s)$, $n_2(d,s)$ and $n_3(d,s)$ (on a logarithmic scale, obtained by binary search) as functions of $d$ for $s=2$ (see the supplementary material for plots of $n_2(d,s)$ for other values of $s$). The behaviour is qualitatively similar for other values of $s$, with a slightly slower increase with $d$ as $s$ grows. Notably, the plot of $n_1(d,2)$ almost perfectly coincides with the dashed line representing $2^{d+1}$. For moderate design sizes, the optimal distribution is uniform on a single sphere. This distribution remains preferable to a uniform distribution on a ball, and hence to $\mu$ itself, until $n$ reaches astronomically large values when $d$ is large.

\begin{figure}[ht!]
\begin{center}
\includegraphics[width=.49\linewidth]{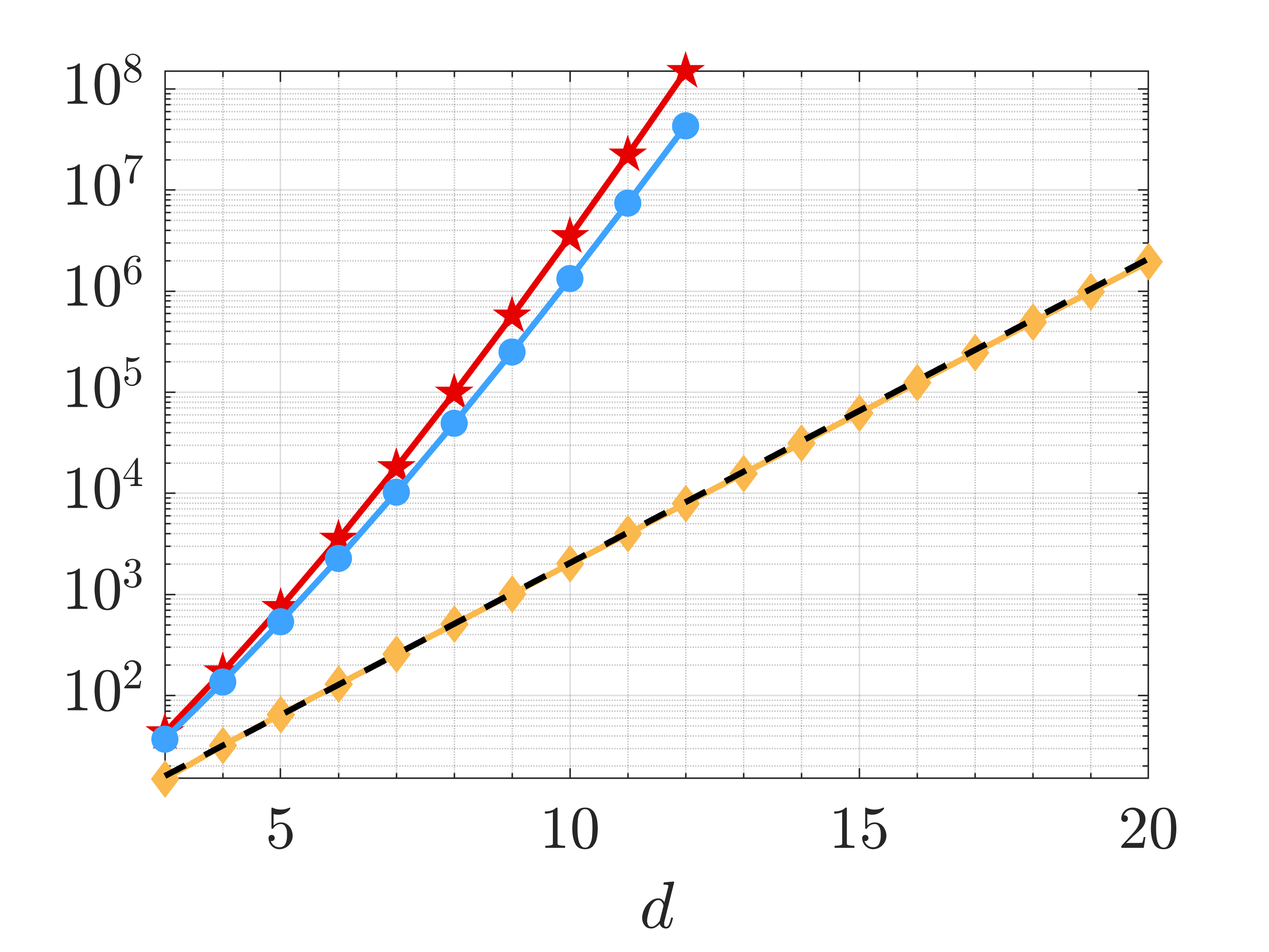}
\includegraphics[width=.49\linewidth]{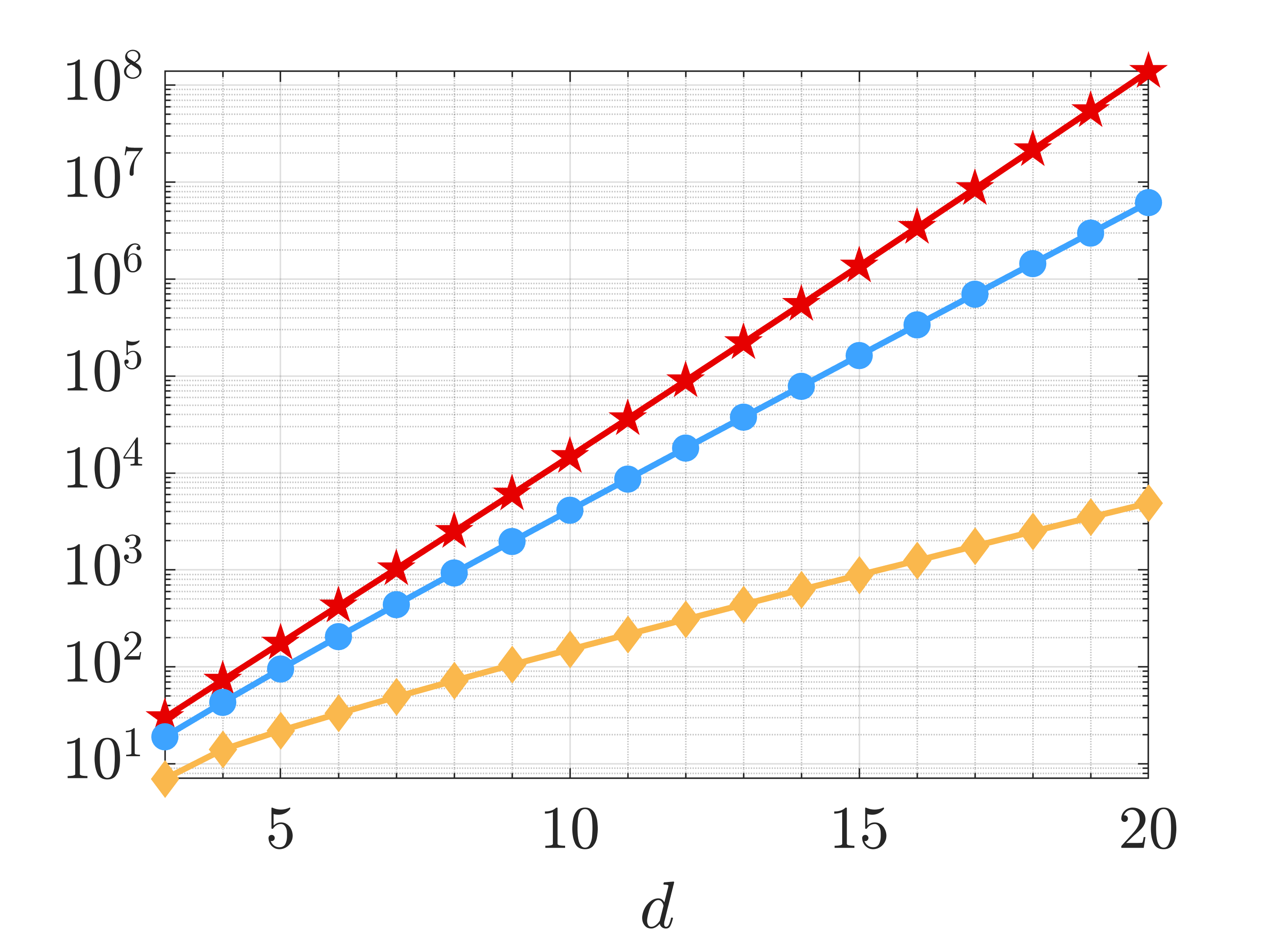}
\end{center}
\caption{\small $n_1(d,2)$ ({\color{Dandelion} $\blacklozenge$}), $n_2(d,2)$ ({\color{Cerulean} $\bullet$}) and $n_3(d,2)$ ({\color{red} $\bigstar$}) as functions of $d$, for $\mu$ uniform on $\SB_d(1)$ (left) and $\mu$ the normal distribution $\SN(\0b_d,\Ib_d/d)$ (right). }
\label{F:asymptotics-ball_s2}
\end{figure}

\subsubsection{$\mu$ is a spherically symmetric normal distribution}\label{S:NSC-examples-normal}

Any spherically symmetric normal distribution in $\RR^d$ can be renormalised as $\SN(\0b_d,\Ib_d/d)$. In this section, we assume that $\mu$ is the corresponding probability measure. When $U\simd\mu$, $d\,\|U\|^2$ has the chi-square distribution with $d$ degrees of freedom, so that the density of $\|U\|$ is
\be\label{psi-normal}
\psi(r)=\frac{d^{d/2}}{2^{d/2-1}\,\Gamma(d/2)}\,r^{d-1}\,\e1^{-d r^2/2} \,, \ r \geq 0\,.
\ee
The distribution of $r$ is more and more concentrated around 1 as $d$ increases. Its $k$-th moment is
\bea
M_{\psi,k} = \Ex_\mu\{\|U\|^k\}=(2^{k/2}/d^{k/2}) \Gamma((k+d)/2)/\Gamma(d/2)\,, \ k=1,2,\ldots
\eea
so that its mean is smaller than 1 for any $d$ and its variance is $1/(2\,d)-1/(8\, d^2)+ \SO(1/d^3)$, $d\to\infty$.

When $\PP=\PP_a$ uniform on $\SS_{d-1}(a)$, the expected $(\mu,s)$-distortion $D_{\mu,s}(\PP_a^{[n]})$ is given by \eqref{DS-Pa},
with $\psi(\cdot)$ given by \eqref{psi-normal}.
When $\PP=\mu_\ms$ is the probability measure of $\SN(\0b_d,\ms^2\Ib_d/d)$, $R=\|X\|$ has the p.d.f.\ $\phi_\ms(\rho)=(1/\ms)\psi(\rho/\ms)$, where $\psi(\cdot)$ is given by \eqref{psi-normal}, and the expected $(\mu,s)$-distortion $D_{\mu,s}(\mu^{[n]}_\ms)$ is given by \eqref{Ls-mean-final}
where
\bea
H(r,t;\Phi,n) =\bigg(1-\frac{1}{\ms}\,\int_{\rho\geq 0} I_\upsilon(\delta,\delta) \,\psi(\rho/\ms)\,\dd\rho \bigg)^n \,,
\eea
with $\upsilon=\upsilon(t,\rho,r)$ given by \eqref{v}.

We denote by $\ms^* = \ms^*(d,n,s)$ the value of $\ms$ that minimises the distortion $D_{\mu,s}(\mu^{[n]}_\ms)$ and by $a^* = a^*(d,n,s)$ the value of $a$ that minimises $D_{\mu,s}(\PP_a^{[n]})$, as in previous sections.
The empirical distribution of an optimal quantiser $\Xb_n^*$ minimising $E_{\mu,s}(\Xb_n)$ weakly converges to $\SN(\0b_d,(\ms_\infty^*)^2\,\Ib_d/d)$ as $n\to\infty$, with $\ms_\infty^*=\sqrt{1+s/d}$; see \cite[Table~7.1]{GrafL2000}.

We take $d=10$ and $s=5$. The left panel of Figure~\ref{F:sensitivity-normal} presents the sensitivity function $S_{\mu,s}(\rho;\Phi^*,n)$ as a function of $\rho\in[0,2]$ when $n=1\,000$, showing that the optimal distribution is a mixture of two uniform distributions on concentric spheres; their radii are indicated by vertical dashed lines. The right panel shows the c.d.f.\ obtained when $\SE_{\mu,s}(\Phi_k;n) > 0.9999$ in the algorithm of Section~\ref{S:NSC-algorithm}; see \eqref{eq:radial-gap-bound}.
In the same situation, the expected-distortion efficiency for $\PP=\mu_{\ms_\infty^*}$ is only about 0.7828.

\begin{figure}[ht!]
\begin{center}
\includegraphics[width=.49\linewidth]{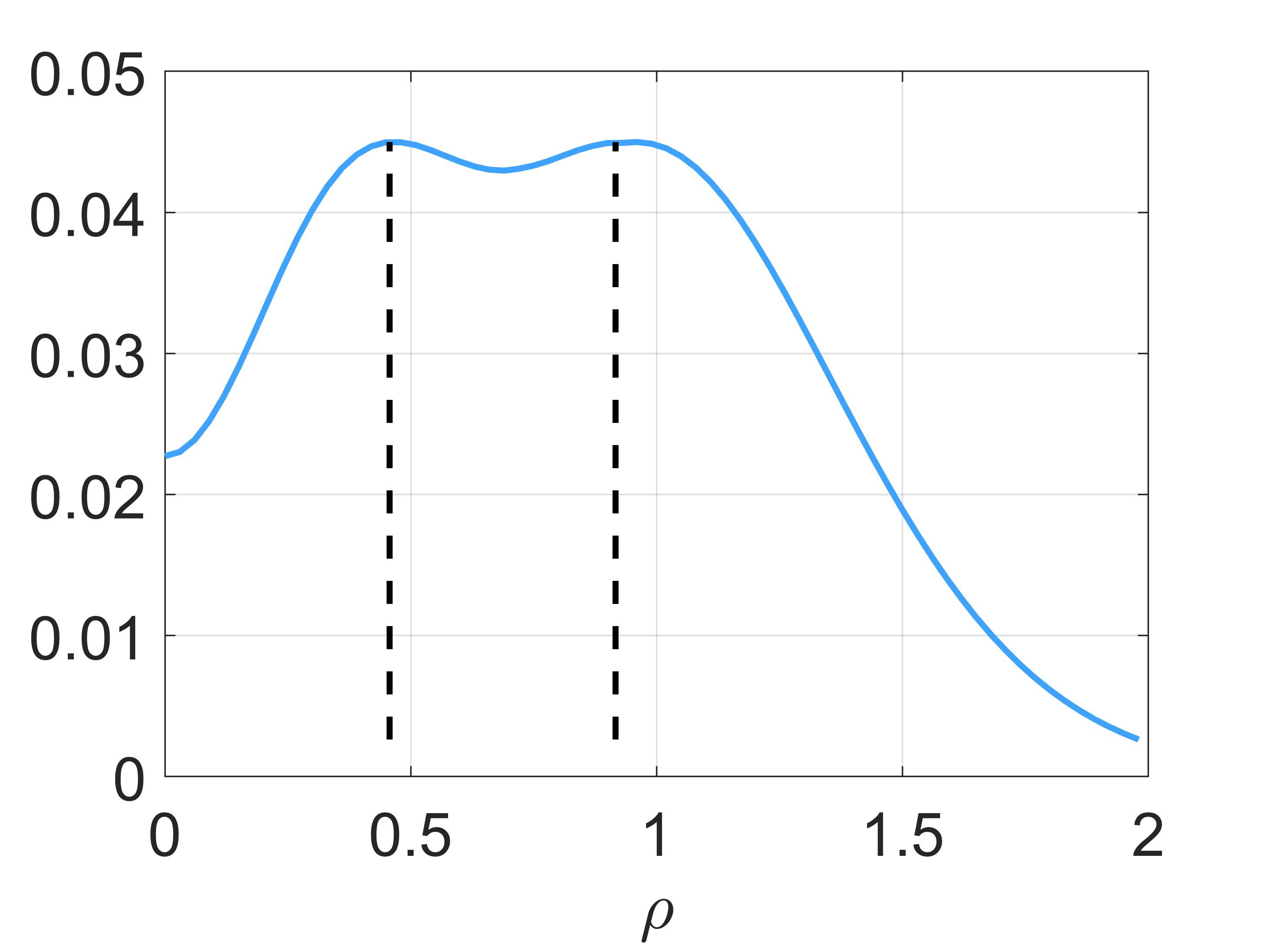}
\includegraphics[width=.49\linewidth]{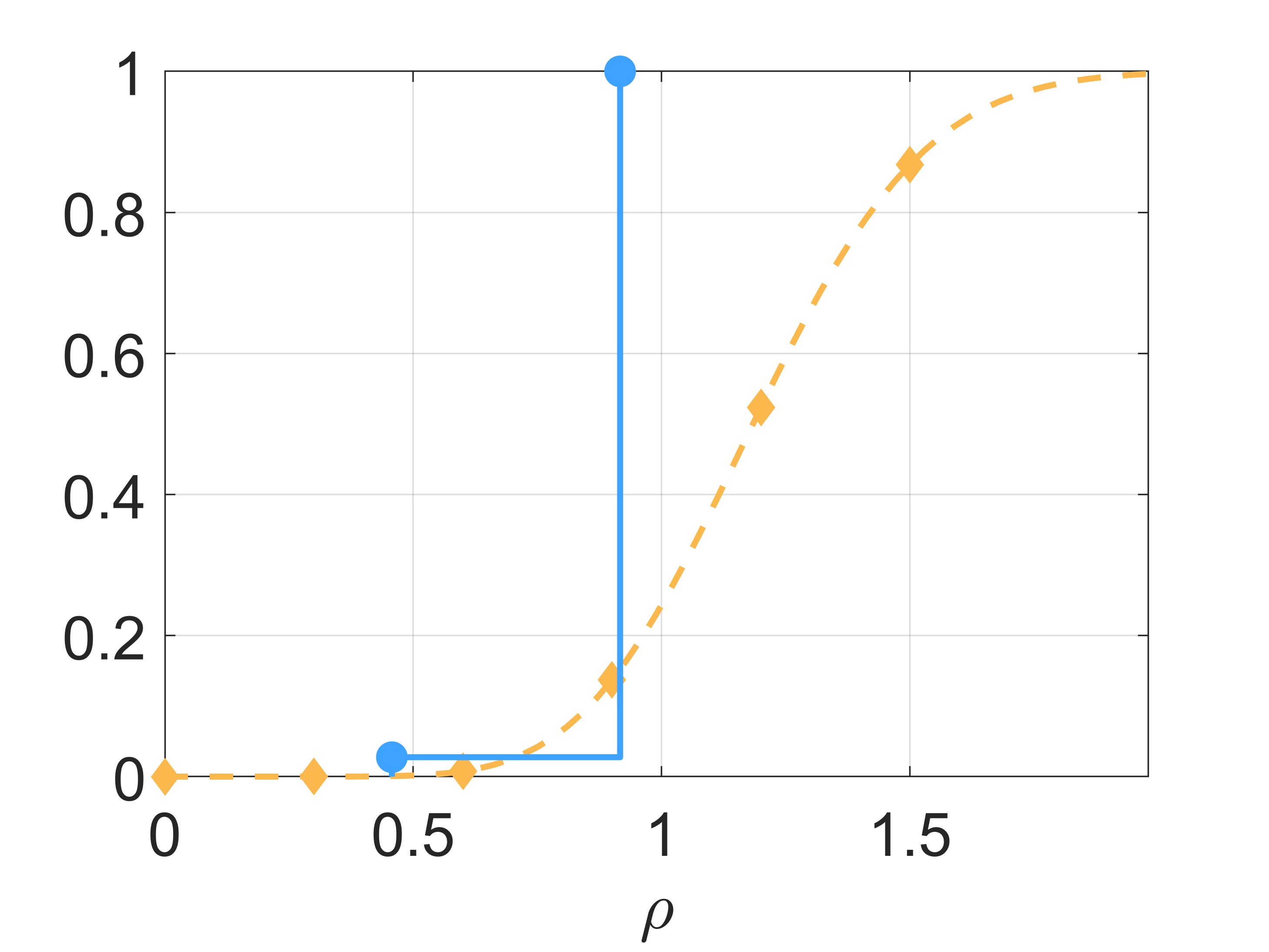}
\end{center}
\caption{\small Normal distribution, $s=5$, $d=10$, $n=1\,000$. Left: Sensitivity function $S_{\mu,5}(\rho;\Phi^*,n)$. Right: $\Phi^*$ ({\color{Cerulean} {\tiny ---}$\bullet$}) and c.d.f.\ for $\PP=\mu_{\ms_\infty^*}$ ({\color{Dandelion} {\small - - -}$\blacklozenge$}); the locations of the support points of $\Phi^*$ are indicated by {\color{Cerulean} $\bullet$}.}
\label{F:sensitivity-normal}
\end{figure}

The left panel of Figure~\ref{F:normal} illustrates that $\ms^*=\ms^*(d,n,s)$ increases with $s$, similarly to $\ms_\infty^*$ (shown as a dashed line), and that $a^*=a^*(d,n,s)$ also increases with $s$, unlike in Section~\ref{S:NSC-examples-sphere}. The reason is that large values of $r=\|U\|$ receive increasing weight as $s$ increases, while $r=1$ in Section~\ref{S:NSC-examples-sphere}. The right panel shows that $\ms^*-a^*$ tends to zero as $d$ increases, a consequence of $\mu_\ms$ tending to $\PP_\ms$ as $d\to\infty$.

\begin{figure}[ht!]
\begin{center}
\includegraphics[width=.49\linewidth]{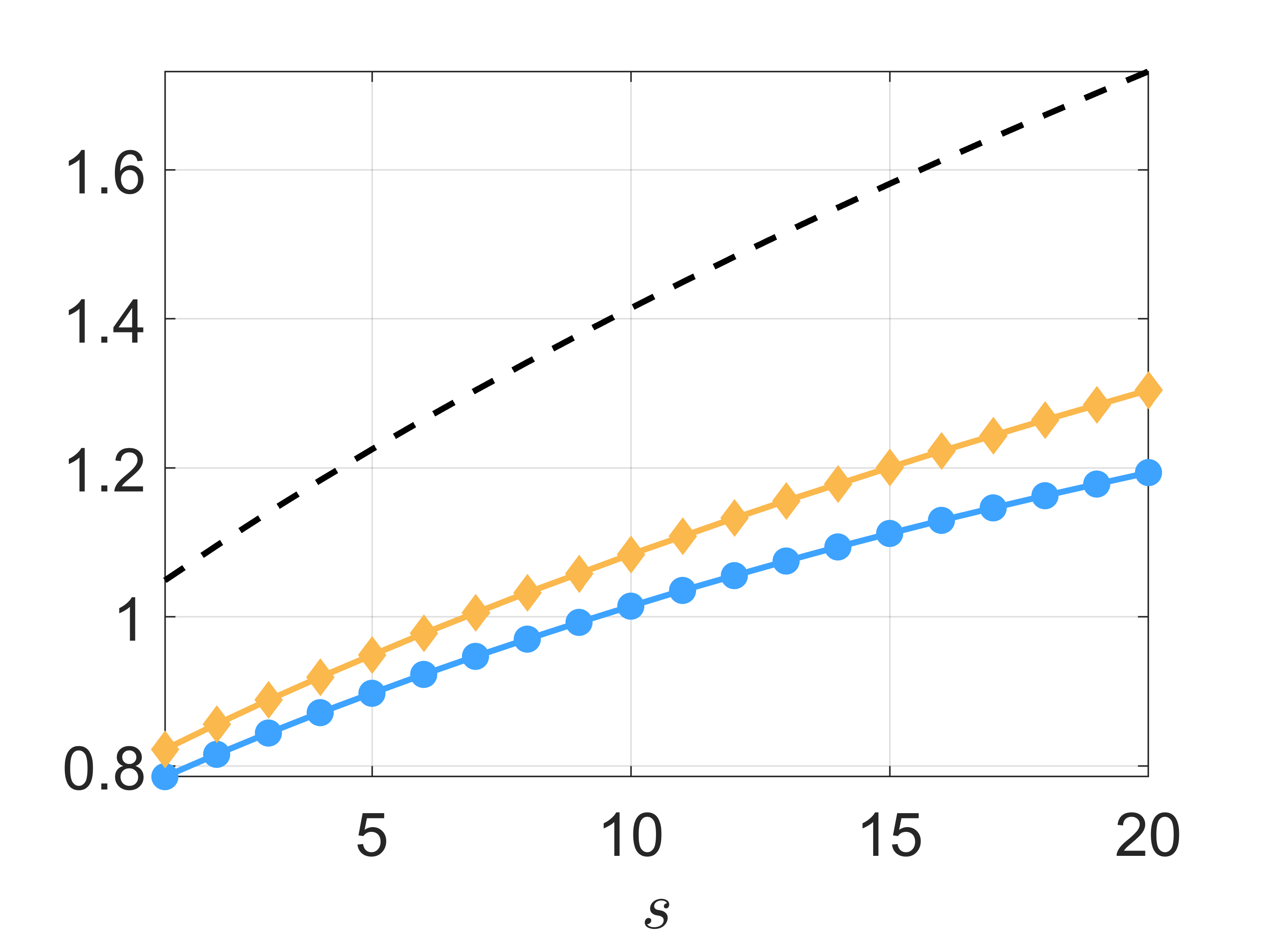}
\includegraphics[width=.49\linewidth]{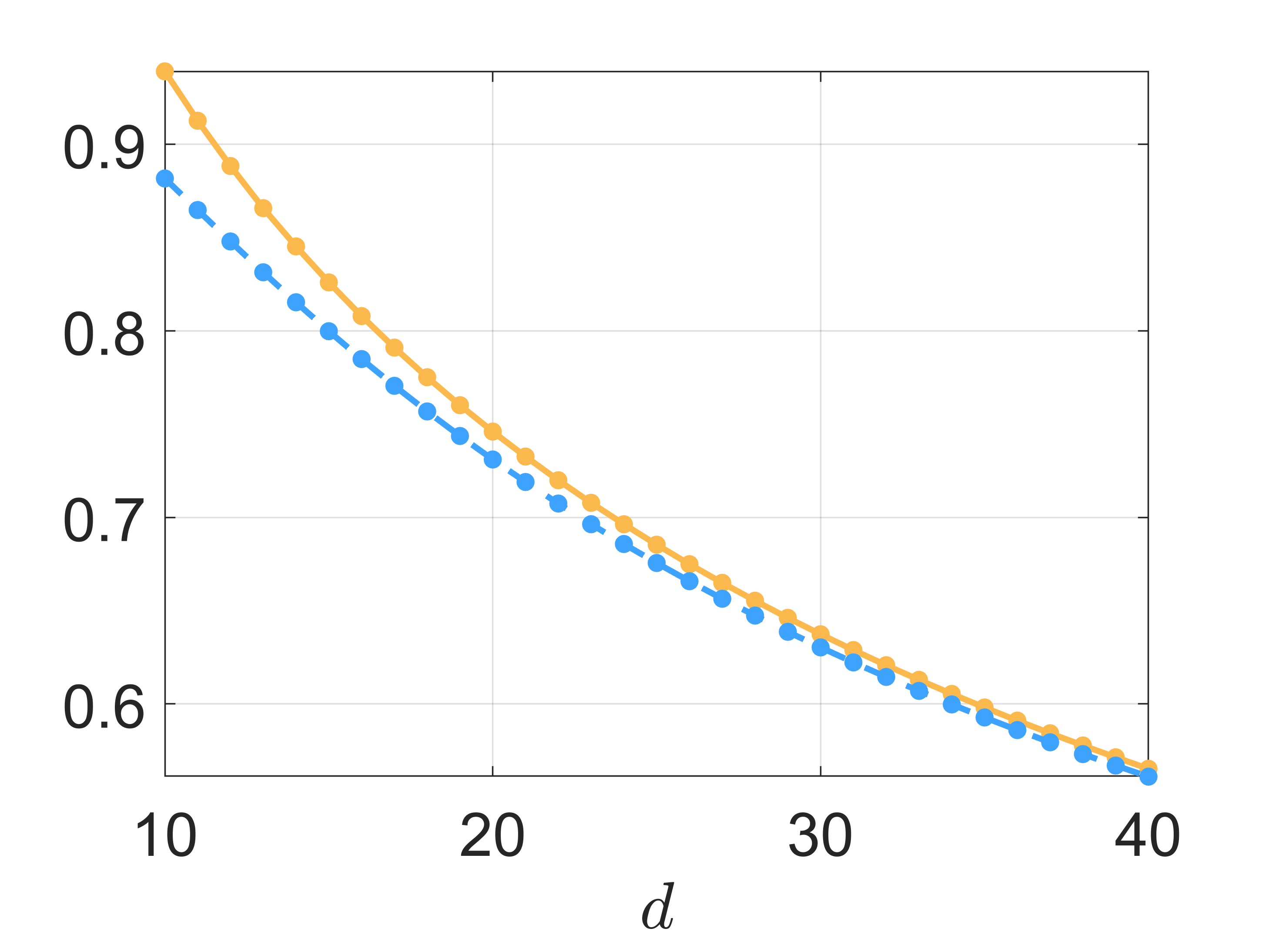}
\end{center}
\caption{\small Optimal values $a^*(d,n,s)$ ({\color{Cerulean} $\bullet$}) and $\ms^*(d,n,s)$ ({\color{Dandelion} $\blacklozenge$}) as functions of $s$ for $d=10$ and $n=1\,000$ (left panel, with $\ms_\infty^*$ shown as a dashed line), and as functions of $d$ for $s=2$ and $n=10\,000$ (right panel).}
\label{F:normal}
\end{figure}

Finally, similarly to Section~\ref{S:NSC-examples-ball}, we compute the values of $n_1(d,s)$, $n_2(d,s)$ and $n_3(d,s)$ at which the following transitions occur, respectively:
$\PP_{a^*}$, the uniform distribution on a sphere with optimal radius $a^*(d,n,s)$, ceases to be optimal; $\PP_{a^*}$ ceases to outperform an optimised normal distribution $\mu_{\ms^*}$; $\PP_{a^*}$ ceases to outperform the asymptotically optimal distribution $\mu_{\ms_\infty^*}$. Their values for $s=2$ are plotted as functions of $d$ on the right panel of Figure~\ref{F:asymptotics-ball_s2} (see the supplementary material for other values of $s$). While the rate of increase with $d$ is smaller than for $\mu$ uniform in $\SB_d(1)$ (left panel), the qualitative behaviour remains similar.

\section{Extreme-value approximations}\label{S:asymptotic}

In this section, extreme-value theory is used to derive approximations of the $(\mu,s)$-distortion $D_{\mu,s}(\PP_a^{[n]})$, where $\PP_a$ is uniform on $\SS_{d-1}(a)$ (Section~\ref{S:Rn-uniform-on-SS}).
In Section~\ref{S:n-growing-with-d}, three asymptotic regimes are identified that govern the limiting behaviour of $D_{\mu,s}(\PP_a^{[n]})$ when $n$ increases with $d$.
This is illustrated in Section~\ref{S:extreme-value-sphere} for the case where $\mu$ is uniform on $\SS_{d-1}(1)$. Further examples, including the case where $\mu$ is spherically normal, are provided in the supplementary material.

\subsection{Uniform quantisers on a sphere}\label{S:Rn-uniform-on-SS}

Consider random quantisers $\Rb_n$ with distribution $\Px_n=\PP_a^{[n]}$ where $\PP_a$ is uniform on $\SS_{d-1}(a)$. For large $n$, the distribution of $\zeta(n,d)$ in \eqref{d2URn} can be approximated using standard results from extreme-value theory. Here we assume $d\geq 3$ to avoid singularity in beta distributions.

\begin{lemma}\label{L:main-extreme-value2}
Assume that $d\geq 3$ and $\Px_n=\PP_a^{[n]}$, with $a>0$. For any fixed $\ub\in\RR^d\setminus\{\0b_d\}$, we have
\be\label{d2URnb}
\frac{1}{4\,a\,\|\ub\|\,\kappa_{n,d}} \left[d^2(\ub,\Rb_n)-(\|\ub\|-a)^2\right] \rad \xi_d \,, \ n\to\infty\,,
\ee
where $\xi_d$ is a random variable with Weibull c.d.f.\ $F_d(t)=1-\exp(-t^\delta)$, $t\geq 0$, and moments $\Ex\{\xi_d^k\}=\Gamma(1+k/\delta)$ with $\delta=(d-1)/2$. Here $\kappa_{n,d}$ is the $1/n$ quantile of $\beta_{\delta,\delta}$, defined by $I_{\kappa_{n,d}}(\delta,\delta)=1/n$.
\end{lemma}

\begin{proof}
From standard extreme-value theory (see, e.g., \cite[p.~59]{KotzN2000}), the random variable
$\zeta(n,d)/\kappa_{n,d}$ in \eqref{d2URn} converges in distribution to the random variable $\xi_d$.
\end{proof}

\begin{remark}\label{R:moments-duRn}
Lemma~\ref{L:main-extreme-value2} yields the following approximations for the $s$-moment of $d(\ub,\Rb_n)$ for large $n$ and any given $\ub$:
\bea
M_s(d,n,a,r) = \int_{\zeta\geq 0} \left[(r-a)^2 + 4\,a\,r\,\kappa_{n,d}\,\zeta\right]^{s/2}\, \dd F_d(\zeta) \,,
\eea
with $r=\|\ub\|$.
Explicit expressions are easily obtained for even $s$. Assuming that $\|\ub\|=1$, this gives in particular the approximations
\bea
\var_2 &=& 16\,a^2\,\kappa_{n,d}^2\,\var\{\xi_d\} = 16\,a^2\,\kappa_{n,d}^2 \left[\Gamma(1+2/\delta)-\Gamma^2(1+1/\delta)\right] \,, \\
\var_4 &=& 64\,a^2\,\kappa_{n,d}^2\, \left\{(1-a)^4\,\var\{\xi_d\} + 4\,a^2\,\kappa_{n,d}^2\,\var\{\xi_d^2\} \right. \\
&& \left. + 4\,(1-a)^2\,a\,\kappa_{n,d}\,\left[\Gamma(1+3/\delta)-\Gamma(1+1/\delta)\Gamma(1+2/\delta)\right] \right\} \,.
\eea
for $\var_s=\var\{d^s(\ub,\Rb_n)\}$ when $s=2$ and $s=4$, respectively.
\fin
\end{remark}

For each fixed $\ub$, equation~\eqref{d2URnb} gives the limiting distribution of the suitably normalised nearest-neighbour distance. Averaging this pointwise extreme-value approximation over a spherically symmetric target $\mu$ motivates the following approximation of $D_{\mu,s}(\PP_a^{[n]})$.

\begin{definition}\label{D:main-extreme-value}
Assume that $d\geq 3$ and that $\mu$ is spherically symmetric. The extreme-value approximation of the expected $(\mu,s)$-distortion $D_{\mu,s}(\PP_a^{[n]})$ is defined by
\be\label{Ds-approx}
\widehat E_{\mu,s}^s(n;a) = \int_{r\geq 0} \int_{\zeta\geq 0} \left[(r-a)^2 + 4\,a\,r\,\kappa_{n,d}\,\zeta\right]^{s/2}\, \dd F_d(\zeta) \, \dd \Psi(r) \,,
\ee
where $\Psi(\cdot)$ is the c.d.f.\ of $\|U\|$ with $U\simd \mu$, $F_d(t)=1-\exp(-t^\delta)$ with $\delta=(d-1)/2$, and $\kappa_{n,d}$ is the $1/n$ quantile of $\beta_{\delta,\delta}$ defined by $I_{\kappa_{n,d}}(\delta,\delta)=1/n$. The extreme-value approximation $\widehat D_{\mu,s}(\PP_a^{[n]})$ of $D_{\mu,s}(\PP_a^{[n]})$ is defined by $\widehat D_{\mu,s}(\PP_a^{[n]})=\widehat E_{\mu,s}^s(n;a)$.
\end{definition}

The approximation \eqref{Ds-approx} is derived from the following considerations. Since $\mu$ is spherically symmetric, we can write $U\simd r\cdot U^{(d)}$ where $U^{(d)}$ is uniformly distributed on $\SS_{d-1}(1)$, $r$ has the c.d.f.\ $\Psi(\cdot)$ and the two variables are independent. Replacing the conditional distribution of $d^2(U,\Rb_n)$, given $r=\|U\|$, by the extreme-value limit in \eqref{d2URnb}, and then averaging with respect to $r$ and the Weibull variable, yields \eqref{Ds-approx}. This is an approximation, not an exact identity of means. A formal proof that the variability of $E_{\mu,s}^s(\Rb_n)$ vanishes asymptotically is provided below in Theorem~\ref{Th:distortion-concentration}. We first introduce some notation.

Let $Z_1,\ldots,Z_n$ be i.i.d.\ uniform on $\SS_{d-1}(1)$ and let $\Rb_{n,a}$ be a random quantiser on $\SS_{d-1}(a)$:
\bea
\Rb_{n,a}=\{aZ_1,\ldots,aZ_n\}\simd\PP_a^{[n]} \,.
\eea
Denote $Q_{n,a}=E_{\mu,s}^s(\Rb_{n,a})$ the realised distortion of this random quantiser, so that $\Ex\{Q_{n,a}\}=D_{\mu,s}(\PP_a^{[n]})$. For $U\simd\mu$, define
\bea
B_{\mu,s}(a)=\Ex_\mu\{(\|U\|+a)^s\}\,.
\eea
For $n\geq2$, let $V_{i,n}$ be the spherical Voronoi cell of $Z_i$ and let
$p_{i,n}$ be its normalised spherical measure (ties may be resolved arbitrarily, since they form a set of spherical measure zero almost surely). Denote by $v_d(\theta)$ the normalised surface measure of a spherical cap of angular radius $\theta$, and by $N_d(\theta)$ the smallest number of such caps required to cover $\SS_{d-1}(1)$.

\begin{theorem}\label{Th:distortion-concentration}
For any $d\geq2$, $n\geq2$, $s>0$, $a>0$, and $\mu$ such that $B_{\mu,s}(a)<\infty$, we have
\be\label{eq:Voronoi-variance1}
\var\{Q_{n,a}\} \leq B_{\mu,s}^2(a)\,\Ex\left\{\sum_{i=1}^n p_{i,n}^2\right\}
\leq B_{\mu,s}^2(a)\,\Ex\left\{\max_{1\leq i\leq n}p_{i,n}\right\} \,,
\ee
and
\be
\var\{Q_{n,a}\} \leq B_{\mu,s}^2(a)\,B(d,n)\,,
\label{eq:Voronoi-variance2}
\ee
where 
\bea
B(d,n) = \inf_{\theta\in(0,\pi)}\left\{
v_d(\theta)+N_d(\theta/2)\exp[-n v_d(\theta/2)]
\right\}\,.
\eea
Therefore,
\bea
\Pr\{|Q_{n,a}-D_{\mu,s}(\PP_a^{[n]})|\geq t\} \leq \frac{B_{\mu,s}^2(a)\,B(d,n)}{t^2}\,,
\eea
and for every fixed $d$,
\be\label{eq:fixed-d-concentration}
Q_{n,a}-D_{\mu,s}(\PP_a^{[n]}) \to 0
\quad\hbox{in $L_2$ and in probability as }n\to\infty.
\ee
\end{theorem}

\begin{proof}
Let $Q_{n,a}^{(-i)}$ denote the distortion obtained after deleting the point $aZ_i$.  This random variable is independent of $Z_i$ conditionally on the other $n-1$ points.  Deleting $aZ_i$ changes the nearest-neighbour distance only for directions in $V_{i,n}$.  Since the distance between a point of norm $r$ and any point of norm $a$ is at most $r+a$, spherical symmetry and independence of the radial and angular components give
\bea
0\leq Q_{n,a}^{(-i)}-Q_{n,a}\leq B_{\mu,s}(a)p_{i,n}.
\eea
The deletion form of the Efron--Stein inequality, see \cite{BoucheronLB2004}, therefore yields
\bea
\var\{Q_{n,a}\}\leq
\sum_{i=1}^n\Ex\{(Q_{n,a}^{(-i)}-Q_{n,a})^2\}
\!\!\! &\leq& \!\!\!B_{\mu,s}^2(a)\, \Ex\left\{\sum_{i=1}^np_{i,n}^2\right\} \\
&\leq& \!\!\! B_{\mu,s}^2(a)\,\Ex\left\{\max_{1\leq i\leq n}p_{i,n}\right\} \,,
\eea
where the last inequality follows from $\sum_i p_{i,n}=1$. This completes the proof of \eqref{eq:Voronoi-variance1}.

Let
$\mo_n=\max_{z\in\SS_{d-1}(1)}\min_i\arccos(z\TT Z_i)$ be the angular covering radius. Every Voronoi cell is contained in a cap of radius $\mo_n$, and hence
$\max_i p_{i,n}\leq v_d(\mo_n)$. Now take a covering of the sphere by $N_d(\theta/2)$ caps of radius $\theta/2$ and centres $c_k$, $k=1,\ldots,N_d(\theta/2)$. Assume that each cap contains at least one $Z_i$. Then, any $z\in\SS_{d-1}(1)$ belongs to some cap with centre $c_k$ and radius $\theta/2$ and this cap also contains a $Z_i$, implying that $\mo_n\leq\theta$.
A union bound gives
\bea
\Pr\{\mo_n>\theta\}
\leq N_d(\theta/2)[1-v_d(\theta/2)]^n
\leq N_d(\theta/2)e^{-nv_d(\theta/2)}\,.
\eea
It follows that
$\Ex\{\max_i p_{i,n}\}\leq v_d(\theta) + \Pr\{\mo_n>\theta\}$, proving \eqref{eq:Voronoi-variance2}.

For fixed $d$, the i.i.d.\ directions are almost surely dense on the sphere, so $\mo_n\to0$ and $\max_i p_{i,n}\to0$ almost surely. Dominated convergence in \eqref{eq:Voronoi-variance1} proves \eqref{eq:fixed-d-concentration}.
\end{proof}

For $s$ even, $\widehat E_{\mu,s}^s(n;a)$ can be expressed explicitly in terms of moments of $\|U\|$, $M_{\Psi,k}=\Ex_\mu\{\|U\|^k\}=\int_{r\geq 0}  r^k\,\dd \Psi(r)$, and moments $\Ex\{\zeta^k\}=\Gamma(1+k/\delta)$ of the Weibull distribution. In particular, for $s=2$ and 4 we get the following extreme-value approximations for the $(\mu,s)$-distortions:
\be
\widehat E_{\mu,2}^2(n;a) &=& M_{\Psi,2} - 2\,a\,M_{\Psi,1}+a^2 + 4\,a\,\kappa_{n,d}\,M_{\Psi,1}\,\Gamma(1+1/\delta)  \,, \label{E2^2-general}\\
\widehat E_{\mu,4}^4(n;a) &=& M_{\Psi,4} - 4\,a\,M_{\Psi,3}+6\,a^2\,M_{\Psi,2}-4\,a^3\,M_{\Psi,1}+a^4 \nonumber \\
&& + \, 8\,a\,\kappa_{n,d}\,\Gamma(1+1/\delta)
\left( M_{\Psi,3} - 2\,a\,M_{\Psi,2}+a^2\,M_{\Psi,1} \right) \nonumber \\
&& + \, 16\,a^2\,\kappa_{n,d}^2\,M_{\Psi,2}\,\Gamma(1+2/\delta) \,. \label{E4^2-general}
\ee

Denote by $\widehat {a^*}=\widehat {a^*}(d,n,s)$ the value of $a$ that minimises $\widehat E_{\mu,s}(n;a)$; it is an extreme-value approximation of the exact minimiser $a^*$ of $D_{\mu,s}(\PP_a^{[n]})$. For $s=2$, Lemma~\ref{L:main-extreme-value2} gives $m_{n,d}\simeq\kappa_{n,d}\Gamma(1+1/\delta)$ in \eqref{exact-mnd}. The next corollary is therefore the extreme-value counterpart of Corollary~\ref{Coro:astar-s2-exact}.

\begin{corollary}\label{Coro:astar-s2-Pa}
For $d\geq 3$, $\mu$ spherically symmetric and $\Rb_n\simd \PP_a^{[n]}$, we have
\be
\widehat{a^*} = \widehat{a^*}(d,n,2) &=& M_{\Psi,1}\,\left[ 1 - 2\,\kappa_{n,d}\,\Gamma(1+1/\delta) \right]\,, \label{hat_a_2} \\
\widehat E_{\mu,2}^2(n;\widehat {a^*}) &=& M_{\Psi,2} - (\widehat {a^*})^2 \,. \nonumber
\ee
\end{corollary}

The top-left panel of Figure~\ref{F:sphere} presents plots of $a^*(d,n,2)$ and $\widehat {a^*}(d,n,2)$ for various $n$ and $d$ when $\mu$ is uniform on $\SS_{d-1}(1)$, showing the fast convergence of $\widehat {a^*}$ to $a^*$ as $n$ increases. On the top-right panel, the plots of
$D_{\mu,s}^{1/s}(\PP^{[n]}_{\widehat{a^*}})/D_{\mu,s}^{1/s}(\PP^{[n]}_1)$ and $D_{\mu,s}^{1/s}(\PP^{[n]}_{a^*})/D_{\mu,s}^{1/s}(\PP^{[n]}_1)$ are indistinguishable. In the right panel of  Figure~\ref{F:normal}, where $\mu$ is normal, the plots of $a^*(d,n,2)$ and $\widehat {a^*}(d,n,2)$ are also indistinguishable.
Other illustrations are provided in the supplementary material.

The expressions for $\widehat {a^*} =\widehat {a^*}(d,n,4)$ and $\widehat E_{\mu,4}^4(n;\widehat {a^*})$ can also be written in analytic form (by finding the roots of a third-degree polynomial), but the expressions are cumbersome.

\subsection{Three asymptotic regimes for $n$ growing with $d$}\label{S:n-growing-with-d}

In view of Lemma~\ref{L:main-extreme-value2}, the asymptotic behaviour of $\kappa_{n,d}$ is the key element in understanding that of $d^2(\ub,\Rb_n)$. For $d=3$, $\beta_{1,1}$ is uniform on $[0,1]$ and we simply have $\kappa_{n,3}=1/n$. For $d>3$ there is no explicit formula for $\kappa_{n,d}$, but it can easily be computed numerically by solving $I_{\kappa_{n,d}}(\delta,\delta)=1/n$. For fixed $t\leq 1/2$, $I_t(\delta,\delta)$ decreases as $\delta$ increases. Therefore, $\kappa_{n,d}$ is an increasing function of $d$ for any fixed $n\geq 2$ (it is also a decreasing function of $n$ for any fixed $d$). When $n$ tends to infinity with $d$ we have the following properties.

\begin{proposition} \label{th:kappa_n}
For fixed $d\geq 3$, the $1/n$ quantile $\kappa_{n,d}$ of  $\beta_{\delta,\delta}$ with $\delta=(d-1)/2$ tends to zero as $n$ tends to infinity. When $d \to \infty$ and $n=n(d)$ grows with $d$, we have the following three cases:\\
\noindent{\textit(i)} if $n^{1/d}\to\infty$, then $\lim_{d\to\infty} \kappa_{n,d} = 0$;\\
\noindent{\textit(ii)} if $n^{1/d}\to\ml$ with $\ml>1$, then
\be \label{asymptotic-kappa} \lim_{d\to\infty} \kappa_{n,d} = \frac12 \left[1-\sqrt{1-1/\ml^2}\right]\, ;
\ee
\noindent{\textit(iii)} if $(\log n)/d \to 0$, then $\lim_{d\to\infty} \kappa_{n,d} = \frac12$.
\end{proposition}

\begin{proof}
Since $x^{\delta-1}(1-t)^\delta\leq x^{\delta-1}(1-t)^{\delta-1}\leq x^{\delta-1}(1-x)^{\delta-1}$ for $\delta\geq 1$ and all $t\in[0,1)$ and $x\in[0,t]$, we have
\be\label{It-ineq1}
\frac{t^\delta (1-t)^\delta}{\delta B(\delta,\delta)} \leq I_t(\delta,\delta) \,. \ee
The derivation of an exploitable upper bound is more delicate.
Assume that $d \geq 3$ and let
$\delta'=\lfloor \delta \rfloor$, the largest integer smaller than or equal to $\delta$.
As $n\geq 2$, we are only interested in values of $t$ less than $1/2$. Since $\delta'\leq \delta$, we have $I_t(\delta,\delta) \leq I_t(\delta',\delta')$, which
can be calculated explicitly by successive integration by parts. Direct calculations yield the following special case of the well-known relation between the c.d.f.\ of the beta and binomial distributions:
\be\label{It1}
I_t(\delta',\delta') &=& \sum_{k=0}^{\delta'-1} \binom{2\delta'-1}{k} t^{2\delta'-1-k}(1-t)^k \,, \\
&=& t^{\delta'}(1-t)^{\delta'-1} \sum_{k=0}^{\delta'-1} \binom{2\delta'-1}{k} \left(\frac{t}{1-t}\right)^{\delta'-1-k} \,.
\ee
Using the property $t/(1-t) \leq 1$ for $t\leq 1/2$, with $\sum_{k=0}^{\delta'-1} \binom{2\delta'-1}{k}=4^{\delta'-1}$, we thus obtain
\be \label{upper_bound}
I_t(\delta,\delta) \leq I_t(\delta',\delta') \leq t^{\delta'}[4\,(1-t)]^{\delta'-1} \leq \frac12\,[4\,t(1-t)]^{\delta'-1}\,, \ t\leq 1/2\,.
\ee

Together with \eqref{It-ineq1}, this implies, for $n\geq 2$ and $d\geq 5$,
\be\label{bounds-kappa}
\frac12 \left[1-\sqrt{1- (2/n)^{1/(\delta'-1)}}\right]  \leq \kappa_{n,d} \leq \frac12 \left[1-\sqrt{(1-4\,c_d\,n^{-1/\delta})_+}\right] \,,
\ee
where $(x)_+=\max\{x,0\}$ and $c_d=[\delta\,B(\delta,\delta)]^{1/\delta} = 1/4+ (\log d)/(4\,d)+\SO(1/d)$, $d\to\infty$ (with, moreover, $(4\,c_d)^\delta<d+2$ for all $d \geq 2$, implying that the right-hand side of \eqref{bounds-kappa} is strictly smaller than  1/2 for $n\geq d+2$).

When $d$ is fixed and $n\to\infty$, or $n^{1/d}\to\infty$ as $d\to\infty$, $n^{-1/d}\to 0$ implying that the right-hand side of \eqref{bounds-kappa} tends to zero. Therefore, $\kappa_{n,d}\to 0$.

In case (\textit{ii}), $n^{-1/d}\to\ml^{-1}$. Taking limits on both sides of \eqref{bounds-kappa}, and using $\delta/d\to1/2$, $\delta'/d\to1/2$, and $c_d\to1/4$, we get $\lim_{d\to\infty}\kappa_{n,d}=(1/2)[1-\sqrt{1-1/\ml^2}]$.

When $\log n = o(d)$, (case (\textit{iii})), $n^{-1/d}\to 1$ and the left-hand side of \eqref{bounds-kappa} tends to 1/2.
\end{proof}

The left panel of Figure~\ref{F:sphere-approx0} plots $\kappa_{n,d}$ as a function of $d$ for different $n$: the continuation of each curve (with $n$ fixed) goes to the limiting value 1/2 in view of Proposition~\ref{th:kappa_n}-(\textit{iii}). The bounds in \eqref{bounds-kappa} are asymptotically accurate for large $n$, particularly the upper bound, but lack precision for smaller $n$. For illustration, only the envelope for $n = 10^3$ is shown.

\vsp
The asymptotic behaviour of $\kappa_{n,d}$ as $n(d)$ grows with $d$ governs that of the extreme-value approximation of the $(\mu,s)$-distortion from Definition~\ref{D:main-extreme-value} and of the extreme-value approximation $\widehat{a^*}$ of the minimiser; see Corollary~\ref{Coro:astar-s2-Pa}. Under suitable (non-restrictive) assumptions on $\mu$, we show below (Theorem~\ref{Th:joint-exact-distortion-limit}) that the optimal radius $a^*(d,n(d),s)$
converges to a limit $a_\infty$ that depends on the growth regime of $n(d)$ but not on $s$. We also derive the limit of the associated expected $(\mu,s)$-distortion.

Define
\be\label{joint-kappa-limits}
\kappa_\infty=
\left\{
\begin{array}{ll}
0 & \mbox{if } n^{1/d}\to\infty\,,\\[1mm]
\displaystyle \frac12\left(1-\sqrt{1-1/\ml^2}\right)
& \mbox{if } n^{1/d}\to\ml\in(1,\infty)\,,\\[3mm]
1/2 & \mbox{if } (\log n)/d\to0\,.
\end{array}
\right.
\ee

\begin{theorem}\label{Th:joint-exact-distortion-limit}
For $s>0$ and each $d$, let $\mu_d$ be a probability measure on $\RR^d$, let $U_d\simd\mu_d$, $R_d=\|U_d\|$, and assume that, for some
$r>0$ and $\eta>0$,
\be\label{joint-radius-assumptions}
R_d \to r \mbox{ in probability as } d\to\infty \mbox{ and } \sup_d\Ex\{R_d^{s+\eta}\}<\infty\,.
\ee
Let $n=n(d)\to\infty$ satisfy one of the three regimes in \eqref{joint-kappa-limits}. If $a_d\to a\geq0$, then
\be\label{joint-exact-D-limit}
D_{\mu_d,s}(\PP_{a_d}^{[n(d)]}) \to
\left[(r-a)^2+4\,ar\kappa_\infty\right]^{s/2}\,, \ d\to\infty\,.
\ee
Let $a_d^*=a^*(d,n(d),s)$ be any minimiser of
$D_{\mu_d,s}(\PP_a^{[n(d)]})$ over $a\geq0$. Then
\be\label{joint-exact-a-limits}
a_d^*\to r(1-2\kappa_\infty)
=\left\{
\begin{array}{ll}
r & \mbox{if } n^{1/d}\to\infty\,,\\[1mm]
r\sqrt{1-1/\ml^2}
& \mbox{if } n^{1/d}\to\ml\in(1,\infty)\,,\\[1mm]
0 & \mbox{if } (\log n)/d\to0\,,
\end{array}
\right.
\ee
and
\be\label{joint-exact-min-D-limits}
\min_{a\geq0}D_{\mu_d,s}(\PP_a^{[n(d)]})
\to
\left[4\,r^2\kappa_\infty(1-\kappa_\infty)\right]^{s/2}
=\left\{
\begin{array}{ll}
0 & \!\!\! \mbox{if } n^{1/d}\to\infty\,,\\[1mm]
(r/\ml)^s
& \!\!\!\mbox{if } n^{1/d}\to\ml\in(1,\infty)\,,\\[1mm]
r^s & \!\!\! \mbox{if } (\log n)/d\to0\,.
\end{array}
\right.
\ee
\end{theorem}

\begin{proof}
In each of the three regimes considered, Proposition~\ref{th:kappa_n} gives
$\kappa_{n(d),d}\to\kappa_\infty$. Consider the random variable $Z_d=\zeta(n(d),d)$ in \eqref{d2URn}. We first show that, for every $p>0$, $Z_d \to \kappa_\infty$ in $L^p$ (and thus in particular, $m_{n(d),d}=\Ex\{Z_d\}\to\kappa_\infty$).

Fix $M>1$ and define
\bea
n_d^+=\lceil M n(d)\rceil
\quad\mbox{and}\quad
n_d^-=\lfloor n(d)/M\rfloor\,.
\eea
Multiplication of $n(d)$ by a fixed positive constant does not change its asymptotic regime. Therefore,
$\kappa_{n_d^+,d}\to\kappa_\infty$ and $\kappa_{n_d^-,d}\to\kappa_\infty$.
For every $\epsilon>0$ and all sufficiently large $d$ we have
\bea
\Pr\{Z_d\leq\kappa_\infty-\epsilon\}
&\leq&
\Pr\{Z_d\leq\kappa_{n_d^+,d}\}
=1-\left(1-\frac1{n_d^+}\right)^{n(d)}
\leq\frac1M\,,\\
\Pr\{Z_d>\kappa_\infty+\epsilon\}
&\leq&
\Pr\{Z_d>\kappa_{n_d^-,d}\}
=\left(1-\frac1{n_d^-}\right)^{n(d)}
\leq\exp(-M)\,.
\eea
Letting first $d\to\infty$ and then $M\to\infty$ proves convergence in probability. Since $0\leq Z_d\leq1$, convergence holds in $L^p$ for every $p>0$, which also gives $m_{n(d),d}=\Ex\{Z_d\}\to\kappa_\infty$.

Now, consider \eqref{d2URn} again. Conditioning on $U_d$ gives the exact representation
\be\label{joint-exact-representation}
D_{\mu_d,s}(\PP_a^{[n(d)]})
=\Ex\left\{\left[(R_d-a)^2+4\,aR_dZ_d\right]^{s/2}\right\}\,,
\ee
where $Z_d$ may be taken independent of $R_d$. As $Z_d\to\kappa_\infty$ in probability, if $a_d\to a$ then
\bea
(R_d-a_d)^2+4\,a_dR_dZ_d \to (r-a)^2+4\,ar\kappa_\infty
\eea
in probability. Moreover,
\bea
0\leq
\left[(R_d-a_d)^2+4\,a_dR_dZ_d\right]^{s/2}
\leq(R_d+a_d)^s\,.
\eea
Assumption \eqref{joint-radius-assumptions} therefore gives uniform integrability, and \eqref{joint-exact-D-limit} follows.
Using truncation of $R_d$ and uniform continuity on compact sets, we obtain that the convergence in \eqref{joint-exact-D-limit} is locally uniform in $a$. The minimisers $a_d^*$ form a bounded sequence: indeed,
\bea
D_{\mu_d,s}(\PP_a^{[n(d)]})
\geq\Ex\{|R_d-a|^s\}\,,
\eea
whereas the choice $a=0$ gives $D_{\mu_d,s}(\PP_0^{[n(d)]})=\Ex\{R_d^s\}$, which is uniformly bounded. The standard argmin argument can therefore be applied to the limiting function
\bea
L(a)=\left[(r-a)^2+4\,ar\kappa_\infty\right]^{s/2}\,.
\eea
Since $s>0$, its unique minimiser is $a=r(1-2\kappa_\infty)$, and its minimum is $[4\,r^2\kappa_\infty(1-\kappa_\infty)]^{s/2}$. Substitution of the three values in \eqref{joint-kappa-limits} gives \eqref{joint-exact-a-limits} and \eqref{joint-exact-min-D-limits}.
\end{proof}

Under the assumptions of Theorem~\ref{Th:joint-exact-distortion-limit}, the extreme-value
approximation $\widehat D_{\mu_d,s}(\PP_a^{[n(d)]})$ of Definition~\ref{D:main-extreme-value} also satisfies $\widehat D_{\mu_d,s}(\PP_a^{[n(d)]})\to
\left[(r-a)^2+4\,ar\kappa_\infty\right]^{s/2}$ as $d\to\infty$ and $\widehat{a^*}(d,n(d),s)$ satisfies \eqref{joint-exact-a-limits}: compared with the proof of Theorem~\ref{Th:joint-exact-distortion-limit}, we now consider the random variable $T_{n,d}=(R_d-a)^2 + 4\,a\,R_d\,\kappa_{n,d}\,\zeta$ in \eqref{Ds-approx}, where again $\kappa_{n,d}\to\kappa_\infty$.

The assumptions \eqref{joint-radius-assumptions} on $\mu$ are satisfied when $\mu$ possesses a {\em norm-concentration property}, i.e., when it satisfies
\bea
\mbox{for all } \me>0\,, \quad \mu\left\{\left| \|U\|-r \right| \geq \me\right\} \leq 2\,\exp(-c\, d\, \me^2)
\eea
when $U\sim\mu$, for some constant $c$ not depending on $d$.
This is the case in particular for $\mu$  uniform in the unit ball $\SB_d(1)$ as in Section~\ref{S:NSC-examples-ball} and for $\mu$ normal as in Section~\ref{S:NSC-examples-normal}. The numerical evolution of $n_1(d,s)$ in Figure~\ref{F:asymptotics-ball_s2} suggests that, in these two situations and for sufficiently small $\ml>1$, quantisers with distribution $\PP_{\sqrt{1-1/\ml^2}}$ may be asymptotically optimal among all random quantisers when $n^{1/d}\to\ml$. This conclusion is empirical and is not established by Theorem~\ref{Th:joint-exact-distortion-limit}, which optimises only over single-sphere distributions.

\vsp
When $\mu$ is uniform on $\SS_{d-1}(1)$, the same arguments as those leading to Definition~\ref{D:main-extreme-value} indicate that, for any $\Rb_n\simd \PP_a^{[n]}$, the extreme-value approximation of the $\mg$-quantile of the distance c.d.f.\ $F(t; \Rb_n,\mu)$ is $\widehat q_\mg = \sqrt{(1-a)^2+4\,a\,\kappa_{n,d}\,t_\mg}$, with $t_\mg=[-\log(1-\mg)]^{1/\delta}$ the $\mg$-quantile of the Weibull distribution with c.d.f.\ $F_d(t)=1-\exp(-t^\delta)$, $\delta=(d-1)/2$.

The minimum of $\widehat q_\mg$ with respect to $a\geq0$ is obtained for $\widehat {a_*}(d,n,\mg)=[1-2\,\kappa_{n,d}\,t_\mg]_+$. Since $t_\mg\to 1$ as $d\to\infty$, for any fixed $\mg$, $\widehat {a_*}(d,n(d),\mg)$ has the same limits as $\widehat {a^*}(d,n(d),s)$, independently of $\mg$, when $n = n(d) \to \infty$ according to one of the three regimes in \eqref{joint-kappa-limits}.

\begin{remark}\label{R:quantiles} Note that $\widehat q_\mg = \widehat E_{\mu,2}(n;a)$ for any $a$ when $\mg=\mg(d,2)=F_d[\Gamma(1+1/\delta)]$, showing that for large $n$, minimising the expected $(\mu,2)$-distortion is equivalent to minimising the $\mg$-quantile of the mean distance c.d.f.\ for
\bea
\mg=\mg(d,2)=1- \e1^{-\e1^{-\mgbb}} + \frac{\pi^2}{6\,d}\, \e1^{-\mgbb-\e1^{-\mgbb}} +\SO(d^{-2})  \,, \ d\to\infty\,,
\eea
where $\mgbb$ is Euler's constant and $1- \e1^{-\e1^{-\mgbb}}\simeq 0.429624$. This suggests that $L_2$-quantisation is roughly equivalent to minimising the median of the distance c.d.f., a phenomenon that can also be observed in other situations.
\fin
\end{remark}

\subsection{Example: $\mu$ uniform on the unit sphere $\SS_{d-1}(1)$} \label{S:extreme-value-sphere}

When $U\simd \mu$ uniform on $\SS_{d-1}(1)$, for any $\Rb_n\simd \PP_a^{[n]}$ the approximations \eqref{E2^2-general} and \eqref{E4^2-general}
become
\bea
\widehat E_{\mu,2}(n;a) \!\!\! &=& \!\!\! \left[ (1-a)^2 + 4\,a\,\kappa_{n,d}\,\Gamma(1+1/\delta) \right]^{1/2} \,, \\
\widehat E_{\mu,4}(n;a) \!\!\! &=& \!\!\! \left[ (1-a)^4 + 16\,a^2\,\kappa_{n,d}^2\,\Gamma(1+2/\delta) + 8\,a(1-a)^2\,\kappa_{n,d}\,\Gamma(1+1/\delta) \right]^{1/4} \,.
\eea
For $s=2$, Corollary~\ref{Coro:astar-s2-Pa} gives $\widehat {a^*}=\widehat {a^*}(d,n,2) = 1 - 2\,\kappa_{n,d}\,\Gamma(1+1/\delta)$ and $\widehat E_{\mu,2}^2(n;\widehat {a^*}) =1 - (\widehat {a^*})^2$.
The right panel of Figure~\ref{F:sphere-approx0} plots $\widehat {a^*}(d,n,2)-a^*(d,n,2)$ as a function of $d$ for different values of $n$: $\widehat {a^*}-a^* >0$ and the accuracy of the approximation of $a^*$ by $\widehat {a^*}$ tends to decrease with $d$ and increase with $n$.

\begin{figure}[ht!]
\begin{center}
\includegraphics[width=.49\linewidth]{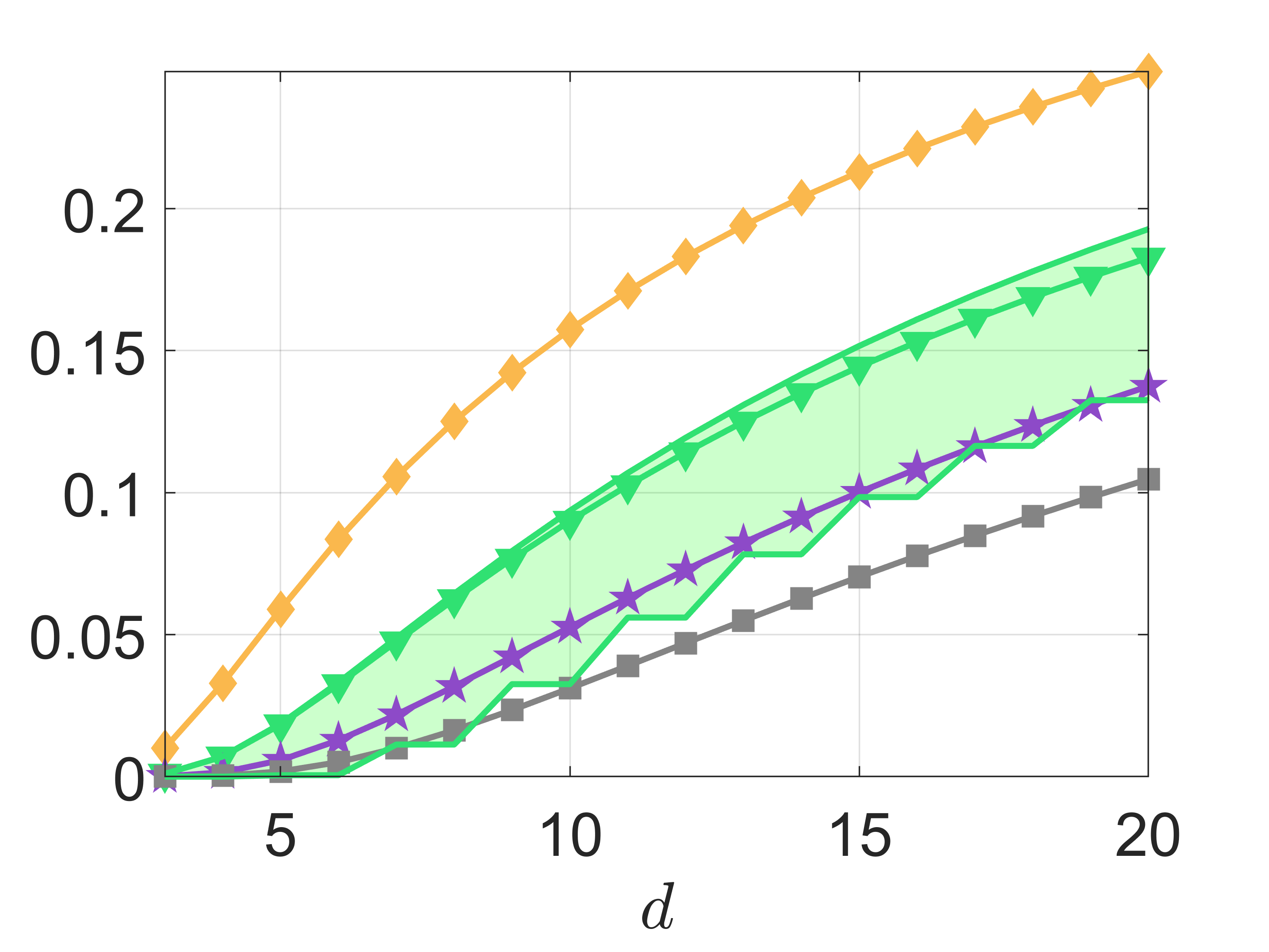}
\includegraphics[width=.49\linewidth]{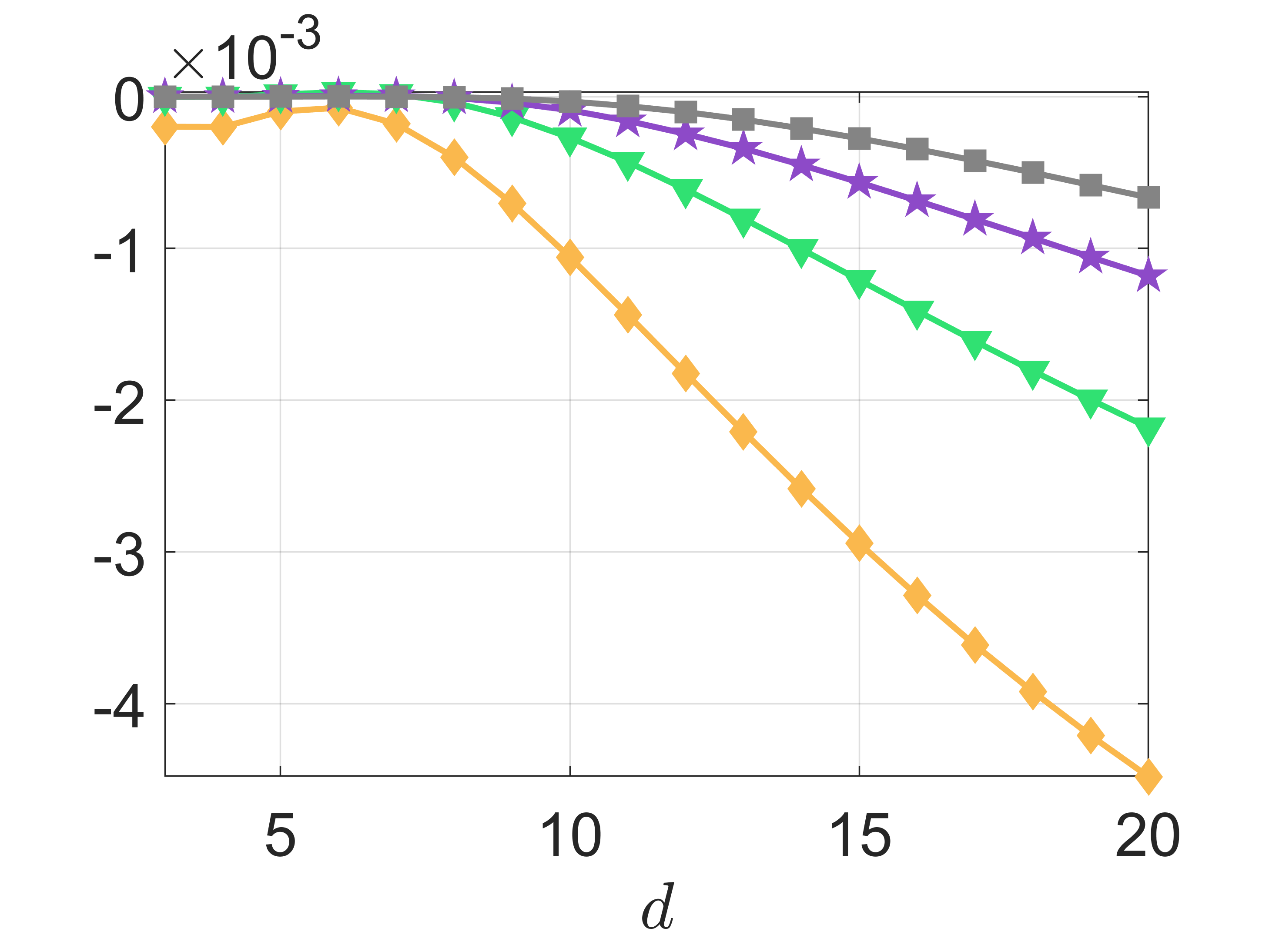}
\end{center}
\caption{\small Values of $\kappa_{n,d}$ (left) and of $\widehat {a^*}(d,n,2)-a^*(d,n,2)$ (right) as functions of $d$ for different values of $n$: $n=10^2$ ({\color{Dandelion} $\blacklozenge$}), $n=10^3$ ({\color{green} $\blacktriangledown$}), $n=10^4$ ({\color{Orchid} $\bigstar$}) and $n=10^5$ ({\color{gray} {\tiny $\blacksquare$}}).
}
\label{F:sphere-approx0}
\end{figure}

The left and right panels of Figure~\ref{F:sphere-approx} show that the two approximations $\widehat D_{\mu,2}(\PP^{[n]}_a)$ and $\widehat D_{\mu,4}(\PP^{[n]}_a)$ are quite accurate when $n$ is large enough for extreme value theory to be applicable. For large $d$, convergence to the extreme-value distribution is slow and $\widehat D_{\mu,s}^{1/s}(\PP^{[n]}_a)$ slightly overestimates the true value $D_{\mu,s}^{1/s}(\PP^{[n]}_a)$.

The left panel of Figure~\ref{F:sphere-approx-2} provides another illustration of the approximation accuracy. It presents normalised values $n^{1/d}\,\widehat D_{\mu,s}^{1/s}(\PP^{[n]}_a)$ and $n^{1/d}\,D_{\mu,s}^{1/s}(\PP^{[n]}_a)$ as functions of $n=2^m$, $m=6,7,\ldots,20$, for $s=4$ and $d=20$, both for fixed $a=0.75$ and for the optimal value $a^*(d,n,s)$ (the plots with $\widehat {a^*}(d,n,s)$ substituted for $a^*(d,n,s)$ are visually indistinguishable).
In the right panel, 100 numerically computed\footnote{We used $2^{12}$ i.i.d.\ $U_i\simd \mu$, which provides an accurate enough estimation of $E_{\mu,s}(\Rb_n)$; see Remark~\ref{R:distortion-evaluation}.} values of $n^{1/d}\,E_{\mu,s}(\Rb_n)$ are shown for each $n$, using independent random quantisers $\Rb_n\simd \PP^{[n]}_{a^*}$. The simulations agree closely with the exact mean $n^{1/d}\,D_{\mu,s}^{1/s}(\PP^{[n]}_{a^*})$ and show very little variability across quantisers, in agreement with Theorem~\ref{Th:distortion-concentration}.

\begin{figure}[ht!]
\begin{center}
\includegraphics[width=.49\linewidth]{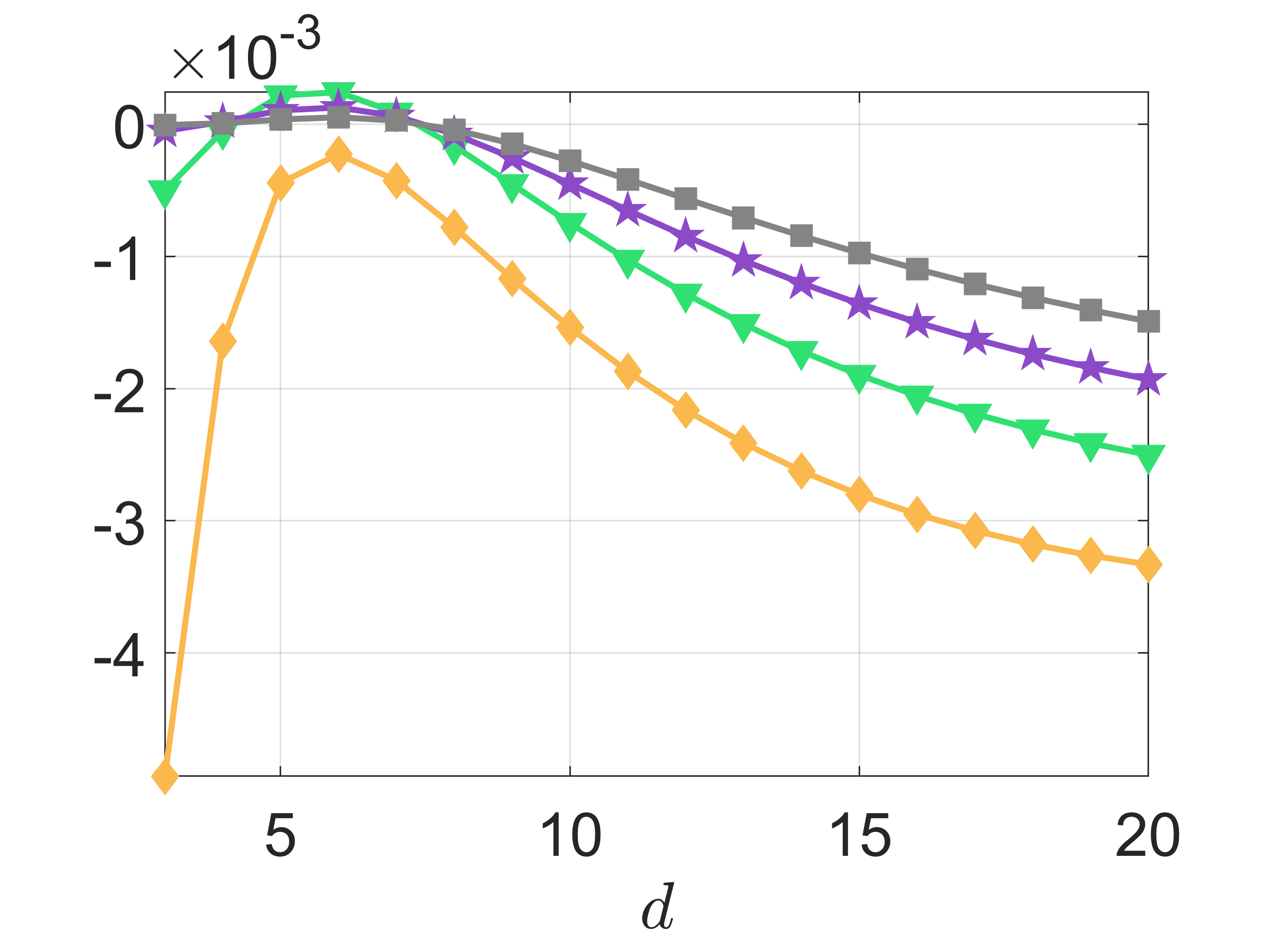}
\includegraphics[width=.49\linewidth]{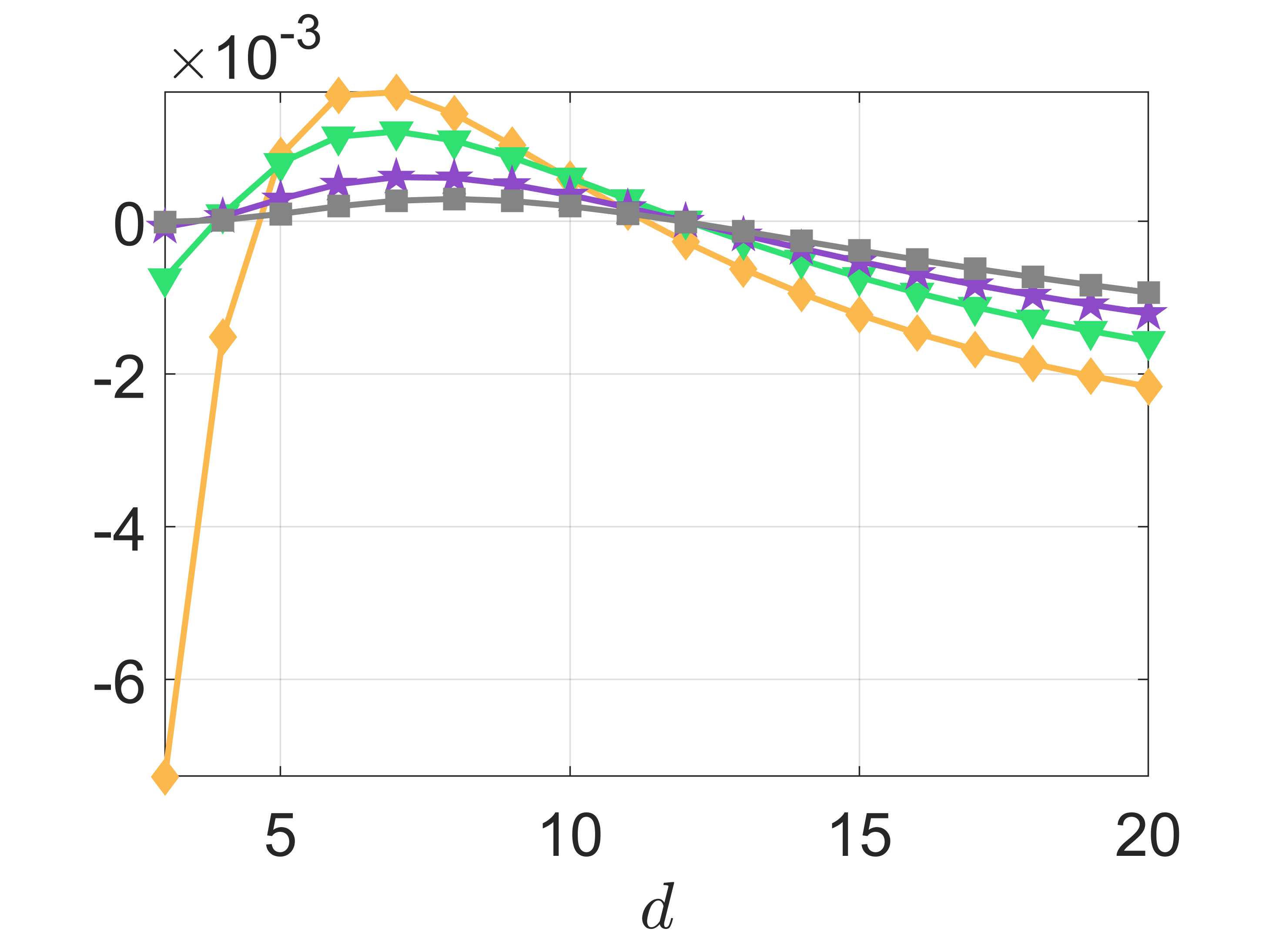} \\
\end{center}
\caption{\small Relative error $1-\widehat D_{\mu,s}^{1/s}(\PP^{[n]}_a)/D_{\mu,s}^{1/s}(\PP^{[n]}_a)$ of the approximation of $D_{\mu,s}^{1/s}(\PP^{[n]}_a)$ (left for $s=2$ and right for $s=4$) based on extreme-value theory, as a function of $d$ for different values of $n$: $n=10^2$ ({\color{Dandelion} $\blacklozenge$}), $n=10^3$ ({\color{green} $\blacktriangledown$}), $n=10^4$ ({\color{Orchid} $\bigstar$}) and $n=10^5$ ({\color{gray} {\tiny $\blacksquare$}}); $a$ is the optimal value $a^*(d,n,s)$.
}
\label{F:sphere-approx}
\end{figure}

\begin{figure}[ht!]
\begin{center}
\includegraphics[width=.49\linewidth]{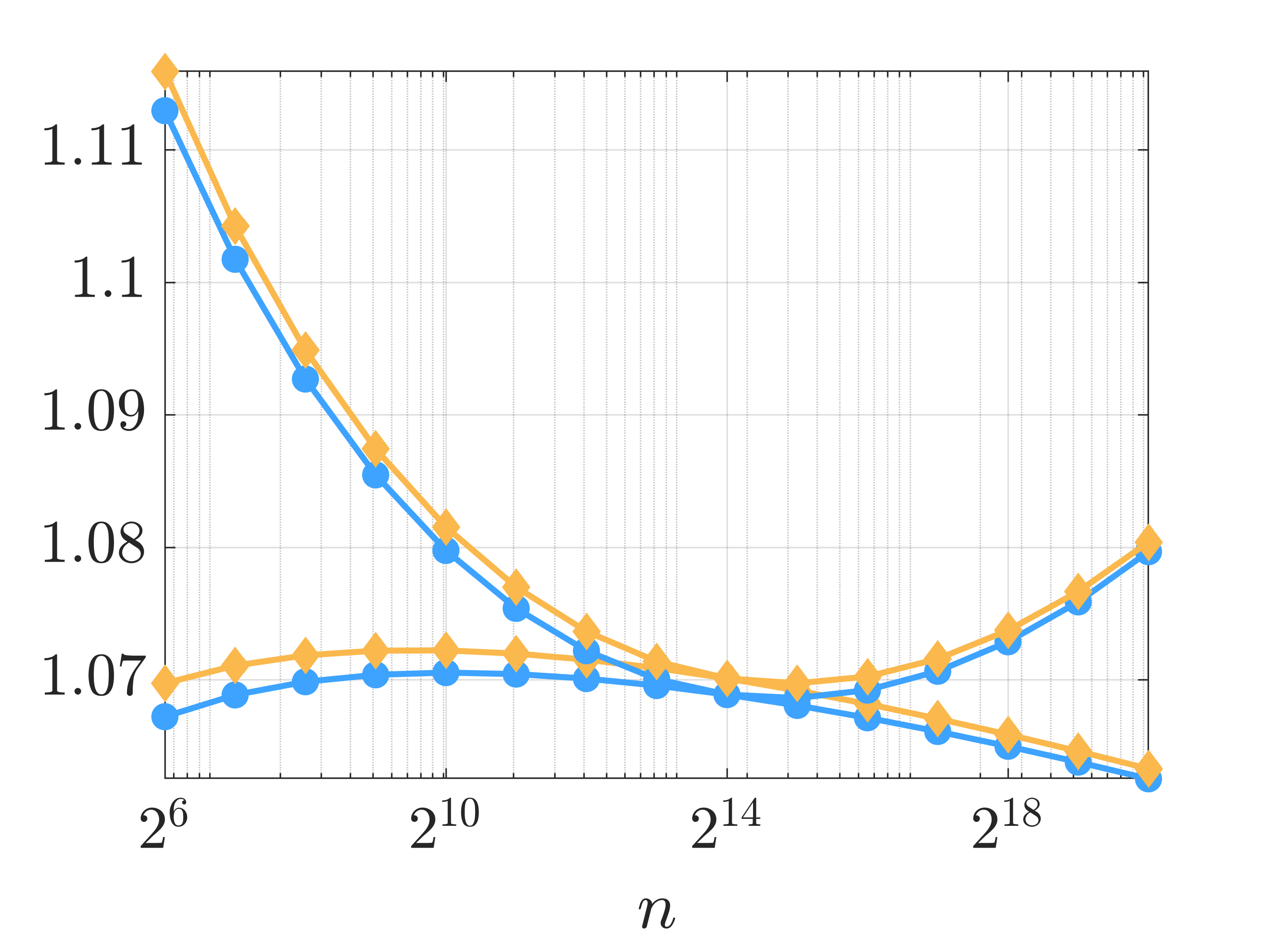}
\includegraphics[width=.49\linewidth]{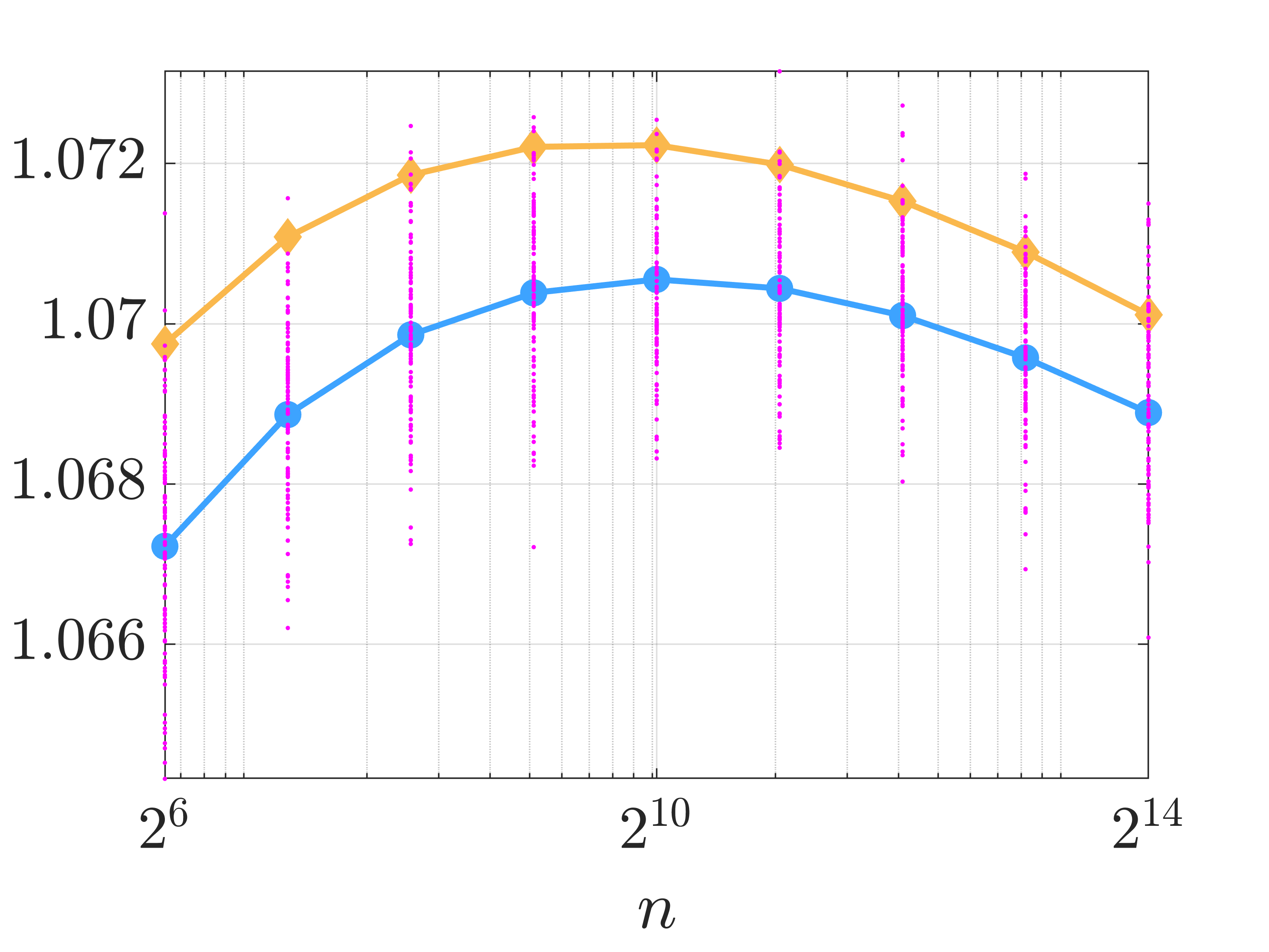}
\end{center}
\caption{\small $n^{1/d}\,\widehat D_{\mu,s}^{1/s}(\PP^{[n]}_a)$ ({\color{Dandelion} $\blacklozenge$}) and $n^{1/d}\,D_{\mu,s}^{1/s}(\PP^{[n]}_a)$ ({\color{Cerulean} $\bullet$}) as functions of $n$ for $s=4$ and $d=20$. Left: $a=0.75$ for top curves and $a=a^*(d,n,s)$ for bottom curves. Right: $a=a^*(d,n,s)$, the values of $n^{1/d}\,E_{\mu,s}(\Rb_n)$ computed numerically for 100 random quantisers $\Rb_n\simd \PP^{[n]}_{a^*}$ are shown as magenta dots.
}
\label{F:sphere-approx-2}
\end{figure}

\section{Conclusions}~\label{S:conclusion}

We have demonstrated that, for a spherically symmetric measure $\mu$ in $\RR^d$, unless the sample size $n$ is extremely large, a random quantiser uniformly distributed on a sphere of suitable radius can significantly outperform a random quantiser whose distribution follows the asymptotic limit predicted by Zador’s theorem. The optimal radius can be determined numerically by minimising a triple integral, which can be evaluated with arbitrary precision. Additionally, an approximation of the optimal radius is available, derived from extreme-value theory. Since the variability of the distortion for such random quantisers rapidly diminishes as $n$ increases, this approach is highly practical for constructing high-performance quantisers.

While the restriction to spherically symmetric distributions may appear limiting, it provides a robust foundation for broader applications. For instance, quantising the uniform measure on the $d$-dimensional hypercube $[-1,1]^d$ presents a compelling and challenging problem, particularly in space-filling design for computer experiments (see, e.g., \cite{PronMul2012}).  While greedy quantisation offers an attractive approach for the incremental construction of designs in low dimensions \cite{LuschgyP2015,NogalesPR2021}, it becomes computationally infeasible for large $d$. In such cases, recent results from \cite{noonanhigh2024} allow the uniform measure to be approximated by a spherically symmetric distribution. Similarly, one can consider random quantisers distributed according to a product measure, which can itself be approximated by a spherically symmetric distribution. Preliminary findings then suggest that quantisers distributed on the vertices of a smaller hypercube exhibit promising performance, paving the way for further advances in this direction.

\bibliographystyle{plain}
\bibliography{high-d} 

\clearpage

\subfile{supp-High-d-spherical-arXiv}

\end{document}

%% file: supp-High-d-spherical-arXiv.tex





\setcounter{section}{0}
\setcounter{subsection}{0}
\setcounter{subsubsection}{0}
\setcounter{equation}{0}
\setcounter{figure}{0}


\section*{Supplementary Material for the paper {\em Optimal random quantisers for spherically symmetric distributions}\\
Luc Pronzato\footnote{(corresponding author) Laboratoire I3S, CNRS-Universit\'e C\^ote d'Azur, Sophia Antipolis, France, {\tt pronzato@i3s.unice.fr}} \ and Anatoly Zhigljavsky\footnote{School of Mathematics, Cardiff University, UK, {\tt ZhigljavskyAA@cardiff.ac.uk}}}

\begin{abstract}
This document contains additional material for the parent paper {\em Optimal random quantisers for spherically symmetric distributions}.
We consider the three cases (\textit{i}) $\mu$ is uniform on the sphere $\SS_{d-1}(1)$, (\textit{ii}) $\mu$ is uniform in the ball $\SB_d(1)$, and (\textit{iii}) $\mu$ is spherically normal, for which we provide numerical examples that complement those in the parent paper and illustrate some of the results given there. In case (\textit{i}) we consider random quantisers whose distribution is a mixture of the uniform distribution on a sphere $\SS_{d-1}(a)$ and the delta measure at $\0b_d$. We also compare the quantisation performance of full-factorial $2^d$ designs with that of random quantisers of the same size that are uniformly distributed on a sphere.
In case (\textit{ii}), we compare the performance of random quantisers uniformly distributed, respectively, on a sphere and in a ball, both with optimised radii. In case (\textit{iii}), we compare the performance of random quantisers uniformly distributed on spheres whose radii are derived from asymptotic considerations when $n=n(d)$ grows exponentially with $d$.

The notation follows that of the parent paper. Equation numbers of the form (xx*) refer to equation (xx) in the parent paper. Similarly, references such as xx* are used for figures, definitions, theorems, and other numbered elements from the parent paper.
\end{abstract}



\section{$\mu$ is uniform on the sphere $\SS_{d-1}(1)$} \label{S:sphere}

Here $\mu$ is the spherical measure (uniform on $\SS_{d-1}(1)$), and the distribution of $\|U\|$ for $U\simd\mu$ is the delta measure at $r=1$; see Section~\ref{S:NSC-examples-sphere}*.

\subsection{$\PP=\PP_a$ uniform on $\SS_{d-1}(a)$}\label{S:uniform-on-S}

We first consider random quantisers with distribution $\PP=\PP_a$ uniform on $\SS_{d-1}(a)$  and minimise the expected $(\mu,s)$-distortion $D_{\mu,s}(\PP^{[n]}_a)$, given by (\ref{Ds-sphere}*), with respect to $a\in[0,1]$. The value $a^*=a^*(d,n,s)$ that minimises $D_{\mu,s}(\PP^{[n]}_a)$ with respect to $a$ can be obtained numerically with arbitrary precision for any given $d,n$ and $s$.
The left column of Figure~\ref{F:sphere-supp} presents $a^*(d,n,s)$ as a function of $d$ for different $n$ and $s=1$ (top) and $s=10$ (bottom). For fixed $s$ and $d$, $a^*$ increases with $n$, and $a^*$ tends to 1 when $n$ tends to infinity. For any fixed $d$ and $n$, $a^*(d,n,s)$ is a decreasing function of $s$, with $a^*(d,n,\infty)=0$ if $n$ is small enough.

The central column of Figure~\ref{F:sphere-supp} shows that, as expected, $D_{\mu,s}^{1/s}(\PP^{[n]}_{a^*})$ decreases with $n$ and increases with $d$ and $s$. The right column presents the efficiency of the naive quantiser with $n$ points i.i.d.\ on $\SS_{d-1}(1)$ relative to optimised random quantisers with $n$ points on $\SS_{d-1}(a^*)$. It shows that the asymptotic optimum $a=1$ can be a poor choice for large $d$: indeed, the benefit of using an appropriate $a$ is significant when $d$ is large or $n$ is small.

\begin{figure}[ht!]
\begin{center}
\includegraphics[width=.32\linewidth]{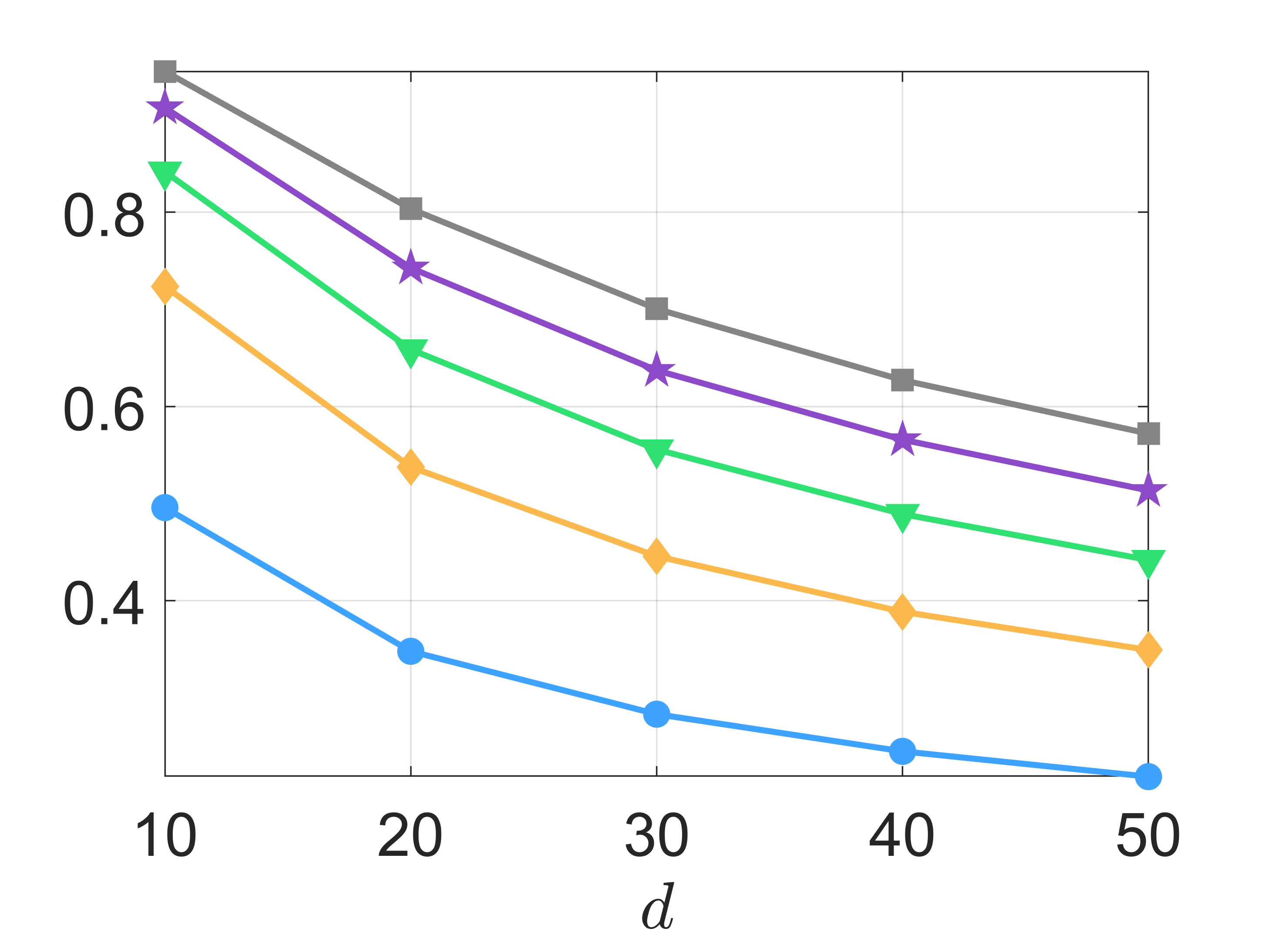}
\includegraphics[width=.32\linewidth]{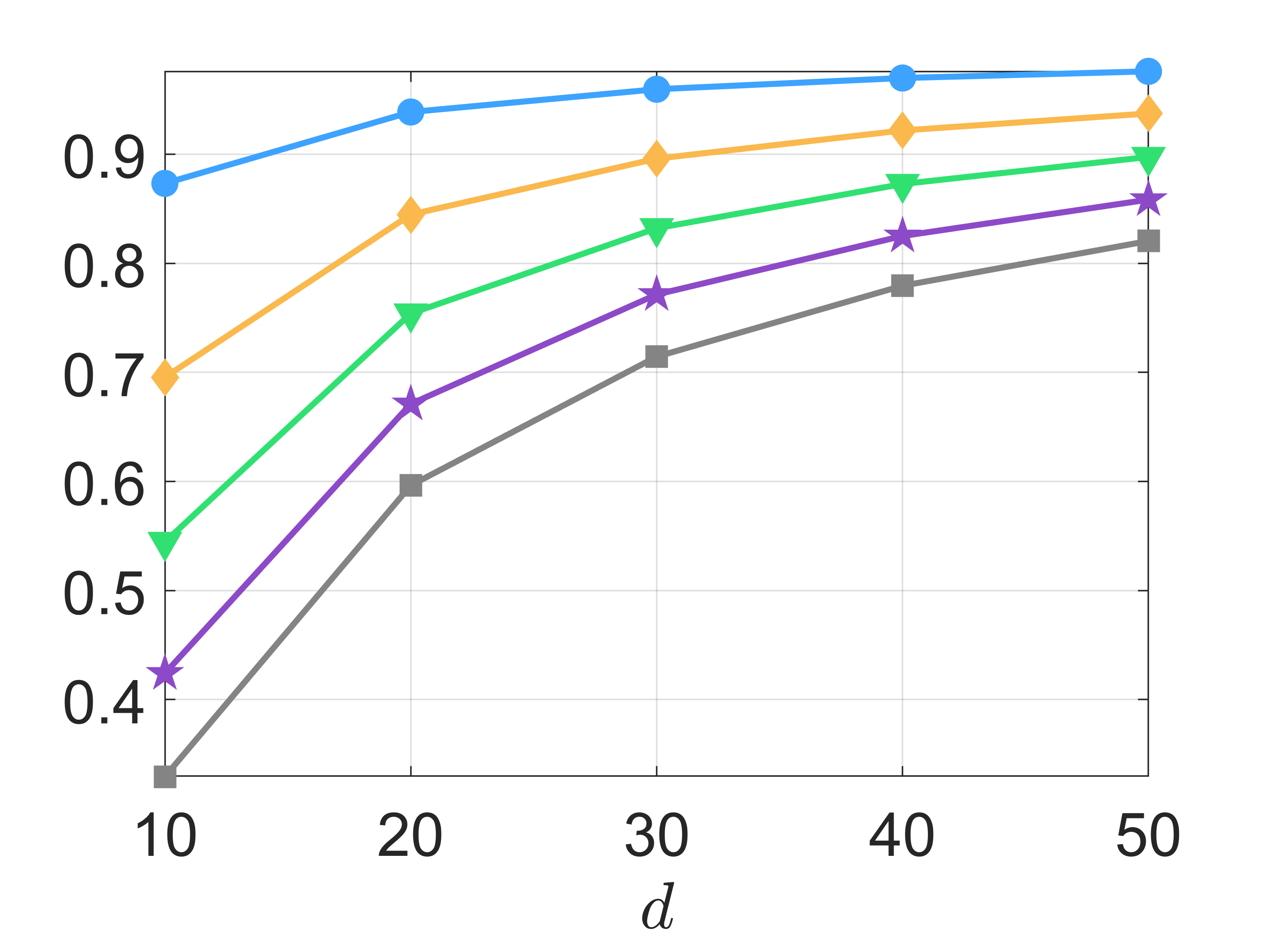}
\includegraphics[width=.32\linewidth]{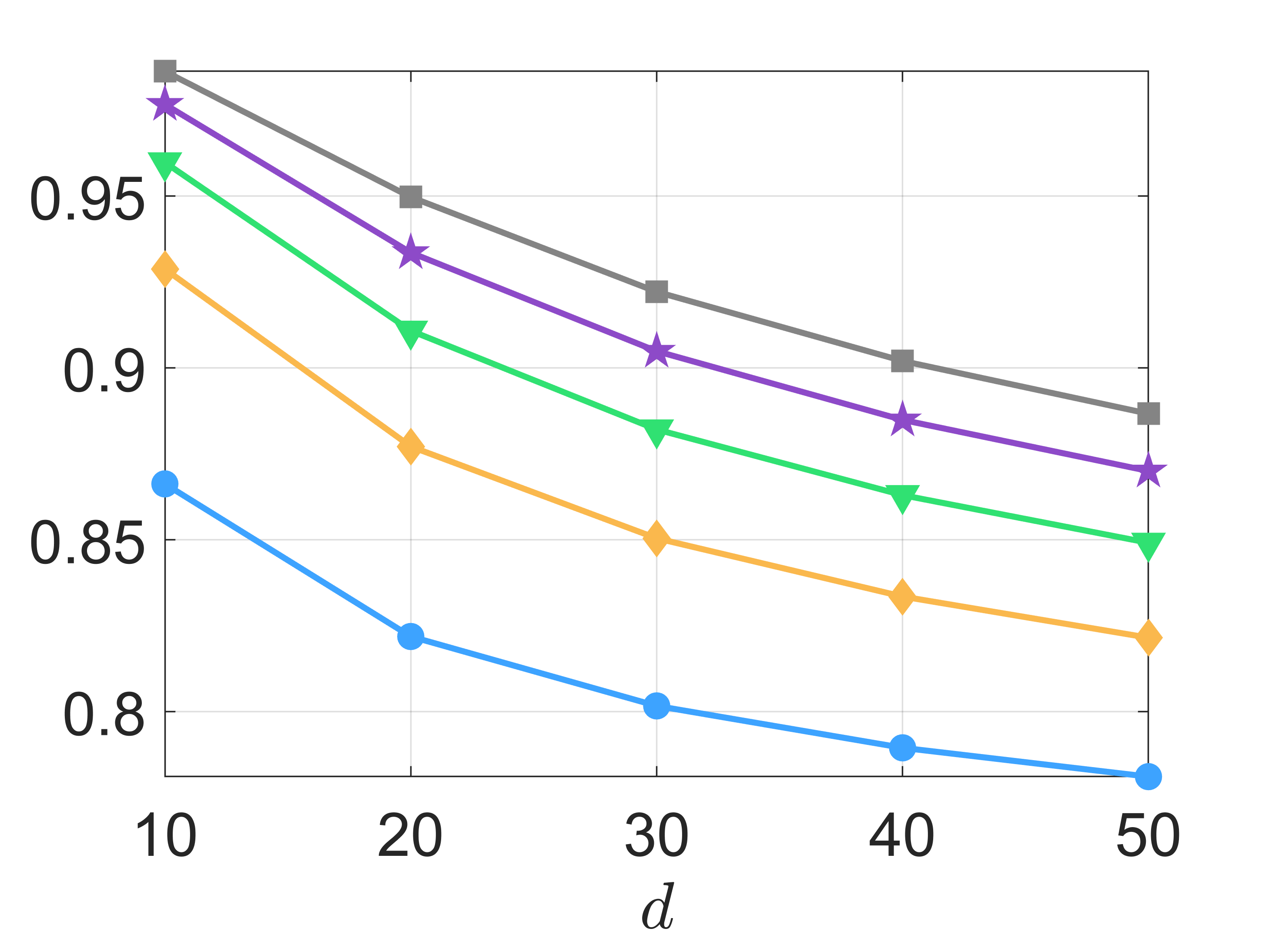} \\
\includegraphics[width=.32\linewidth]{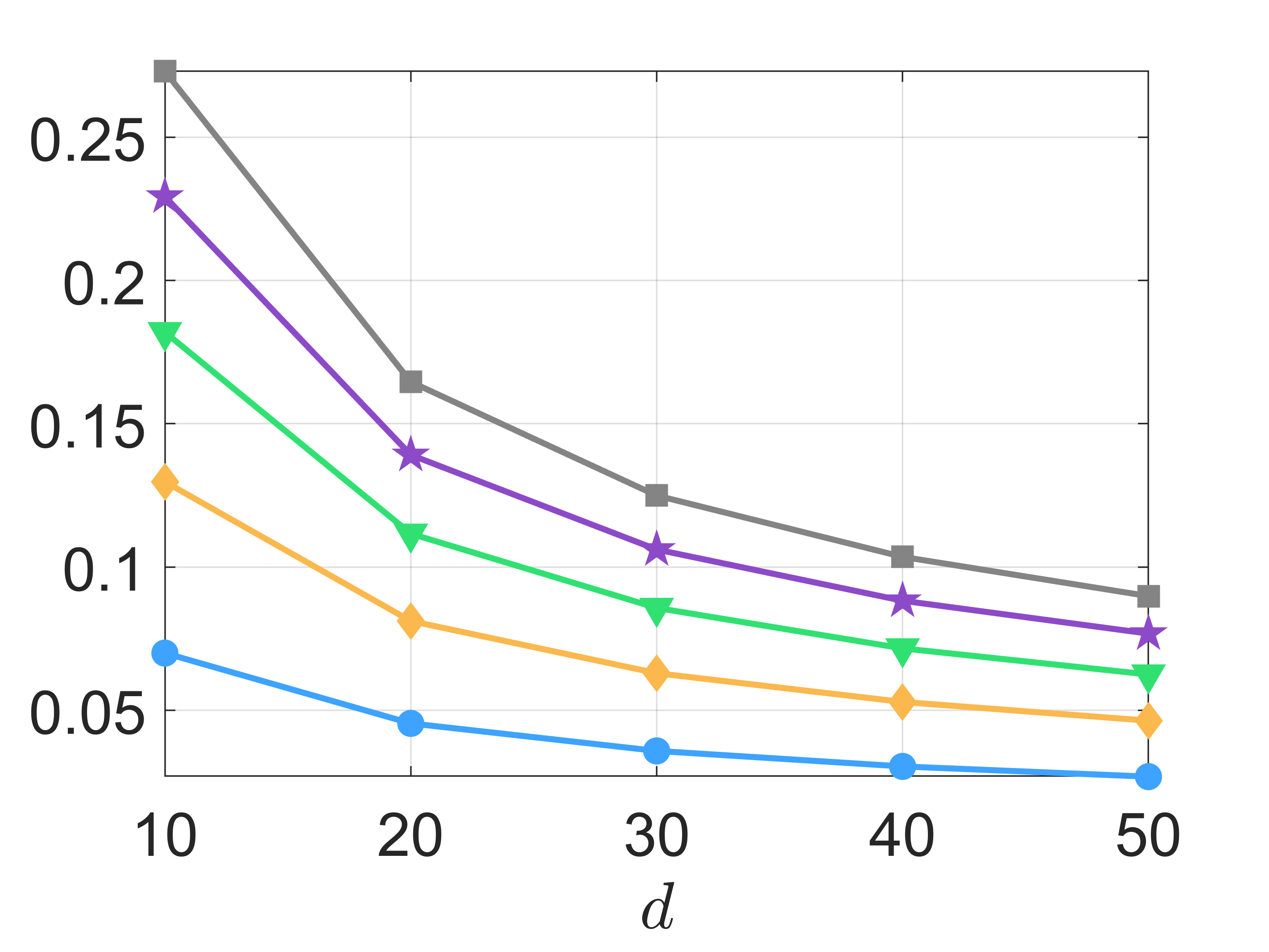}
\includegraphics[width=.32\linewidth]{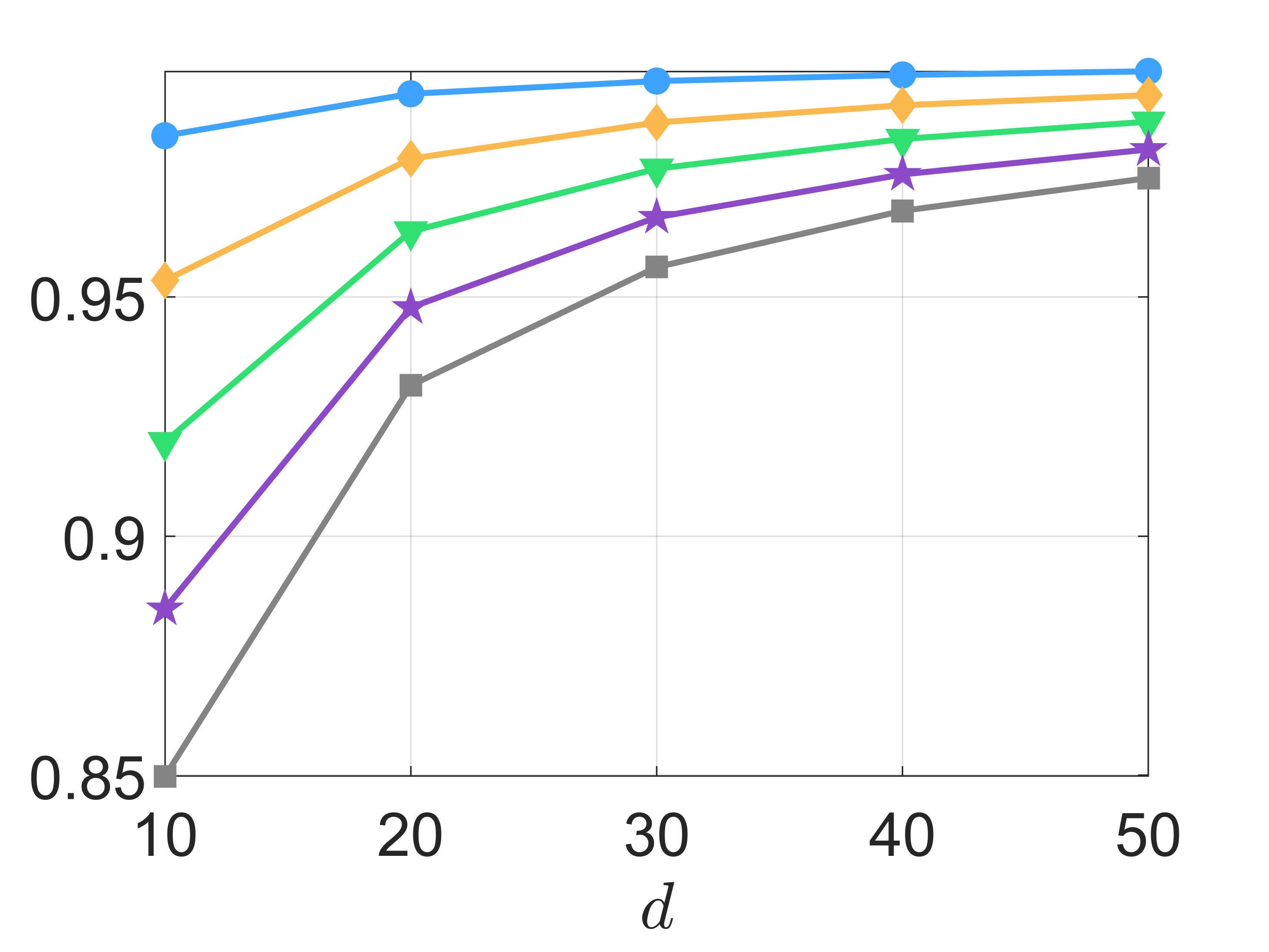}
\includegraphics[width=.32\linewidth]{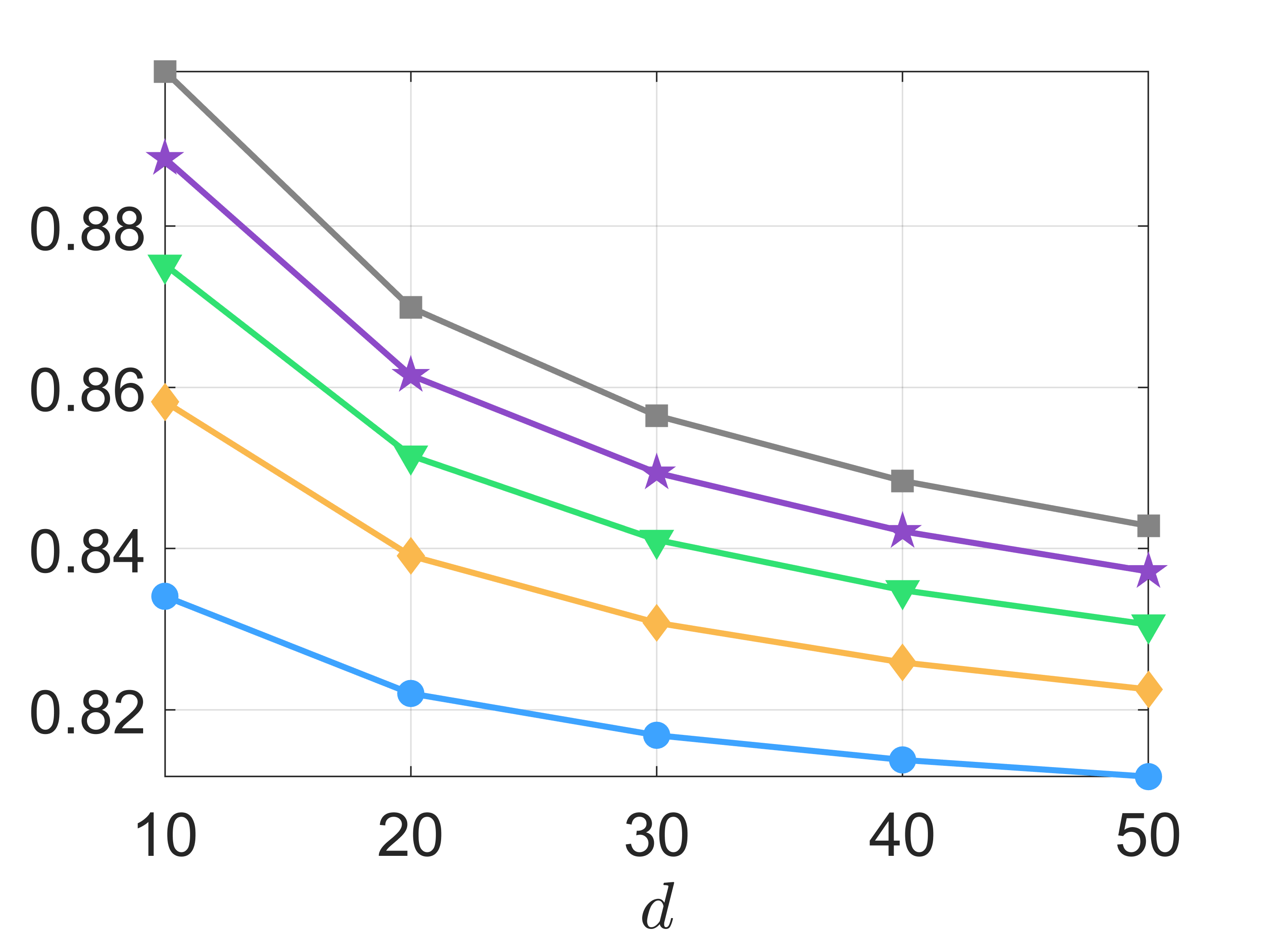}
\end{center}
\caption{\small Sphere. Value $a^*$ minimising $D_{\mu,s}(\PP^{[n]}_a)$ (left column), $D_{\mu,s}^{1/s}(\PP^{[n]}_{a^*})$ (central column) and ratio $D_{\mu,s}^{1/s}(\PP^{[n]}_{a^*})/D_{\mu,s}^{1/s}(\PP^{[n]}_1)$ (right column) as functions of $d$ for $s=1$ (top row) and $s=10$ (bottom row) and different values of $n$: $n=10$ ({\color{Cerulean} $\bullet$}), $n=10^2$ ({\color{Dandelion} $\blacklozenge$}), $n=10^3$ ({\color{green} $\blacktriangledown$}), $n=10^4$ ({\color{Orchid} $\bigstar$}), and $n=10^5$ ({\color{gray} {\tiny $\blacksquare$}}).}
\label{F:sphere-supp}
\end{figure}

\subsection{Mixture of $\PP=\PP_a$ uniform on $\SS_{d-1}(a)$ and $\delta_{\0b_d}$}

The next figure (Figure~\ref{F:sphere2}, where $s=10$, $d=10$) shows that adding a delta measure at $\0b_d$ to $\PP_1$ may also yield a decrease of the expected distortion for small~$n$. Denote $\QQ_{\ma,a}=\ma\,\PP_a+(1-\ma)\,\delta_{\0b_d}$. Similarly to (\ref{Ds-sphere}*), we obtain
\bea
D_{\mu,s}(\QQ^{[n]}_{\ma,a})  &=&  (1-a)^s + s \int_{1-a}^1 t^{s-1} \left[1-\ma\, I_\upsilon(\delta,\delta)\right]^n\,\dd t \\
&& + s\,\ma^n \int_1^{1+a} t^{s-1} \left[1- I_\upsilon(\delta,\delta)\right]^n\,\dd t \,,
\eea
where $\upsilon=[t^2 -(1-a)^2]/(4 a)$, $\delta=(d-1)/2$ and $I_t(\cdot,\cdot)$ is the regularised incomplete beta function with the convention (\ref{convention}*).
The left panel presents the optimal $\ma^*$ minimising $D_{\mu,s}(\QQ^{[n]}_{\ma,1})$ as a function of $n$: unsurprisingly, $\ma^*$ quickly approaches 1 as $n$ increases. One may observe that, for all $n$, $(1-\ma^*)>1/n$.
Next, for a given $n$ ($n=20$), we use the corresponding optimal $\ma^*$ minimising $D_{\mu,s}(\QQ^{[n]}_{\ma,1})$ ($\ma^*\simeq 0.846$) and plot $D_{\mu,s}^{1/s}(\QQ^{[n]}_{\ma^*,a})$ and $D_{\mu,s}^{1/s}(\QQ^{[n]}_{1,a})=D_{\mu,s}^{1/s}(\PP^{[n]}_a)$ as functions of $a$ (central panel): the effect of introducing a delta measure at $\0b_d$ is significant for $a$ close to one, but choosing a suitable radius $a$ has a larger impact. Finally, we consider the efficiencies of random quantisers with $\QQ_{\ma^*,1}$ with respect to quantisers distributed with $\PP_1$ and
$\PP_{a^*}$ (with $\ma^*$ and $a^*$ depending on $n$) as functions of $n$ (right panel): when $a$ is set to 1, the choice of a suitable $\ma$ has some positive effect for small $n$, but the impact of choosing a suitable $a$ is stronger for all $n$. In Section~\ref{S:NSC-examples-sphere}*, we show numerically that for $n$ larger than some $n_*(d,s)$, $\PP_{a^*}$ uniform on a single sphere is optimal among all spherically symmetric distributions (see Figure~\ref{F:nmin_sphere}*).

\begin{figure}[ht!]
\begin{center}
\includegraphics[width=.32\linewidth]{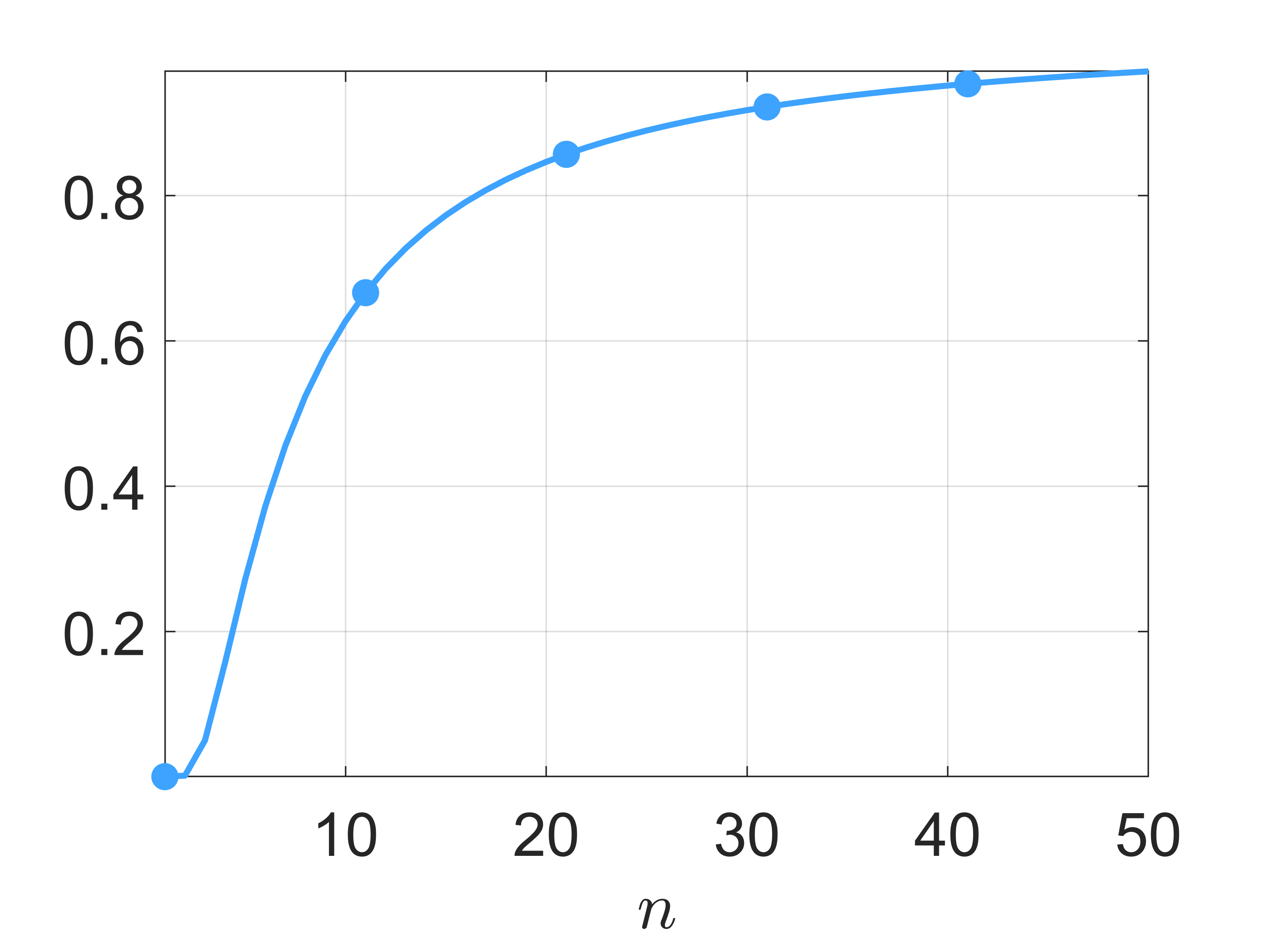}
\includegraphics[width=.32\linewidth]{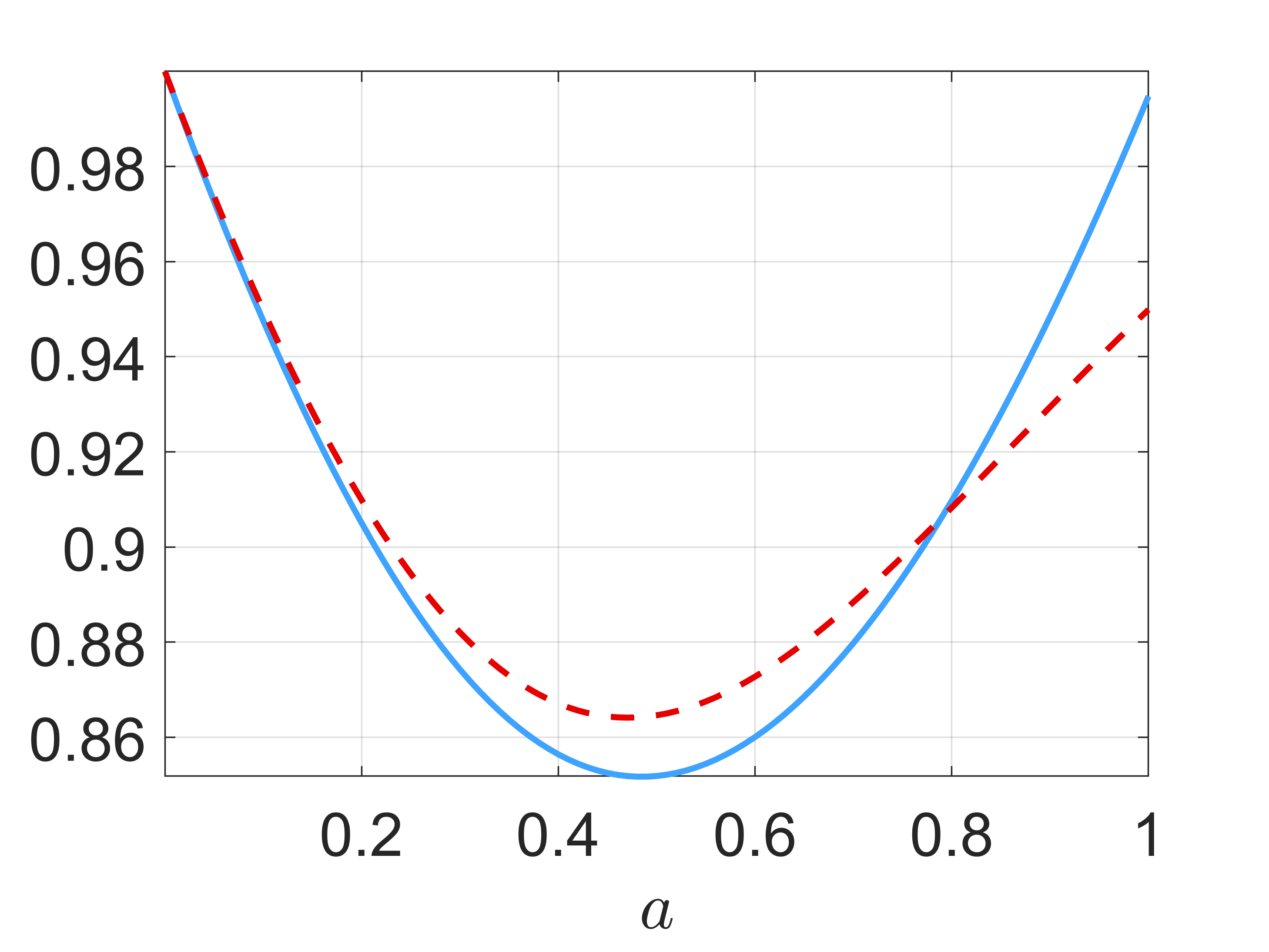}
\includegraphics[width=.32\linewidth]{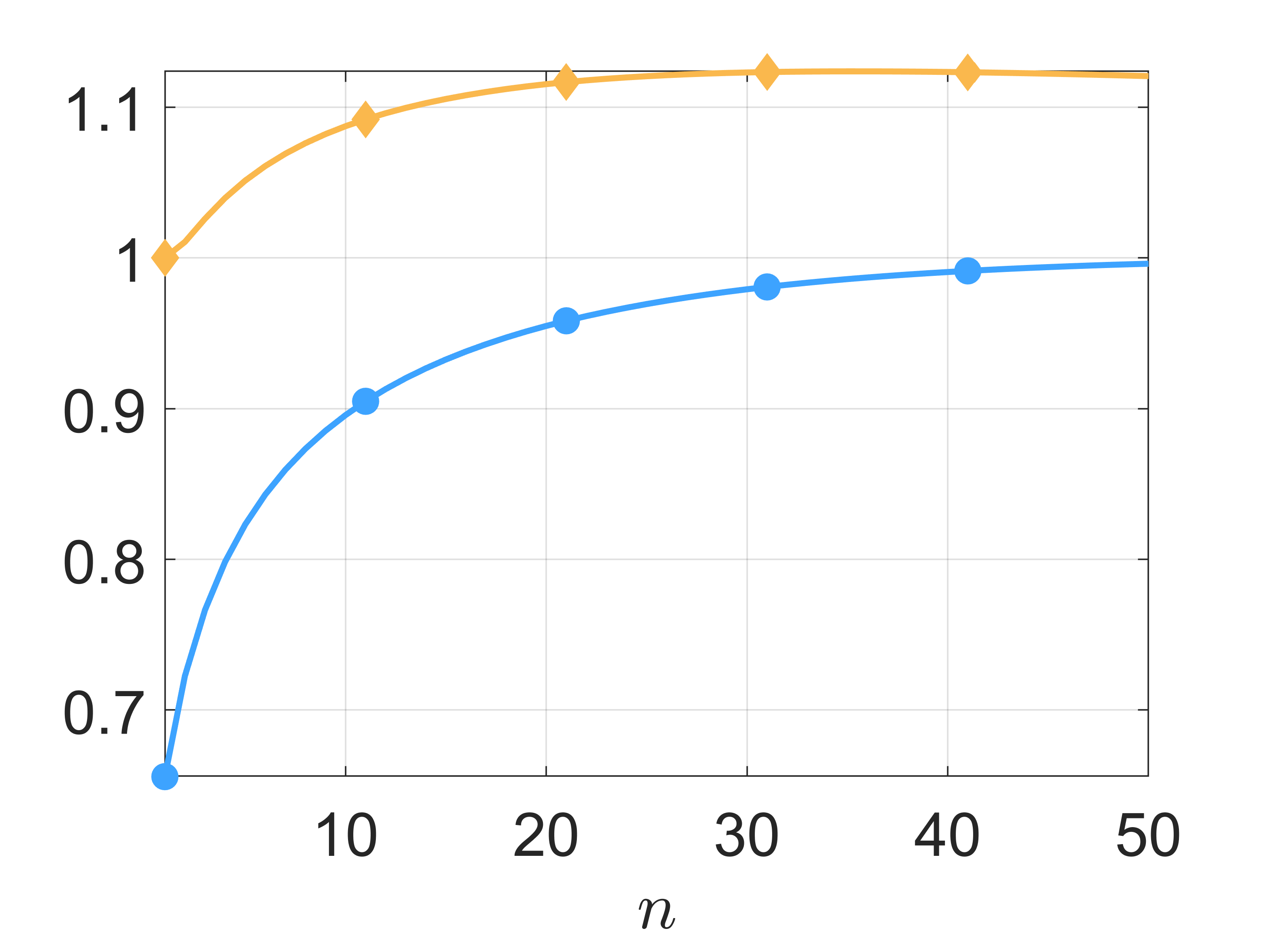} \\
\end{center}
\caption{\small Left: value $\ma^*$ minimising $D_{\mu,s}(\QQ^{[n]}_{\ma,1})$ as a function of $n$. Centre: $D_{\mu,s}^{1/s}(\QQ^{[n]}_{\ma^*,a})$ ({\color{red} {\bf - - -}}) and $D_{\mu,s}^{1/s}(\QQ^{[n]}_{1,a})$ ({\color{Cerulean} {\bf ---}}) as functions of $a$ for $n=20$ and $\ma^* = 0.846$. Right: $D_{\mu,s}^{1/s}(\QQ^{[n]}_{\ma^*,1})/D_{\mu,s}^{1/s}(\PP^{[n]}_1)$ ({\color{Cerulean} $\bullet$}) and $D_{\mu,s}^{1/s}(\QQ^{[n]}_{\ma^*,1})/D_{\mu,s}^{1/s}(\PP^{[n]}_{a^*})$ ({\color{Dandelion} $\blacklozenge$}) as functions of $n$; $s=10$, $d=10$.}
\label{F:sphere2}
\end{figure}

\subsection{Comparison between random quantisers uniform on a sphere and optimised full factorial $2^d$ designs} \label{S:sphere-2^d}
Let $n=2^d$ and $\Xb_n(b)$ be the full factorial $2^d$ design with points $\xb_i$ having coordinates $\pm b$, $b\in(0,1]$: the $\xb_i$ are vertices of the cube $[-b,b]^d$; they also belong to the sphere $\SS_{d-1}(a)$ with $a=\sqrt{d}\,b$. For $\SX=\SS_{d-1}(1)$, the Voronoi regions $\SV(\xb_i)=\{\xb\in\SX: \|\xb-\xb_i\| = \min_{\xb_j\in\Xb_n(b)} \|\xb-\xb_j\|\}$ are all identical up to a permutation of coordinates and do not depend on the value of $b$.

Consider first the case $s=2$. For $\mu$ uniform on $\SX=\SS_{d-1}(1)$, $D_{\mu,2}[\Xb_n(b)]=\Ex\{\|V-\bb\|^2\}$ where $\bb=(b,\ldots,b)\in\RR^d$ and $V=(V_1,\ldots,V_d)$ is uniform in $\SV(\bb)=\SS_{d-1}(1)\cap \RR_+^d$. Therefore
\bea
D_{\mu,2}[\Xb_n(b)]= \Ex\{ \|V\|^2 - 2\,b\sum_{i=1}^d  V_i + d b^2\} = 1 - 2\,bd\, \Ex\{V_1\} + db^2 \,.
\eea
Minimisation of $D_{\mu,2}[\Xb_n(b)]$ with respect to $b$ yields $b_2^*=\Ex\{V_1\}$ and the optimal $(\mu,2)$-distortion of a $2^d$ factorial design is $D_{\mu,2}[\Xb_n(b_2^*)]=1-d (b_2^*)^2$.
For $U^{(d)}=(U_1,\ldots,U_d)\simd \mu$, the first component $U_1$ has the density $\varphi_1(\cdot)$ of Lemma~\ref{L:projections-beta}* on $[-1,1]$, therefore, $V_1$ has the density $2\,\varphi_1(\cdot)$ on $[0,1]$, so that
\bea
b_2^* =\Ex\{V_1\} &=& \frac{\Gamma(d/2)}{\sqrt{\pi}\,\Gamma((d+1)/2)}\,, \\
D_{\mu,2}[\Xb_n(b_2^*)] &=& 1 - \frac{d}{\pi} \left(\frac{\Gamma(d/2)}{\Gamma((d+1)/2)}\right)^2 \,.
\eea

More generally, the analytic expression for the $(\mu,s)$-distortion $D_{\mu,s}[\Xb_n(b)]$ can be derived for any even $s$, but the calculations become more complicated as $s$ increases. For $s=4$ we obtain
\bea
D_{\mu,4}[\Xb_n(b)] &=& \Ex\{\|V-\bb\|^4\} \\
&=& 1 +4\,b^2\, \Ex\left\{\left(\sum_{i=1}^d V_i \right)^2\right\} + d^2 b^4 + 2\,d b^2 - 4\, bd\, \Ex\{V_1\} (1+d b^2) \,.
\eea
Using $\Ex\{(\sum_{i=1}^d V_i)^2\}=1+d(d-1) \Ex\{V_1V_2\}$, where $\Ex\{V_1V_2\}=2/(\pi d)$ from $\varphi_2(\cdot,\cdot)$ of Lemma~\ref{L:projections-beta}*, together with the expression above for $\Ex\{V_1\}$, we can express $D_{\mu,4}[\Xb_n(b)]$ as a fourth-degree polynomial in $b$ with coefficients depending on $d$. The optimal value $b_4^*$ can be obtained explicitly.

Denote by $a_s^*=\sqrt{d}b_s^*$ the radius of the sphere on which the optimised full factorial design $\Xb_n(b_s^*)$ lies. The values of $a_2^*$ and $a_4^*$ are very close for large $d$, since $a_j^*=\sqrt{d}\,b_j^*=\sqrt{2/\pi}+\SO(1/d)$ for $j=2,4$.

For $s=\infty$, $E_{\mu,\infty}[\Xb_n(b)]=\CR[\Xb_n(b)]$, the covering radius of $\Xb_n(b)$, with $(\CR[\Xb_n(b)])^2=1+d\,b^2-2\,b$ being minimised at $b_\infty^*=1/d$, which gives
$\CR[\Xb_n(b_\infty^*)]=\sqrt{1-1/d}$.

Figure~\ref{F:sphere3} compares the optimal full factorial designs $\Xb_n(b_s^*)$ for $s=2$ and 4 with optimised random quantisers $\PP_{a^*}^{[n]}$ on $\SS_{d-1}(1)$ with the same sample size $n=2^d$, with $a^*=a^*(d,n,s)$; see Section~\ref{S:uniform-on-S}. The left panel shows $a_s^*$ and
$a^*$ as functions of $d$. The right panel presents $D_{\mu,s}^{1/s}(\PP^{[n]}_{a^*})$ for the random quantisers and $E_{\mu,s}[\Xb_n(b_s^*)]$ for the full factorial designs, for $s=2$ and 4. For both types of point sets, the behaviours for $s=2$ and $s=4$ are similar. As $d$ tends to infinity, for the full factorial designs $a_s^*=\sqrt{d}\,b_s^*$ tends to the limit $\sqrt{2/\pi}\simeq 0.797885$ from above. For random quantisers, $a^*(s)=a^*(d,n,s)$ tends (from below) to the larger limiting value $\sqrt{3}/2 \simeq 0.8660$; see Theorem~\ref{Th:joint-exact-distortion-limit}*.
The expected $(\mu,s)$-distortion of random quantisers shows little sensitivity to the choice of $a$ in the neighbourhood of $a^*(s)$.
In particular, the plots of $D_{\mu,s}^{1/s}(\PP^{[n]}_{a_s^*})$ (not shown) and $D_{\mu,s}^{1/s}(\PP^{[n]}_{a^*(s)})$ almost coincide for $d>3$.

For both values of $s$, $a_s^* \approx a^*(s)$ for $d=7$. In the numerical comparisons shown here, this is the dimension at which random quantisers start exhibiting a smaller quantisation error than full factorial designs. As a consequence of the probabilistic method, the strict inequality between the mean random-quantiser distortion and the factorial-design distortion implies that for every dimension $d\geq 7$ there exist non-random quantisers with $n=2^d$ that have smaller quantisation errors for $s=2$ and 4 than optimised full factorial designs. The numerical optimality certificates in Section~\ref{S:NSC-examples-sphere}* further indicate that $\PP_{a^*}$ is optimal among i.i.d.\ random-quantiser distributions in the cases considered. The concentration result of Theorem~\ref{Th:distortion-concentration}* explains why typical realisations are close to their mean performance, but it is not needed for the probabilistic-method existence argument.

\begin{figure}[ht!]
\begin{center}
\includegraphics[width=.49\linewidth]{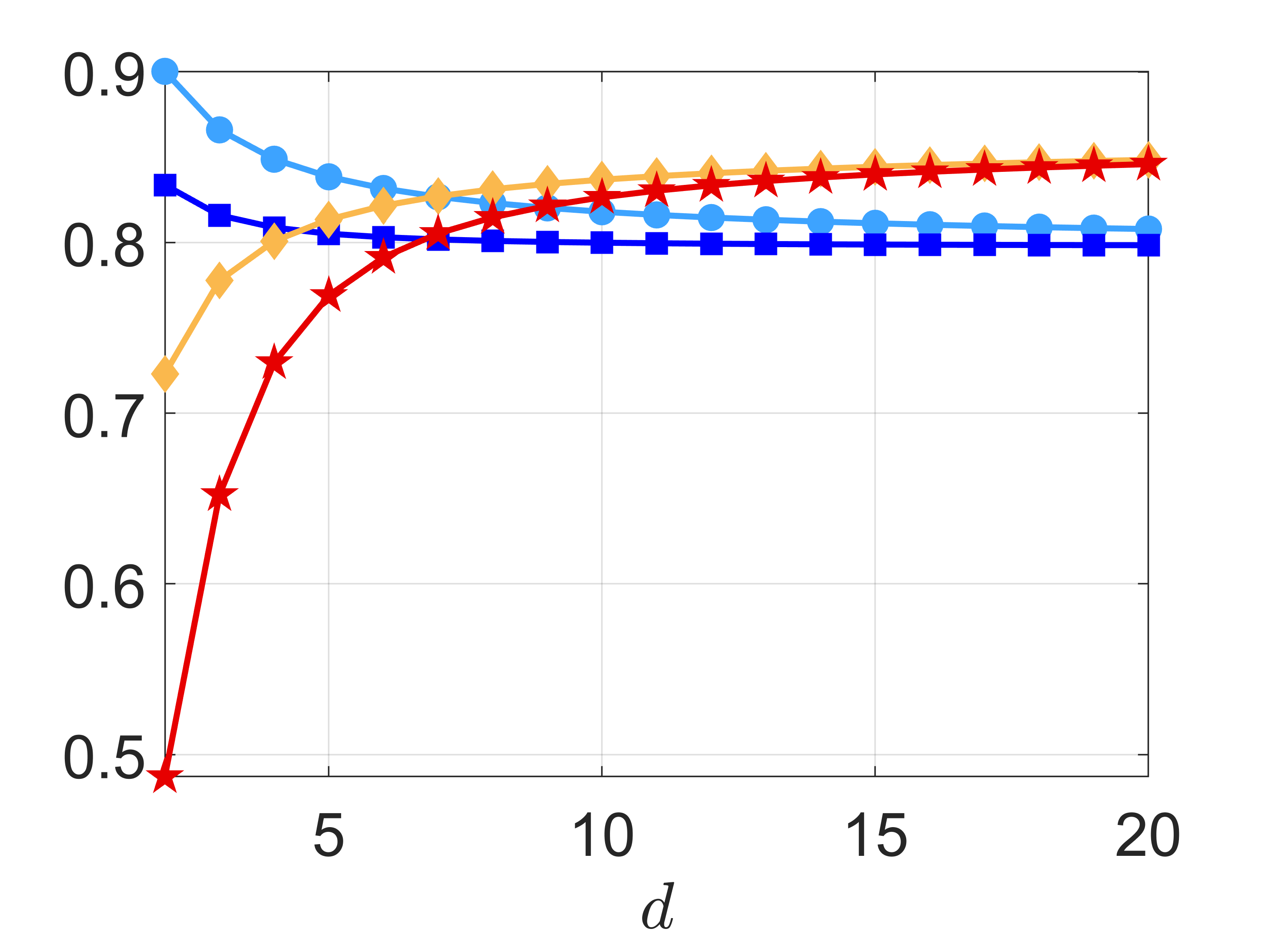}
\includegraphics[width=.49\linewidth]{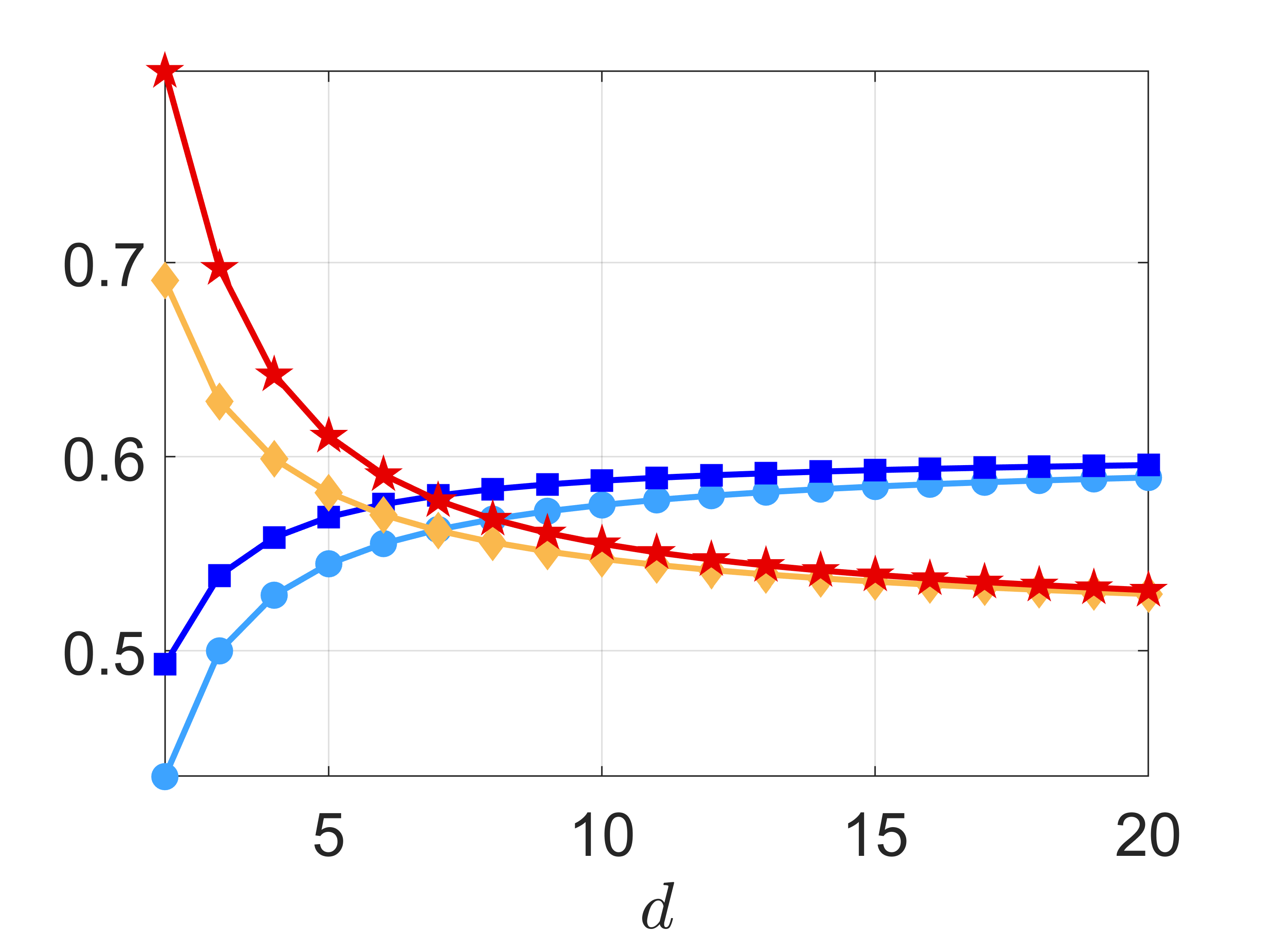} \\
\end{center}
\caption{\small Left: Values $a^*(s)$ minimising $D_{\mu,s}(\PP^{[n]}_a)$ for $n=2^d$ with $s=2$ ({\color{Dandelion} $\blacklozenge$}) and $s=4$ ({\color{red} $\bigstar$}), and values $a_s^*=\sqrt{d}\,b_s^*$ minimising $E_{\mu,s}[\Xb_n(b)]$ for $s=2$ ({\color{Cerulean} $\bullet$}) and $s=4$ ({\color{blue} {\tiny $\blacksquare$}}), as functions of $d$.
Right: $D_{\mu,s}^{1/s}(\PP^{[n]}_{a^*})$ for $n=2^d$ with $a^*=a^*(d,n,s)$, $s=2$ ({\color{Dandelion} $\blacklozenge$}) and $s=4$ ({\color{red} $\bigstar$}), and values $E_{\mu,s}[\Xb_n(b_s^*)]$ for $s=2$ ({\color{Cerulean} $\bullet$}) and $s=4$ ({\color{blue} {\tiny $\blacksquare$}}), as functions of $d$.}
\label{F:sphere3}
\end{figure}


\section{$\mu$ is uniform in the unit ball $\SB_d(1)$} 
\label{S:ball}

Here $\mu$ is uniform in $\SB_d(1)$, so that the density of $r=\|U\|$ for $U\simd\mu$ is $\psi(\tau)=d\,\tau^{d-1}$; see Section~\ref{S:NSC-examples-ball}*.

When the random quantiser has the distribution $\PP=\PP_a$ uniform on $\SS_{d-1}(a)$, $D_{\mu,s}(\PP_a^{[n]})$ is given by (\ref{DS-Pa}*).
When $\PP=\PP_{0,b}$ uniform in $\SB_d(b)$, the c.d.f.\ $\Phi(\cdot)$ of $R=\|X\|$ has the density $\phi_{0,b}(\rho)=d\,\rho^{d-1}/b^d$ for $\rho\in[0,b]$.
More general spherically symmetric distributions $\PP$ could also be used. Here, the class $\PP_{0,b}$ is considered as a simple one-parameter benchmark; it does not contain the globally optimal radial distribution, which consists of a mixture of uniform distributions on concentric spheres (see Section~\ref{S:NSC-examples-ball}*).
The expected $(\mu,s)$-distortion $D_{\mu,s}(\PP_{0,b}^{[n]})$ is given by (\ref{Ls-mean-final}*)
where
\bea
H(r,t;\Phi,n) =\bigg(1- \int_0^b I_\upsilon(\delta,\delta) \,\phi_{0,b}(\rho)\,\dd\rho \bigg)^n \,,
\eea
with $\upsilon=\upsilon(t,\rho,r)$ given by (\ref{v}*).

Figure~\ref{F:ball} presents the same information as Figure~\ref{F:sphere-supp} for the ball configuration: we consider the distribution $\PP_{0,b^*}$, which is optimal within the one-parameter class $\{\PP_{0,b}:b\geq0\}$, where $b^*=b^*(d,n,s)$ minimises $D_{\mu,s}(\PP_{0,b}^{[n]})$, and the right column compares the performances of $\PP_{0,b^*}$ and $\PP_{0,1}$.

\begin{figure}[ht!]
\begin{center}
\includegraphics[width=.32\linewidth]{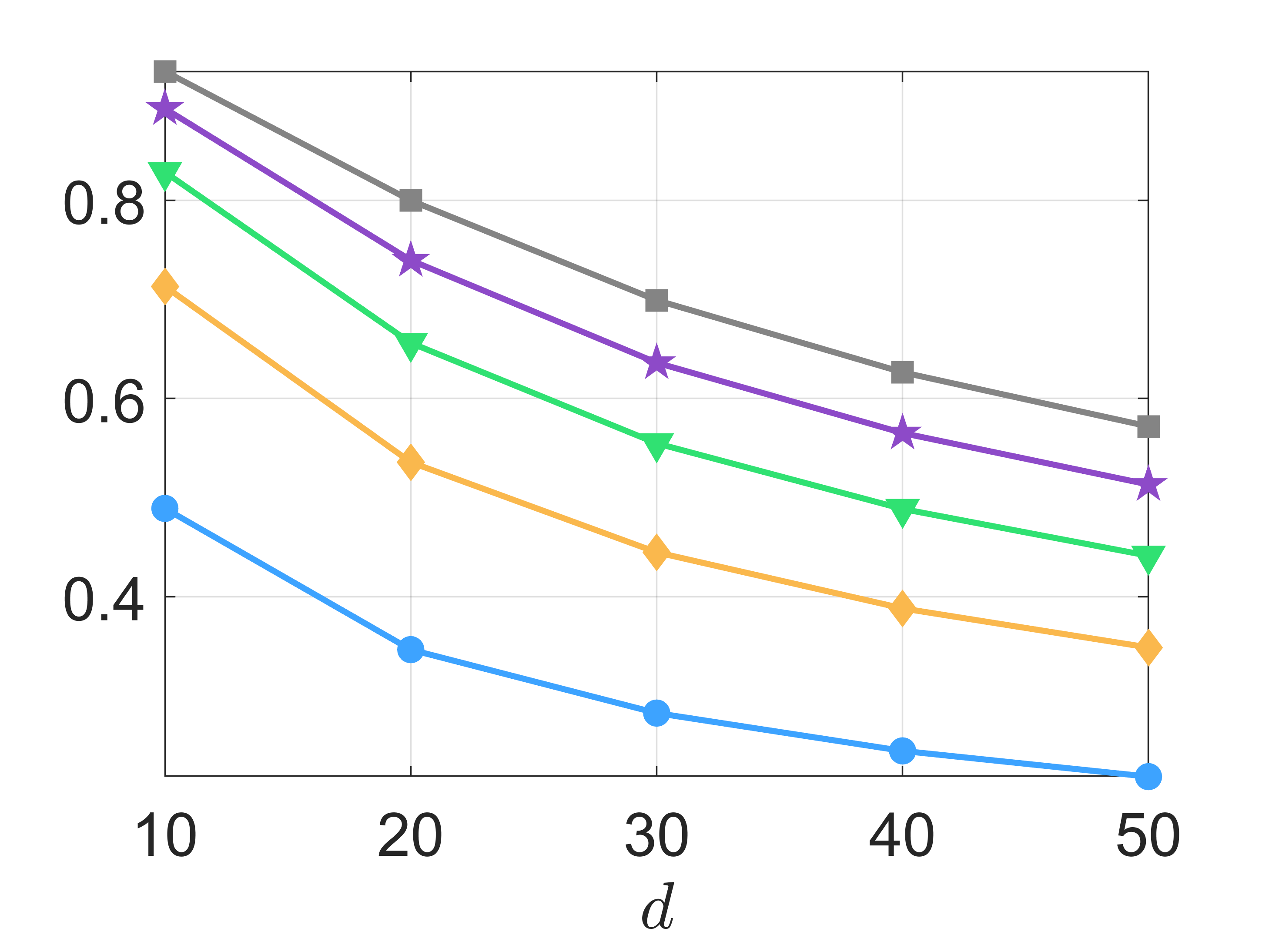}
\includegraphics[width=.32\linewidth]{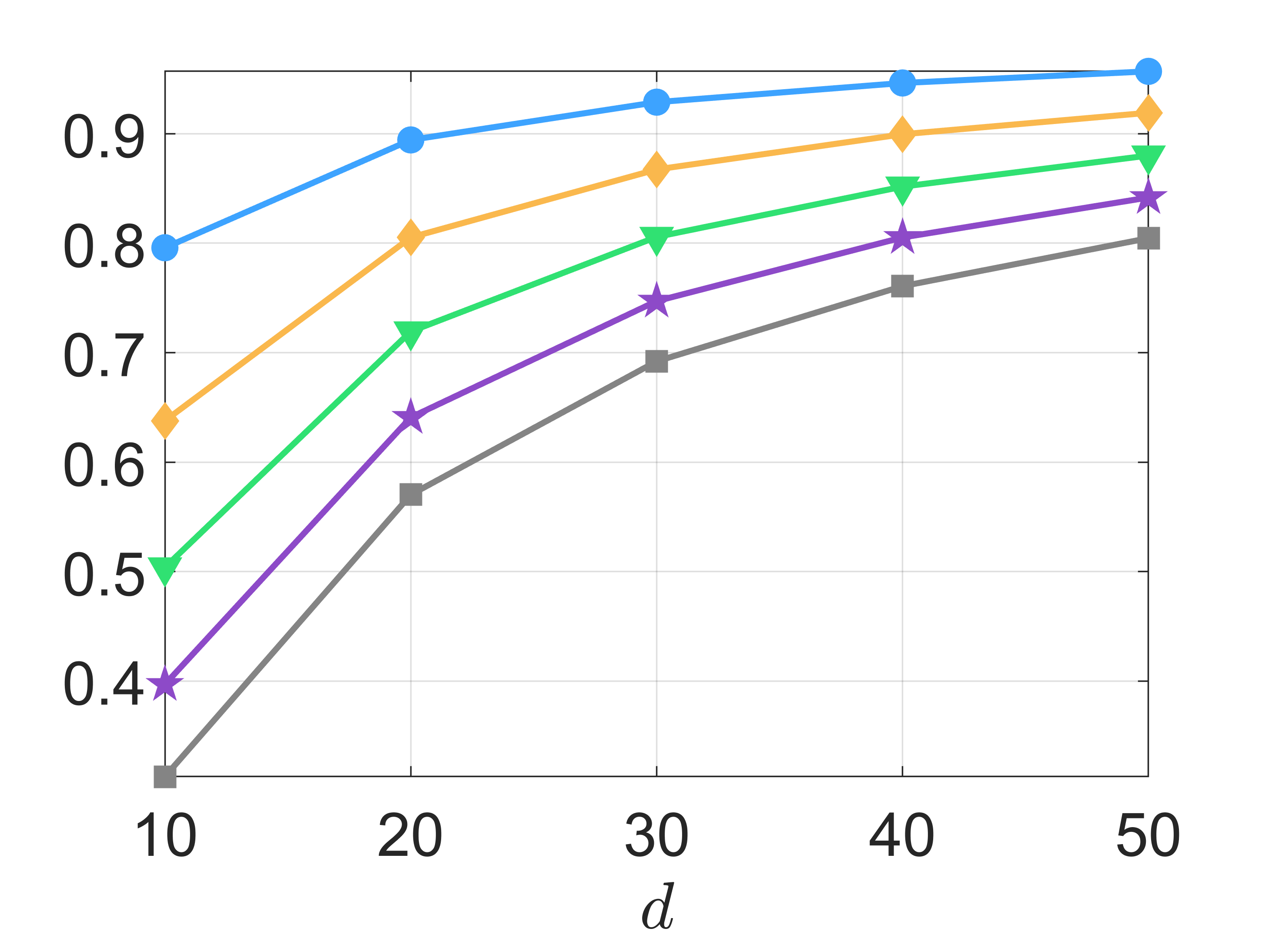}
\includegraphics[width=.32\linewidth]{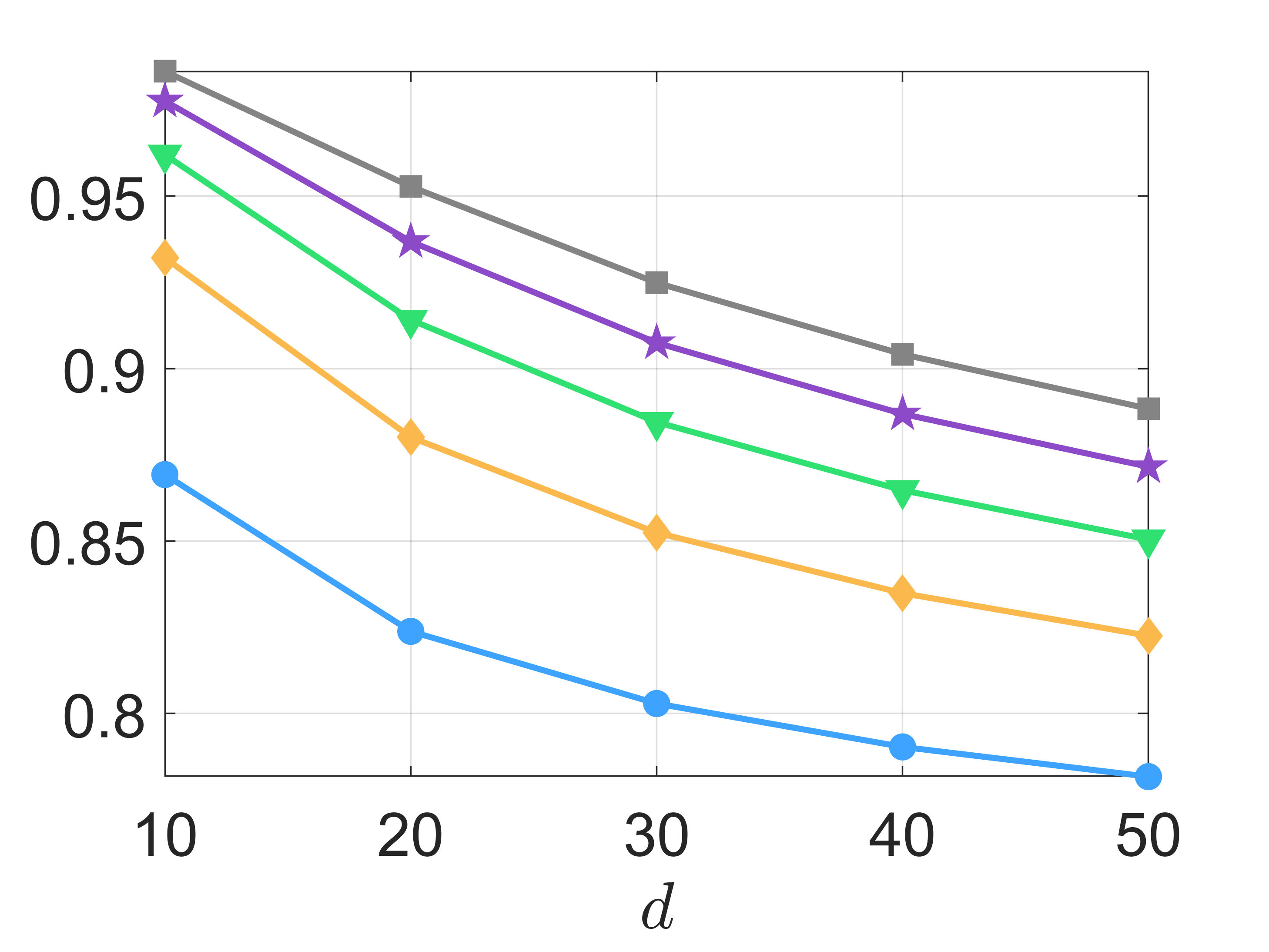} \\
\includegraphics[width=.32\linewidth]{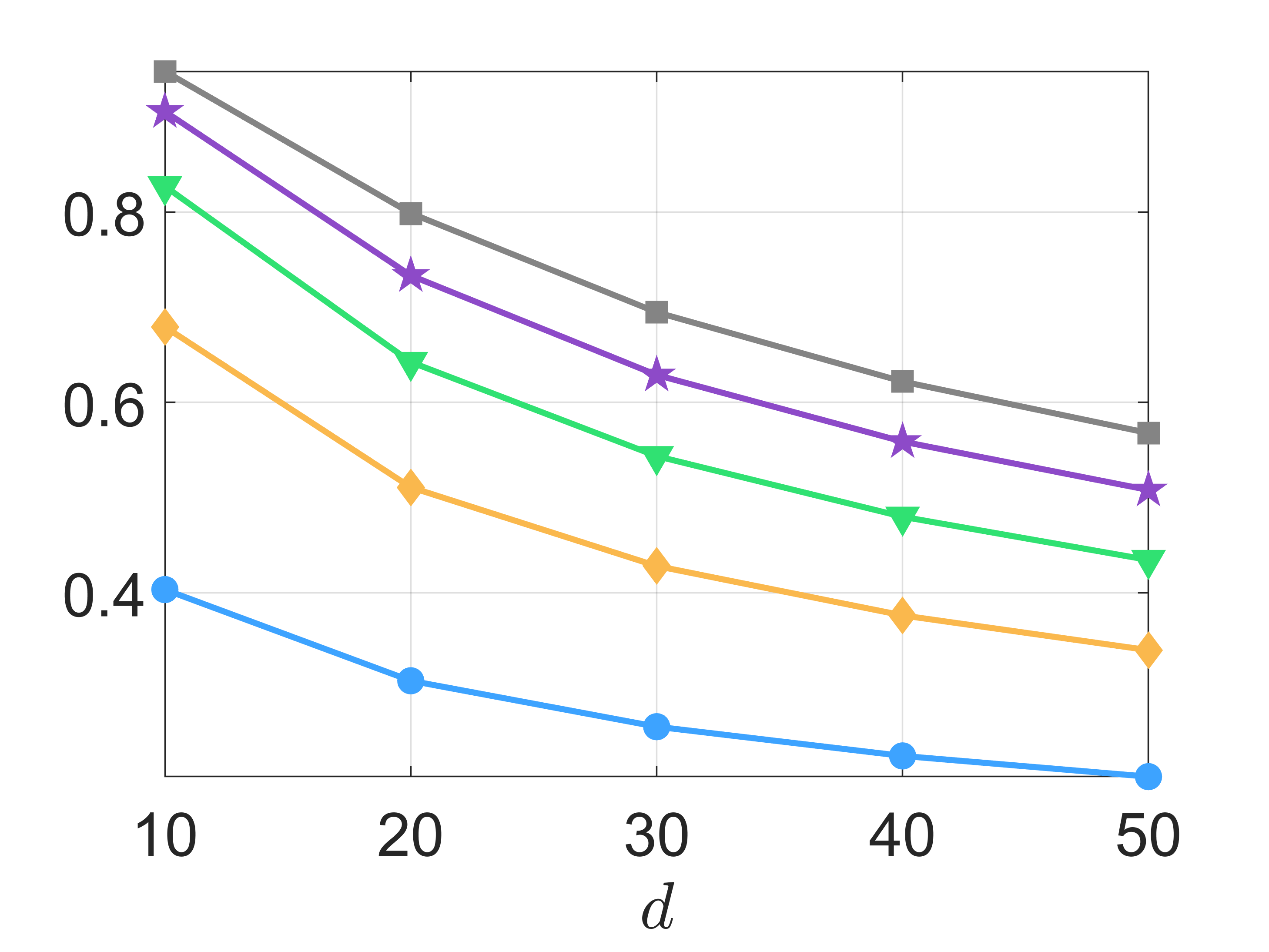}
\includegraphics[width=.32\linewidth]{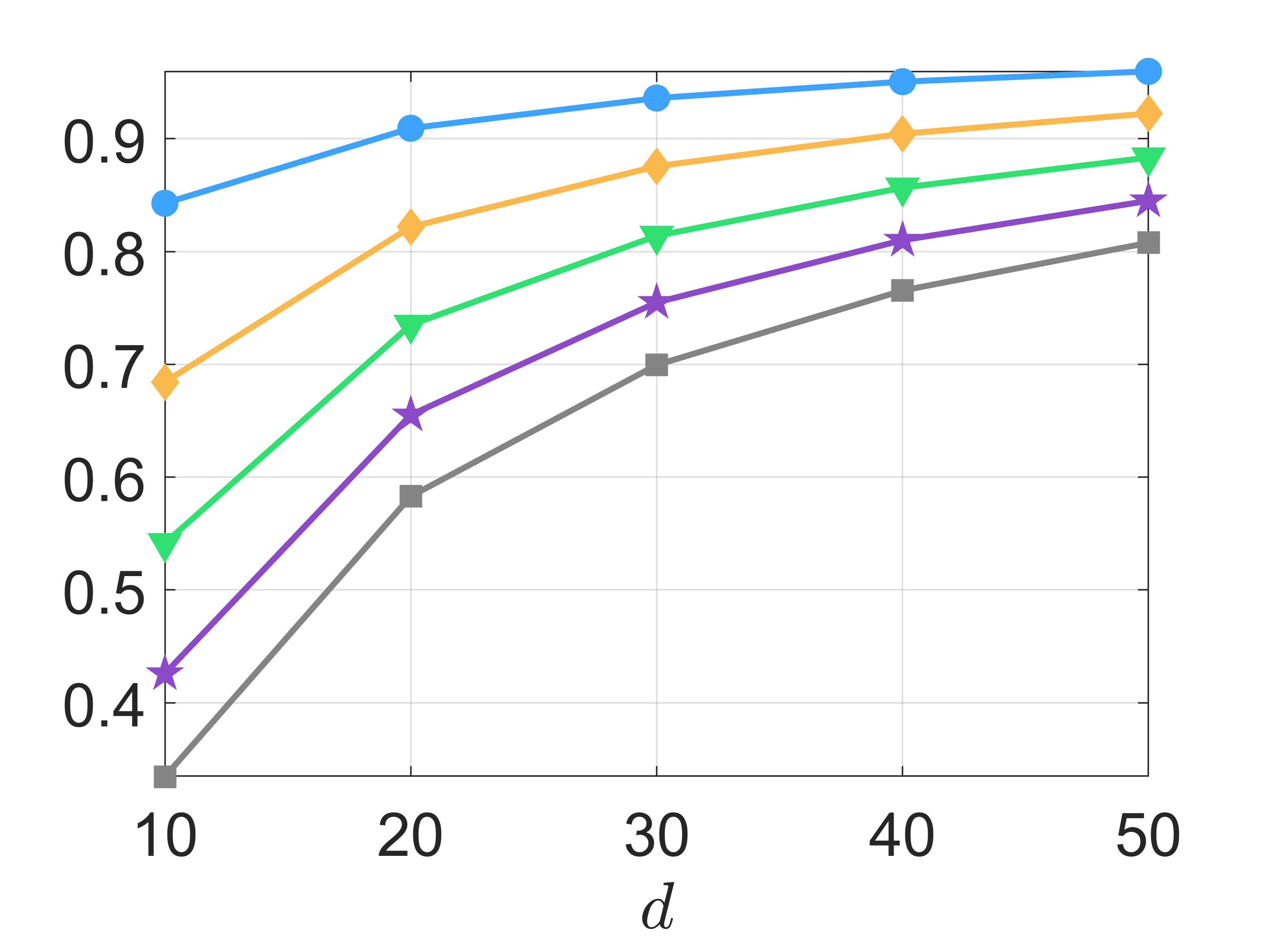}
\includegraphics[width=.32\linewidth]{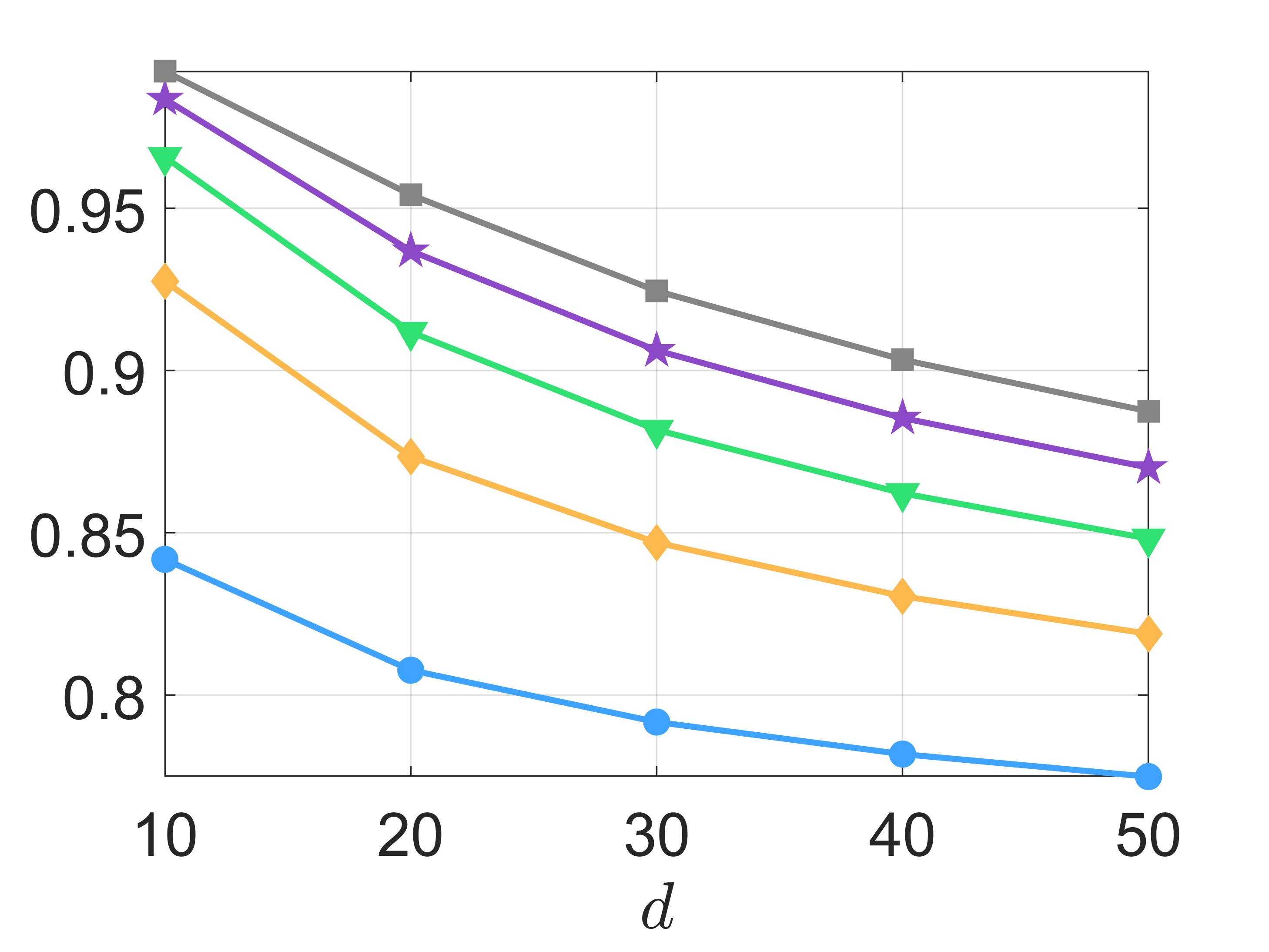}
\end{center}
\caption{\small Ball. Value $b^*$ minimising $D_{\mu,s}(\PP^{[n]}_{0,b})$ (left column), $D_{\mu,s}^{1/s}(\PP^{[n]}_{0,b^*})$ (central column) and ratio $D_{\mu,s}^{1/s}(\PP^{[n]}_{0,b^*})/D_{\mu,s}^{1/s}(\PP^{[n]}_{0,1})$ (right column) as functions of $d$ for $s=1$ (top row) and $s=10$ (bottom row) and different values of $n$: $n=10$ ({\color{Cerulean} $\bullet$}), $n=10^2$ ({\color{Dandelion} $\blacklozenge$}), $n=10^3$ ({\color{green} $\blacktriangledown$}), $n=10^4$ ({\color{Orchid} $\bigstar$}), and $n=10^5$ ({\color{gray} {\tiny $\blacksquare$}}).}
\label{F:ball}
\end{figure}

The quantities considered behave qualitatively as in Figure~\ref{F:sphere-supp}.  Note, however, that $b^*$ is larger than $a^*$ plotted in Figure~\ref{F:sphere-supp} (left columns; in particular, on the bottom row with $s=10$, $b^*$ is much larger than $a^*$), that the displayed values of $D_{\mu,s}(\PP_{0,b^*}^{[n]})$ for the ball target are smaller than the corresponding displayed values of $D_{\mu,s}(\PP_{a^*}^{[n]})$ for the sphere target (central columns), and that the efficiency of a ``classical'' (i.e., asymptotically optimal) random design compared to an optimised design is slightly larger for the ball than for the sphere (right columns). These distortion values correspond to different target measures and therefore should not be interpreted as a direct performance comparison between the two quantisers. 
For fixed $d$, the empirical measure of an optimal quantiser minimising $E_{\mu,s}(\Xb_n)$ converges weakly to the uniform probability measure $\PP_{0,1}$ as $n\to\infty$ (see \cite[Table~7.1]{GrafL2000}). Consistent with this result, we empirically observe that the optimal parameters $b^*$ and $a^*$ both converge to 1.

Figure~\ref{F:bstar_ALLastar_ball_s2_D}
displays the optimal values $b^*(d,n,s)$ and $a^*(d,n,s)$, respectively minimising $D_{\mu,s}(\PP_{0,b}^{[n]})$ and $D_{\mu,s}(\PP_a^{[n]})$, plotted as functions of $d$ for $s=2$ and two distinct values of $n$. The figure shows that $b^*>a^*$, and the gap $b^*-a^*$ narrows as $d$ increases, reflecting the observed convergence of both $a^*$ and $b^*$ to zero as $d\to\infty$. The behaviour remains qualitatively unchanged for other values of $s$, especially when the dimension $d$ is large, in agreement with the results in Section~\ref{S:n-growing-with-d}*.

\begin{figure}[ht!]
\begin{center}
\includegraphics[width=.49\linewidth]{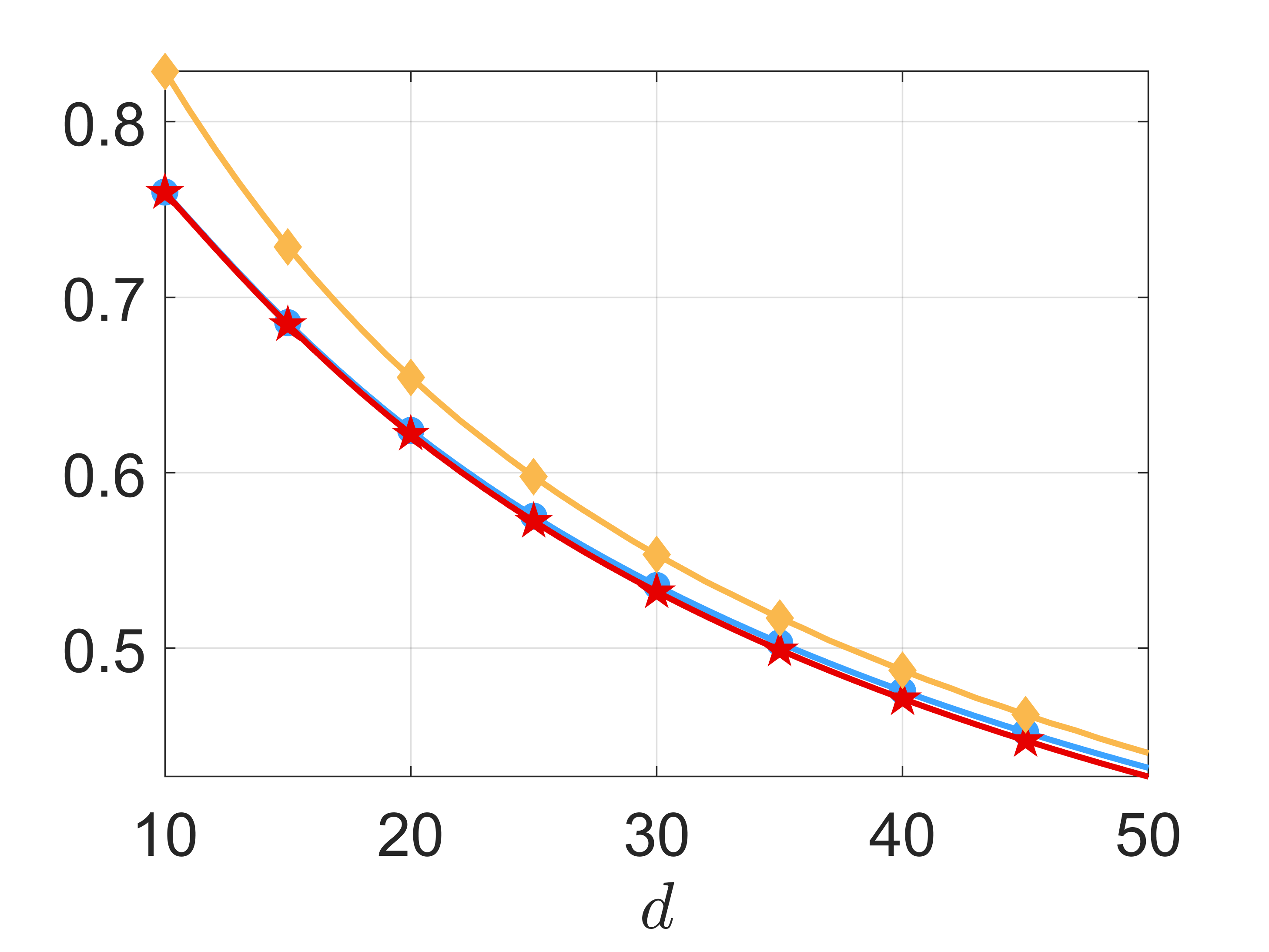}
\includegraphics[width=.49\linewidth]{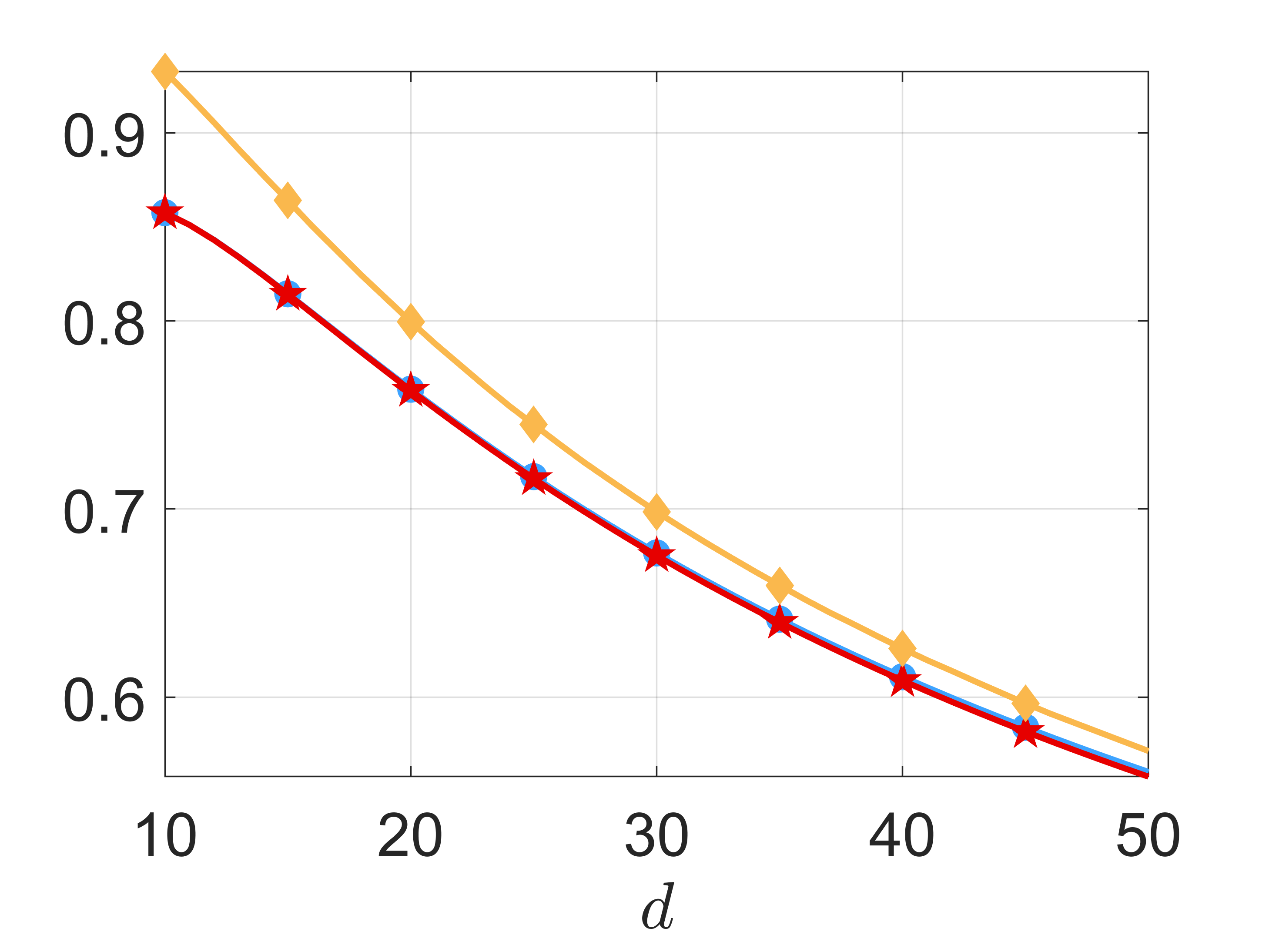}
\end{center}
\caption{\small $a^*(d,n,2)$ ({\color{Cerulean} $\bullet$}), $\widehat{a^*}(d,n,2)$ ({\color{red} $\bigstar$})
and $b^*(d,n,2)$ ({\color{Dandelion} $\blacklozenge$}) as functions of $d$ for $n=1\,000$ (left) and $n=100\,000$ (right).}
\label{F:bstar_ALLastar_ball_s2_D}
\end{figure}

Given the observations in Figures~\ref{F:ball} and 
\ref{F:bstar_ALLastar_ball_s2_D}, two key patterns emerge:
(\textit{i}) for large dimensions $d$, the values of $a^*(d,n,s)$ and $b^*(d,n,s)$ become very close; (\textit{ii}) the quantiser $\PP_{0,b^*}^{[n]}$ significantly outperforms $\PP_{0,1}^{[n]}$ when $d$ is sufficiently large. Moreover, we also observe that $\PP_{a^*}^{[n]}$ outperforms $\PP_{0,b^*}^{[n]}$ for reasonable design sizes $n$; see Figure~\ref{F:ALLeff_ball_s2_D}.

\begin{figure}[ht!]
\begin{center}
\includegraphics[width=.49\linewidth]{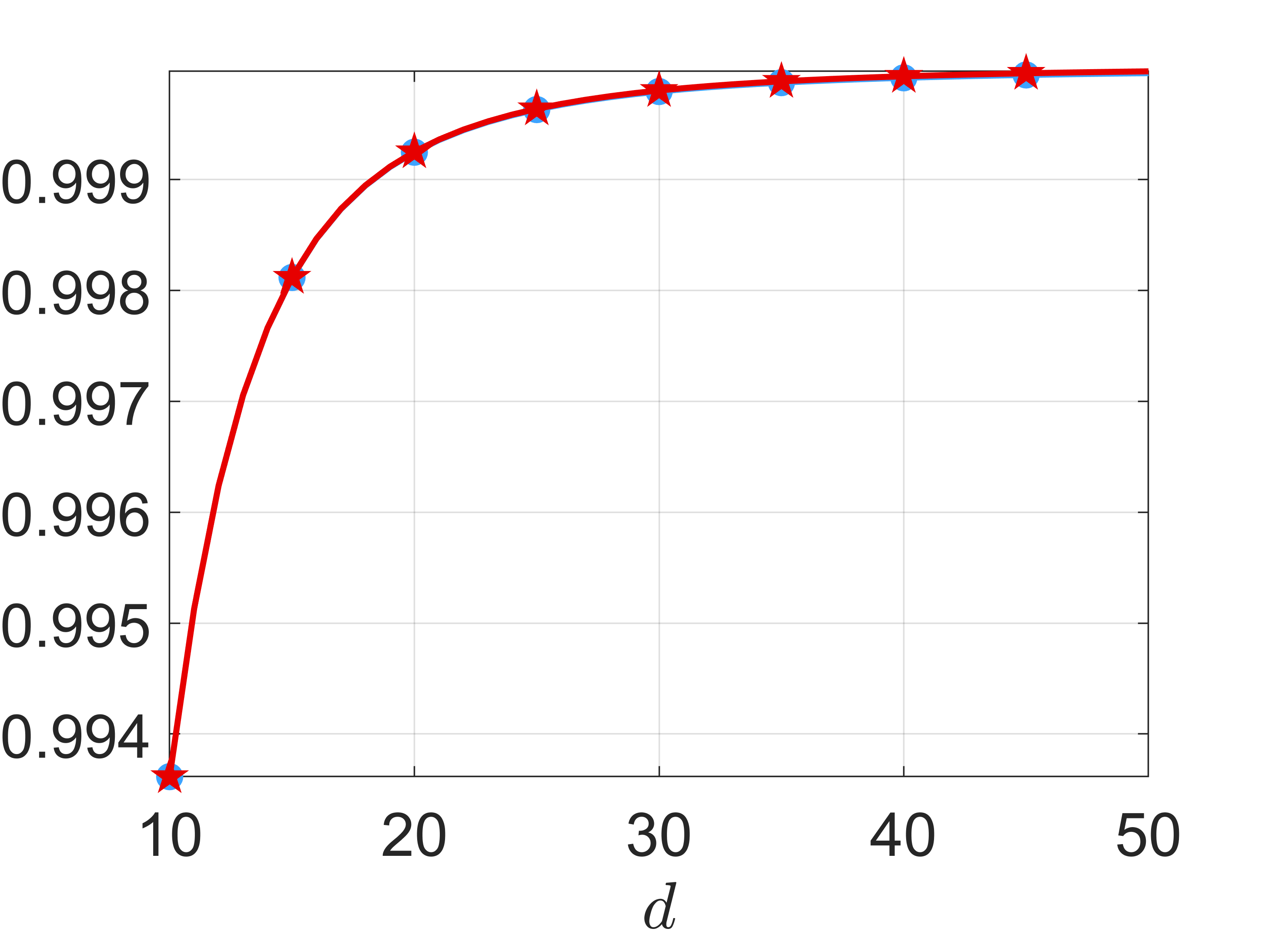}
\includegraphics[width=.49\linewidth]{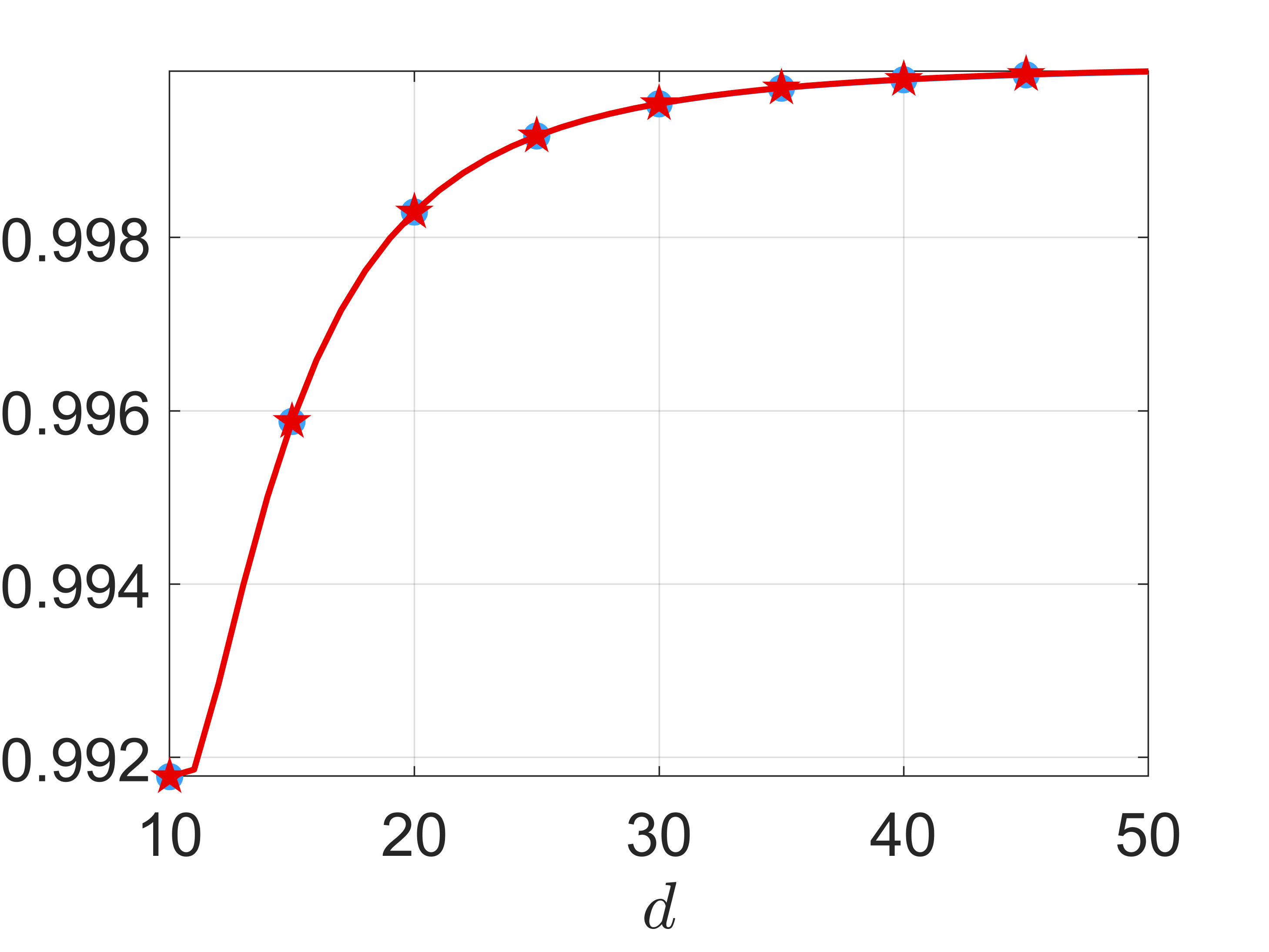}
\end{center}
\caption{\small Efficiencies $D_{\mu,s}^{1/s}(\PP_a^{[n]})/D_{\mu,s}^{1/s}(\PP_{0,b^*}^{[n]})$ as functions of $d$: $a=a^*(d,n,2)$ ({\color{Cerulean} $\bullet$}), $a=\widehat{a^*}(d,n,2)$ ({\color{red} $\bigstar$});
$n=1\,000$ (left) and $n=100\,000$ (right).}
\label{F:ALLeff_ball_s2_D}
\end{figure}

Since $b^*$ converges to 1 as $n \to \infty$, with $\PP_{0,1}^{[n]}$ being asymptotically optimal, a natural question arises: for a fixed dimension $d$, at what sample size $n=n_2(d,s)$ does $\PP_{a^*}^{[n]}$ stop being preferable to $\PP_{0,b^*}^{[n]}$? Empirical analysis reveals that this transition occurs only for extremely large values of $n$.
The left panel of Figure~\ref{F:nstar-ball_s2-4-10} presents $n_2(d,s)$ on a logarithmic scale (computed by bisection) as a function of $d=3,\ldots,20$ for three values of $s$. Over this range, the results suggest a super-exponential increase of $n_2(d,s)$ with $d$.

\begin{figure}[ht!]
\begin{center}
\includegraphics[width=.49\linewidth]{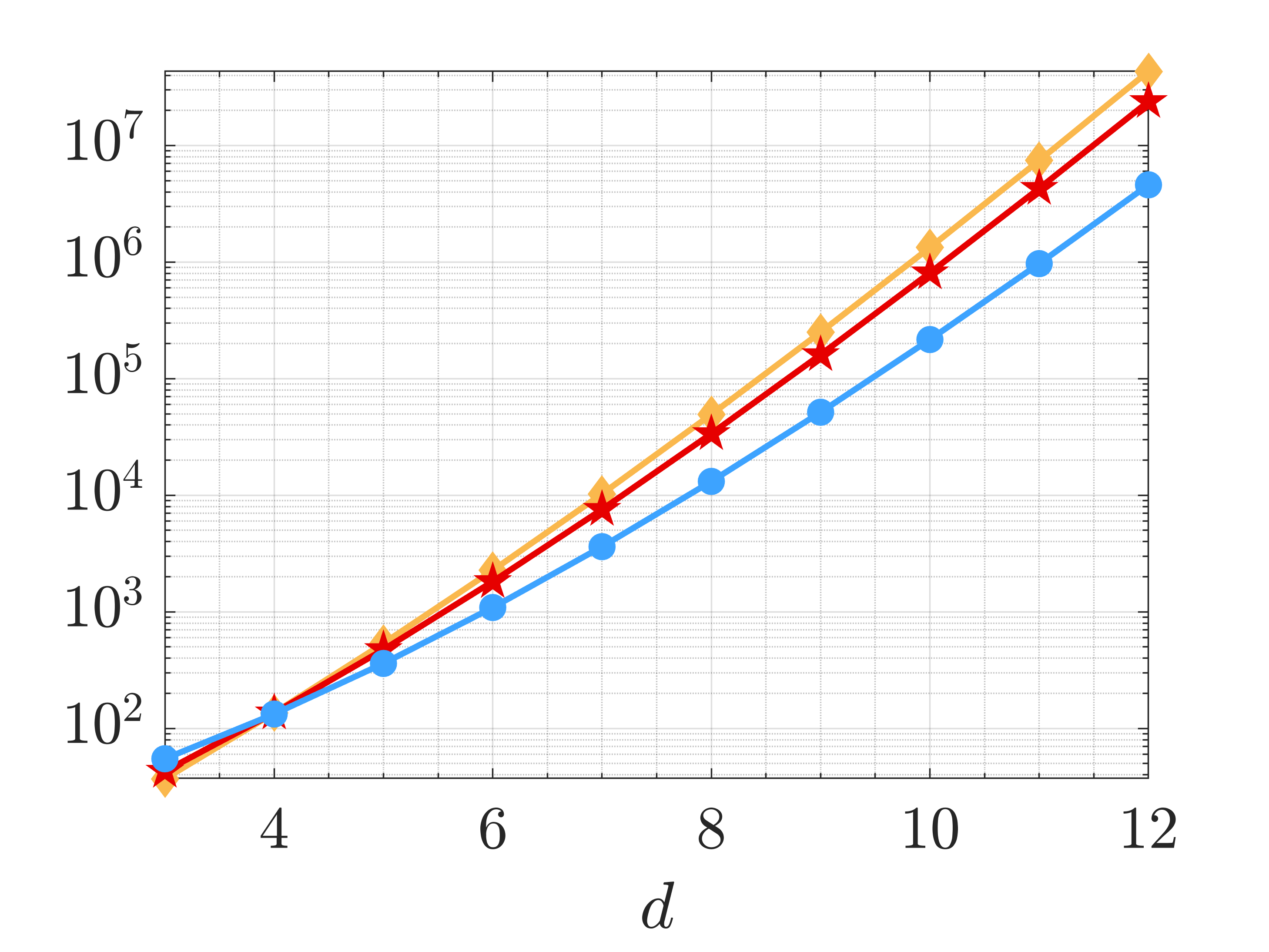}
\includegraphics[width=.49\linewidth]{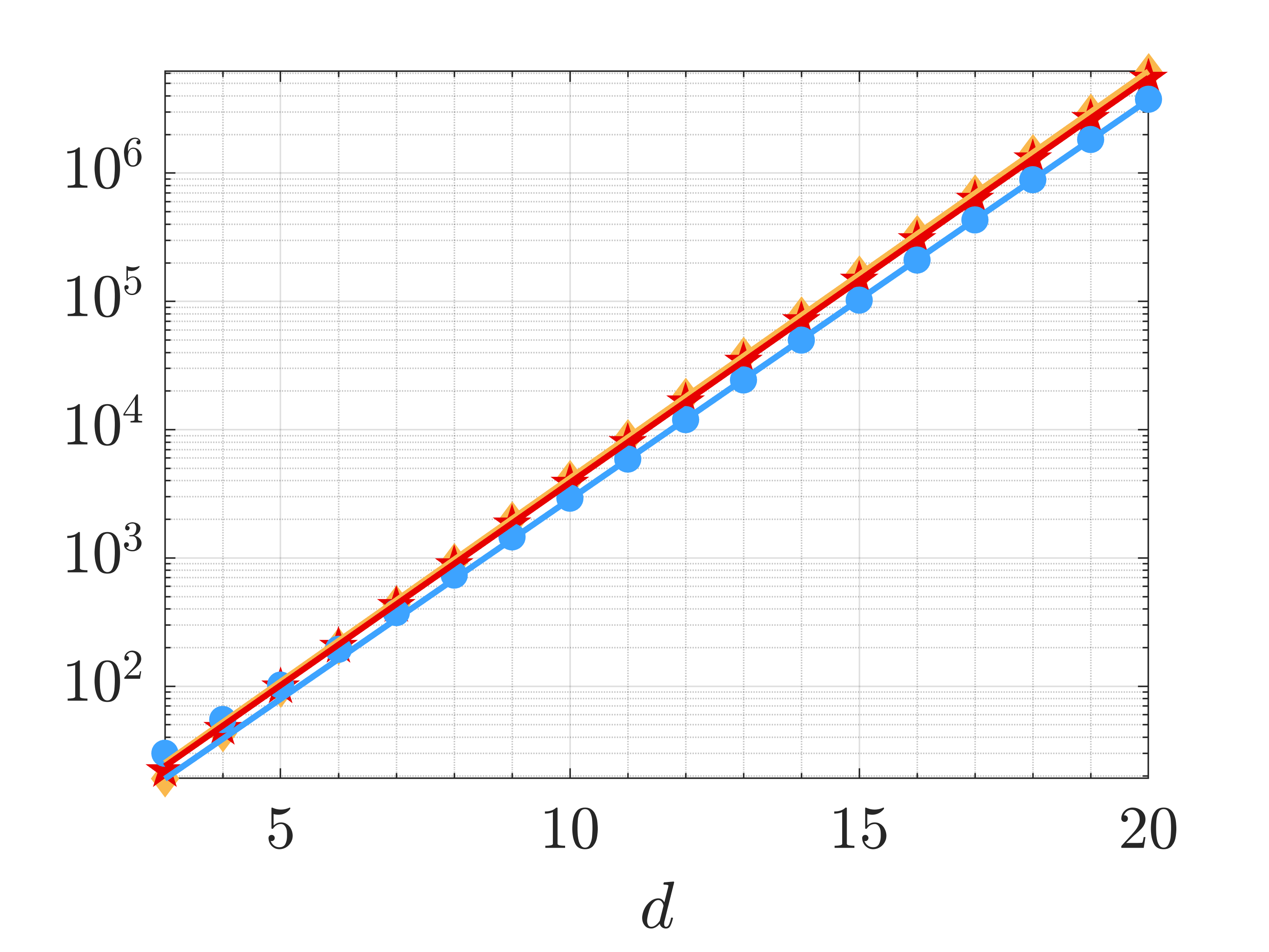}
\end{center}
\caption{\small $n_2(d,s)$ as a function of $d$ for $s=2$ ({\color{Dandelion} $\blacklozenge$}), $s=4$ ({\color{red} $\bigstar$}) and $s=10$ ({\color{Cerulean} $\bullet$}). Left: $\mu$ is uniform in $\SB_d(1)$. Right: $\mu$ is spherically normal.
}
\label{F:nstar-ball_s2-4-10}
\end{figure}

In the non-asymptotic regime, the extreme-value approximation $\widehat{a^*}(d,n,2)$ of $a^*(d,n,2)$ is given by (\ref{hat_a_2}*) for $s=2$; for $s=4$, $\widehat{a^*}(d,n,4)$ is a root of the third-degree polynomial corresponding to the derivative of (\ref{E4^2-general}*).
Figure~\ref{F:bstar_ALLastar_ball_s2_D} displays $\widehat{a^*}(d,n,2)$ as a function of $d$ for two values of $n$: it is practically indistinguishable from $a^*(d,n,2)$. The efficiencies $D_{\mu,s}^{1/s}(\PP_{\widehat{a^*}}^{[n]})/D_{\mu,s}^{1/s}(\PP_{0,b^*}^{[n]})$ and $D_{\mu,s}^{1/s}(\PP_{a^*}^{[n]})/D_{\mu,s}^{1/s}(\PP_{0,b^*}^{[n]})$ are virtually indistinguishable in Figure~\ref{F:ALLeff_ball_s2_D}; both $\PP_{a^*}^{[n]}$ and $\PP_{\widehat{a^*}}^{[n]}$ outperform $\PP_{0,b^*}^{[n]}$ for the values of $n$ and $d$ considered. Moreover, $\PP_{0,b^*}^{[n]}$ significantly outperforms $\PP_{0,1}^{[n]}$; see Figure~\ref{F:ball}.

\section{$\mu$ is a spherically symmetric normal distribution}\label{S:normal}

Here $\mu$ is the probability measure of the normal distribution $\SN(\0b_d,\Ib_d/d)$, with $\Ib_d$ the $d$-dimensional identity matrix; see Section~\ref{S:NSC-examples-normal}*.

Let $\ms^* = \ms^*(d,n,s)$ denote the value of $\ms$ that minimises the distortion $D_{\mu,s}(\mu^{[n]}_\ms)$; as previously $a^* = a^*(d,n,s)$ is the value of $a$ that minimises $D_{\mu,s}(\PP_a^{[n]})$.

The empirical distribution of an optimal quantiser $\Xb_n^*$ minimising $E_{\mu,s}(\Xb_n)$ converges weakly to $\SN(\0b_d,(\ms_\infty^*)^2\Ib_d/d)$ as $n\to\infty$, where $\ms_\infty^*=\sqrt{1+s/d}$; see \cite[Table~7.1]{GrafL2000}. The left panel of Figure~\ref{F:normal-a-sigma-s} (for $d=10$ and $n=100\,000$) illustrates that $\ms^*=\ms^*(d,n,s)$  increases with $s$, as does $\ms_\infty^*$, and that $a^*=a^*(d,n,s)$ also increases with $s$ (see Figure~\ref{F:normal}*-left for the case $n=1\,000$).

\begin{figure}[ht!]
\begin{center}
\includegraphics[width=.49\linewidth]{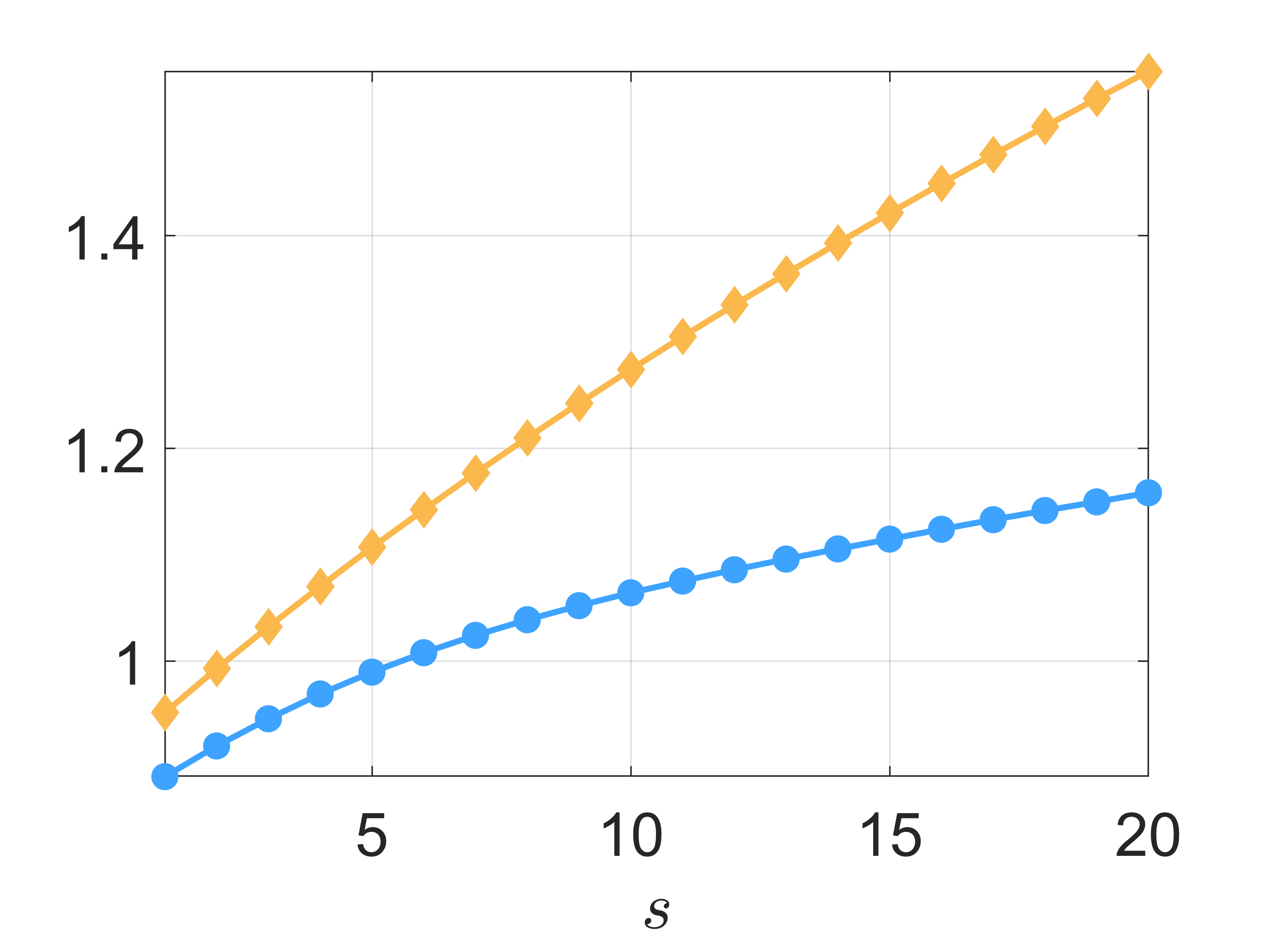}
\includegraphics[width=.49\linewidth]{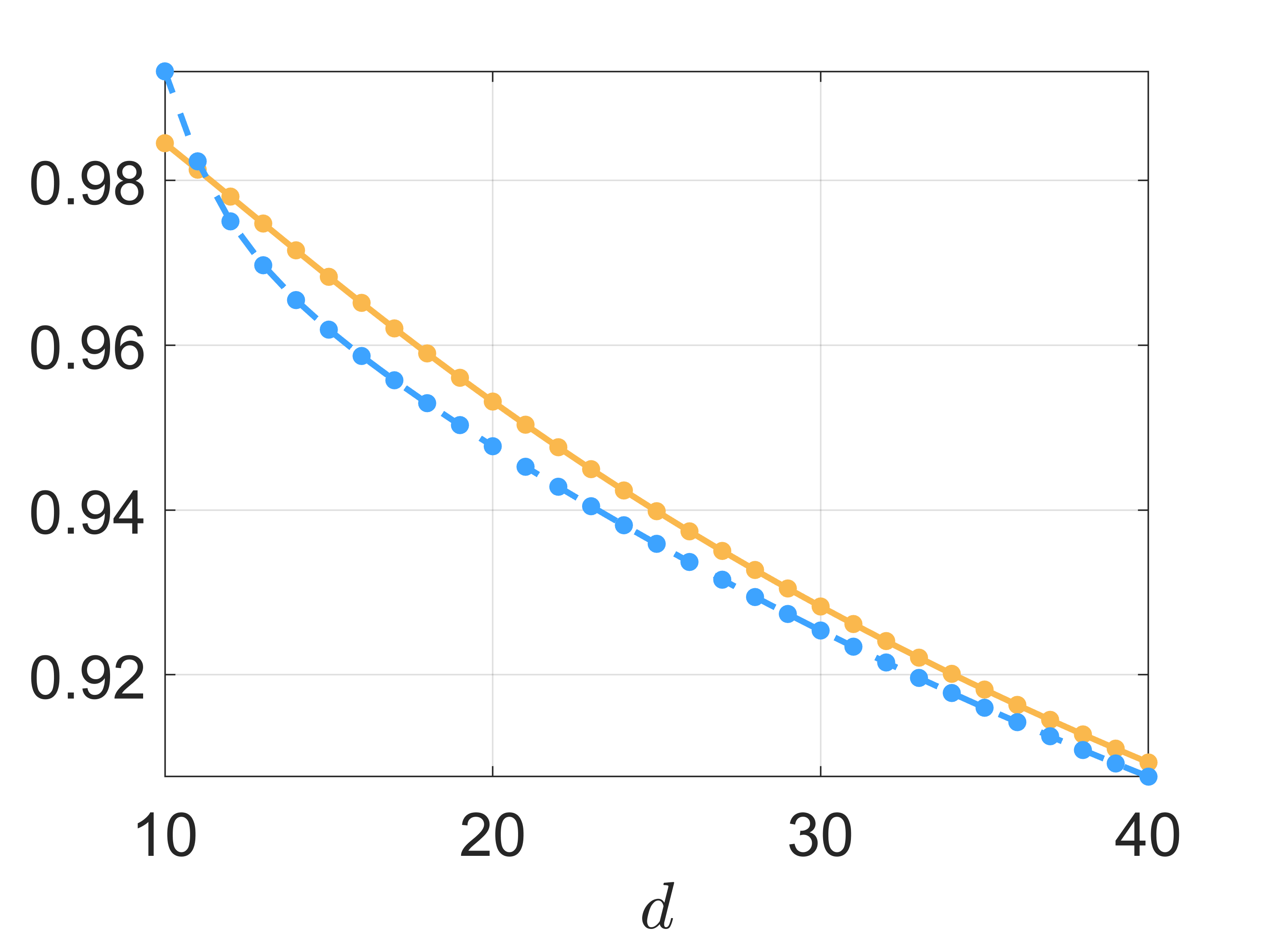}
\end{center}
\caption{\small Left: optimal values $a^*(d,n,s)$ ({\color{Cerulean} $\bullet$}) and $\ms^*(d,n,s)$ ({\color{Dandelion} $\blacklozenge$}) as functions of $s$ for $d=10$ and $n=100\,000$. Right: ratios $D_{\mu,s}^{1/s}(\PP^{[n]}_{a^*})/D_{\mu,s}^{1/s}(\mu^{[n]}_{\ms_\infty^*})$ ({\color{Cerulean} {\bf - - -}}) and $D_{\mu,s}^{1/s}(\mu^{[n]}_{\ms^*})/D_{\mu,s}^{1/s}(\mu^{[n]}_{\ms_\infty^*})$ ({\color{Dandelion} {\bf ---}}) as functions of $d$ for $s=2$ and $n=10\,000$.}
\label{F:normal-a-sigma-s}
\end{figure}

Figure~\ref{F:normal-supp} presents $D_{\mu,s}^{1/s}(\mu^{[n]}_\ms)$ and $D_{\mu,s}^{1/s}(\PP_{a=\ms}^{[n]})$ as functions of $\ms$ for $s=2$, $n=1\,000$ and $d=5$, 10, and $20$.
For $d=5$ (left panel), the asymptotically optimal scale is $\ms_\infty^*=\sqrt{1.4}\simeq 1.1832$. For $d\geq 10$ (centre and right panels), the plots of $D_{\mu,s}^{1/s}(\PP_\ms^{[n]})$ and $D_{\mu,s}^{1/s}(\mu^{[n]}_\ms)$ are close, especially near their minima.
The behaviour is similar when using other values of $s$.

The right panel of Figure~\ref{F:normal-a-sigma-s} presents the efficiencies of the asymptotically optimal distribution $\mu_{\ms_\infty^*}$ relative to $\PP_{a^*}$ and $\mu_{\ms^*}$, for $n=10\,000$, as functions of $d$. These efficiencies are always less than 1, meaning that the asymptotic regime is not yet reached for the values of $d$ considered. At $d=10$, however, we can observe that, contrary to the central panel of Figure~\ref{F:normal-supp} where $n=1\,000$, now $D_{\mu,2}(\PP_{a^*}^{[n]}) > D_{\mu,2}(\mu^{[n]}_{\ms^*})$: with $n=10\,000$ points we approach the asymptotic regime and $\PP_{a^*}$ is dominated by a suitably chosen normal distribution. The smallest $n=n_2(d,s)$ at which this occurs is plotted in the right panel of Figure~\ref{F:nstar-ball_s2-4-10} for $s=2,4,10$ (the case $s=2$ is also presented in Figure~\ref{F:asymptotics-ball_s2}*). These plots suggest that $D_{\mu,s}(\PP_{a^*}^{[n]}) < D_{\mu,s}(\mu^{[n]}_{\ms^*})$ for $n \lesssim n_0\, \ml_0^d$ for some $n_0$ and $\ml_0$. We now investigate the behaviour of $D_{\mu,s}(\PP_{a}^{[n]})$ for other choices of $a$ in a similar asymptotic regime where $n=n(d)$ grows exponentially fast with $d$. We take $n(d)=2^d$. According to the numerical comparisons, $\PP_{a^*}^{[n(d)]}$ is then not optimal among all spherically symmetric distributions, but it still outperforms the optimised normal distribution $\mu^{[n(d)]}_{\ms^*}$.

\begin{figure}[ht!]
\begin{center}
\includegraphics[width=.32\linewidth]{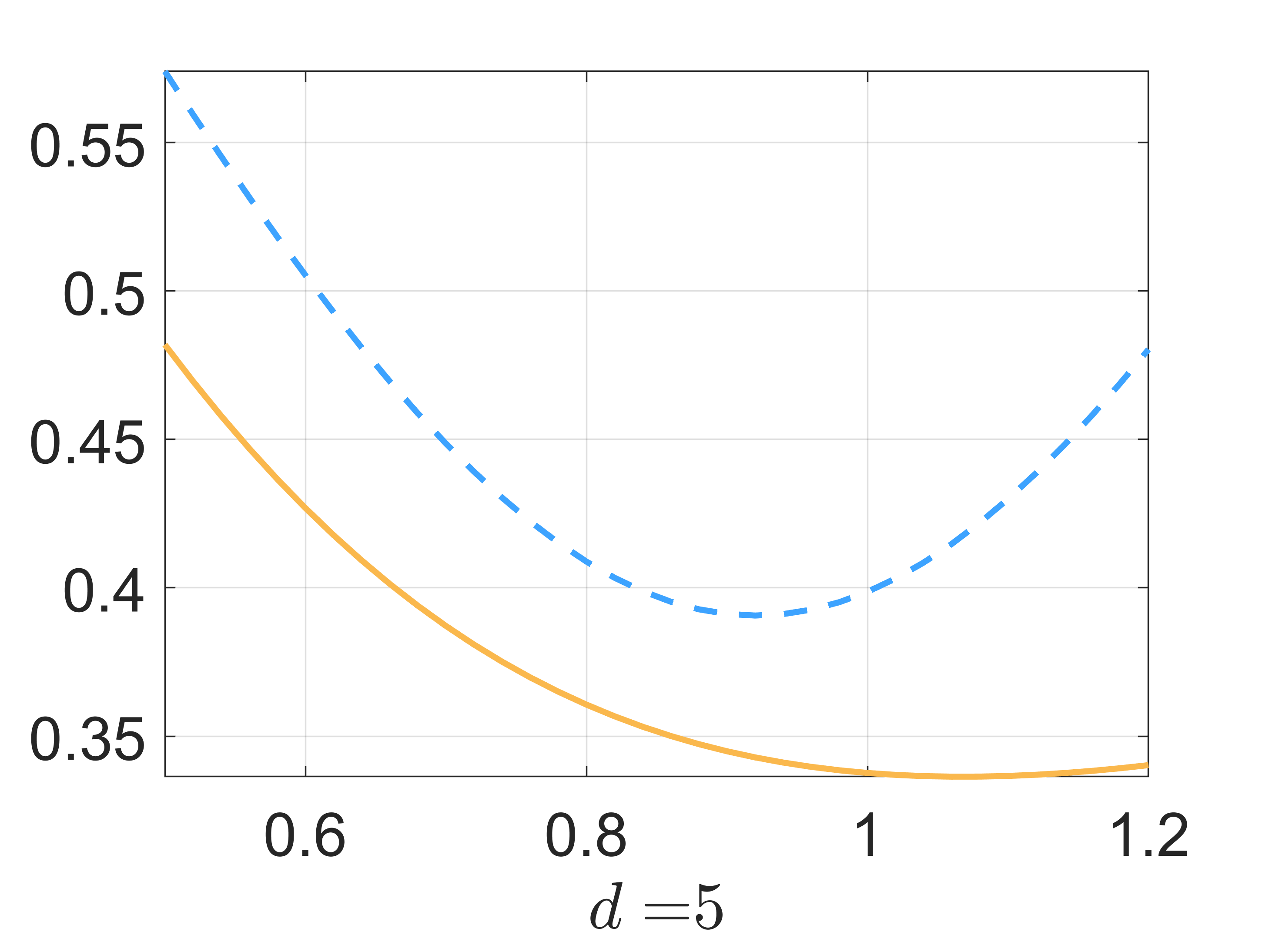}
\includegraphics[width=.32\linewidth]{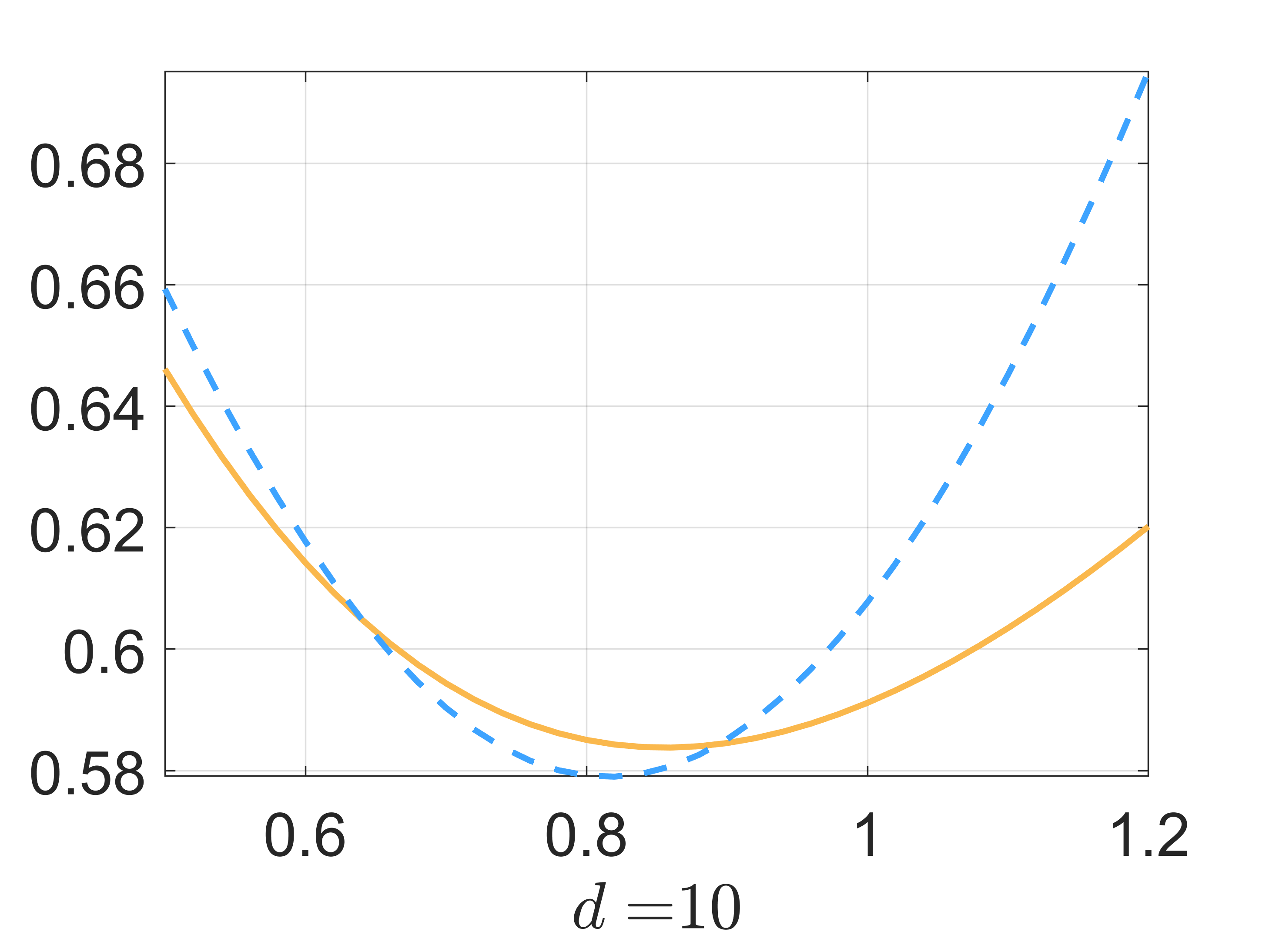}
\includegraphics[width=.32\linewidth]{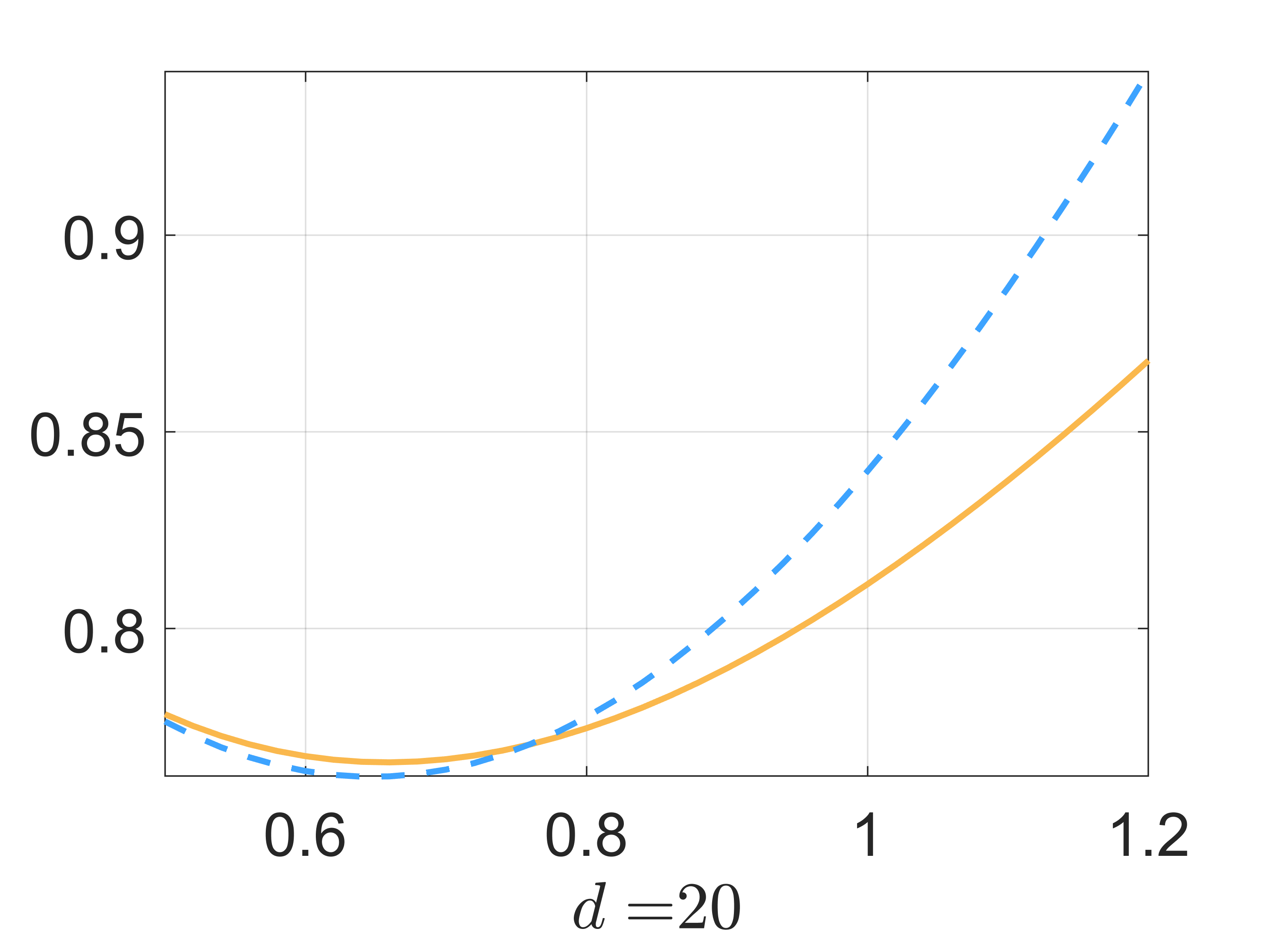}
\end{center}
\caption{\small $D_{\mu,s}^{1/s}(\PP^{[n]}_{a=\ms})$ ({\color{Cerulean} {\bf - - -}}) and $D_{\mu,s}^{1/s}(\mu^{[n]}_\ms)$ ({\color{Dandelion} {\bf ---}}) as functions of $\sigma$ for different $d$ with $n=1\,000$ and $s=2$.}
\label{F:normal-supp}
\end{figure}

Given that $M_{\psi,1}=\Ex\{\|U\|\} \to 1$ and $\var\{\|U\|\} \to 0$ as $d\to\infty$ when $U \sim \mu$, we analyse the following cases for $a$.
(\textit{i}) Uniform distribution on $\SS_{d-1}(1)$:
we consider $a = \widetilde{a^*}(d,n,s)$ which minimises $D_{\widetilde\mu,s}^{1/s}(\PP_a^{[n]})$ for $\widetilde\mu$ uniform on $\SS_{d-1}(1)$, as discussed in Section~\ref{S:sphere}.
(\textit{ii}) Scaled uniform distribution: we also consider $a = M_{\psi,1}\widetilde{a^*}(d,n,s)$, obtained when $\widetilde\mu$ is uniform on $\SS_{d-1}(M_{\psi,1})$.
As we consider the asymptotic regime $n = 2^d$, we additionally examine (\textit{iii}) the limiting value $\widehat{a^*_\ml} = \sqrt{3}/2$, see (\ref{joint-exact-a-limits}*);
(\textit{iv}) the value $\widehat{a^*}(d,s)$, obtained by minimising the extreme-value approximation $\widehat{E}_{\mu,s}^s(n;a)$ defined in (\ref{Ds-approx}*) after the quantile $\kappa_{n,d}$ is replaced by its asymptotic value $(1/2)(1 - \sqrt{3}/2)$. (Formula~(\ref{hat_a_2}*) applies only when $s=2$.)

The left panel of Figure~\ref{F:normal_n2d_d3-20_s4} displays the four parameters $\ms^*(d,n,s)$, $a^*(d,n,s)$, $\widetilde{a^*}(d,n,s)$, and $M_{\psi,1}\widetilde{a^*}(d,n,s)$ as functions of $d$ for $s=4$, together with $\widehat{a^*_\ml}$ (indicated by a horizontal line) and $\widehat{a^*}(d,s)$.
The right panel in the same figure shows the efficiency of the optimal normal distribution $\mu_{\ms^*}$ compared to $\PP_a$ uniform on $\SS_{d-1}(a)$, that is, $D_{\mu,s}^{1/s}(\PP_a^{[n]})/D_{\mu,s}^{1/s}(\mu_{\ms^*}^{[n]})$, for the five choices of $a$ considered: these five quantiser distributions $\PP_{a}$ perform better than $\mu_{\ms^*}$ (for $d\geq 4$ when $a=M_{\psi,1}\widetilde{a^*}(d,n,s)$). Note that the efficiencies for $\widehat{a^*_\ml}$ (purple triangles) and $a^*(d,n,s)$ (blue dots) almost coincide for $d\geq 5$.

\begin{figure}[ht!]
\begin{center}
\includegraphics[width=.49\linewidth]{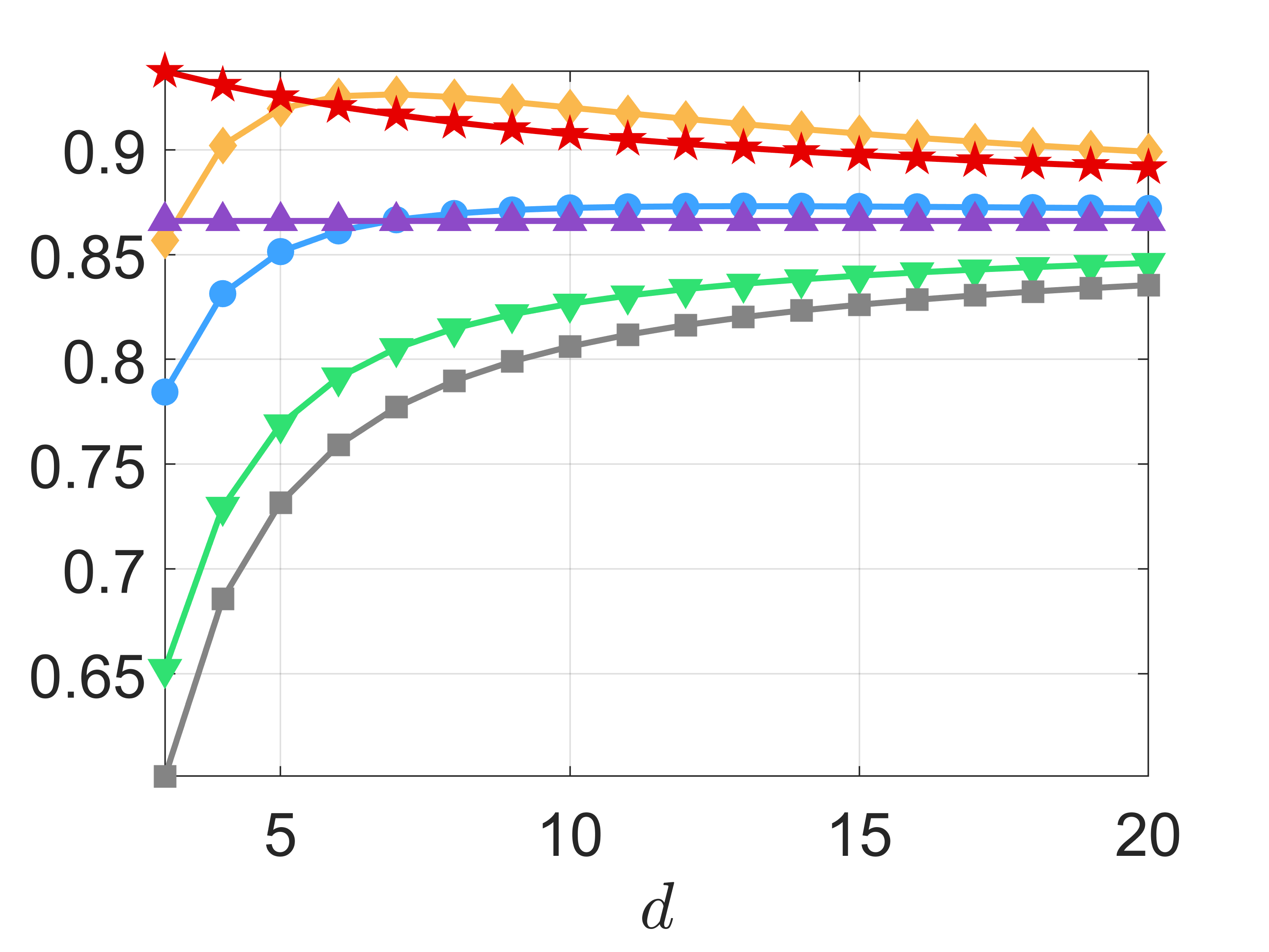}
\includegraphics[width=.49\linewidth]{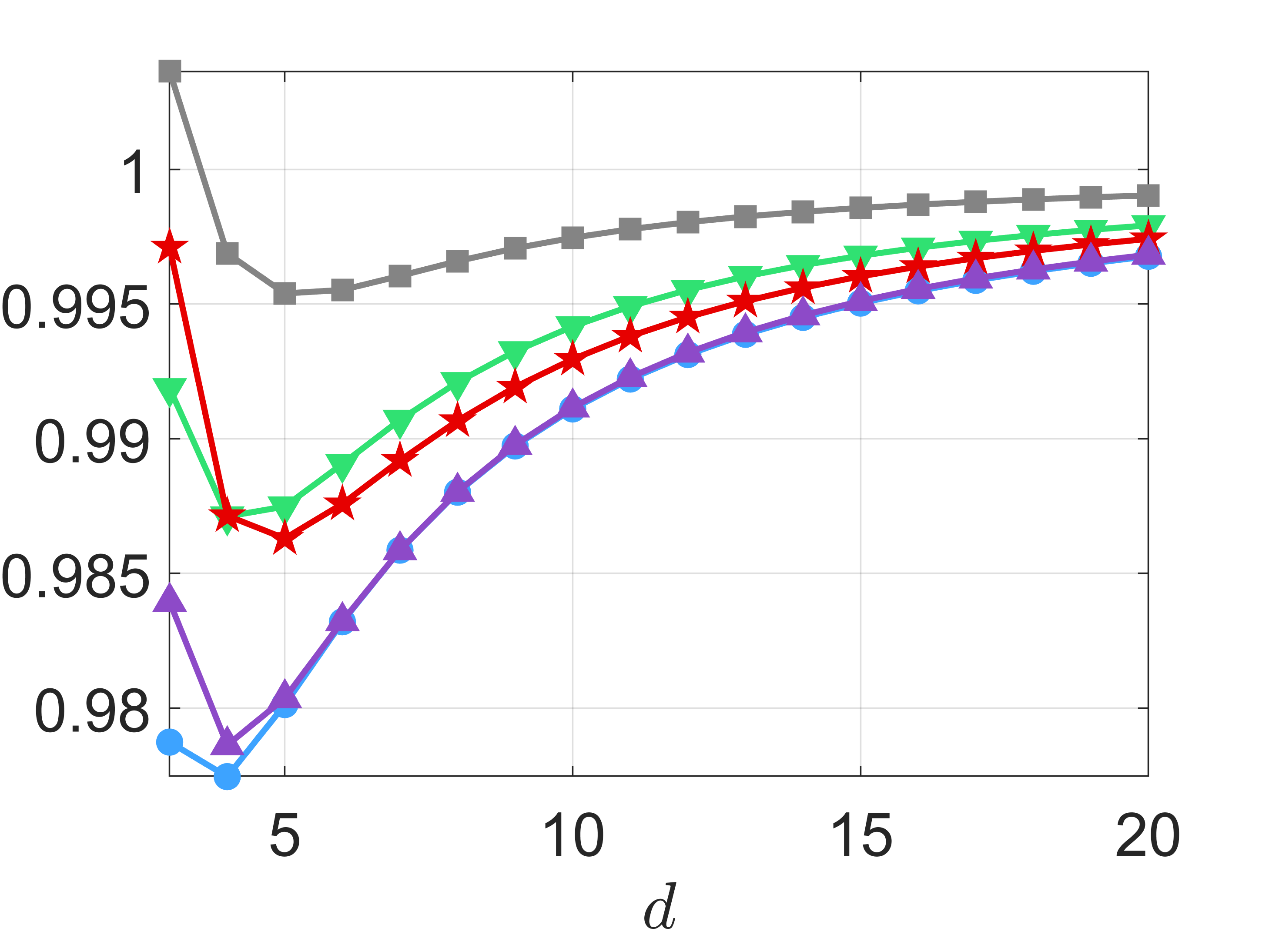}
\end{center}
\caption{\small  Left: $\ms^*(d,n,s)$ ({\color{Dandelion} $\blacklozenge$}),
$a^*(d,n,s)$ ({\color{Cerulean} $\bullet$}),
$\widetilde{a^*}(d,n,s)$ ({\color{green} $\blacktriangledown$}),
$M_{\psi,1}\widetilde{a^*}(d,n,s)$ ({\color{gray} {\tiny $\blacksquare$}}),
$\widehat{a^*_\ml}=\sqrt{3}/2$ ({\color{Orchid} $\blacktriangle$}) and
$\widehat{a^*}(d,s)$ ({\color{red} $\bigstar$}) as functions of $d$.
Right: efficiencies $D_{\mu,s}^{1/s}(\PP_a^{[n]})/D_{\mu,s}^{1/s}(\mu_{\ms^*}^{[n]})$ as functions of $d$ for the choices of $a$ indicated on the left panel ($n=2^d$, $s=4$).}
\label{F:normal_n2d_d3-20_s4}
\end{figure}

\begin{remark}
From the results in Section~\ref{S:n-growing-with-d}*, the four quantities $a^*(d,n,s)$, $\widetilde{a^*}(d,n,s)$,
$M_{\psi,1}\widetilde{a^*}(d,n,s)$ and
$\widehat{a^*}(d,s)$ tend to
$\widehat{a^*_\ml}=\sqrt{3}/2$ when $d\to\infty$. Our numerical results suggest that, in the considered asymptotic regime, $\ms^*(d,n,s)$ also converges to the limiting value $\widehat{a^*_\ml} = \sqrt{3}/2$. However, a rigorous theoretical analysis is still required to confirm this observation. The same is true for the asymptotic behaviour of $b^*(d,n,s)$ in the optimised quantiser $\PP_{0,b^*}^{[n]}$ when $\mu$ is uniform in $\SB_d(1)$. From (\ref{eq:quantile2}*) in Proposition~\ref{th:dist}*, we have to analyse the behaviour of $d(\ub,\Rb_n)=\min_{i=1,\ldots,n} \left[(\|\ub\|-R_i)^2 +4\,\|\ub\|\,R_i \, \zeta_i\right]$ when $n\to\infty$, where the $R_i$ are i.i.d., the $\zeta_i$ are i.i.d.\ with $\zeta_i\eqd\beta_{\delta,\delta}$, and the two families are mutually independent. A norm-concentration property for $R=\|X\|$ as $d\to\infty$ does not by itself imply an extreme-value limit for $d(\ub,\Rb_n)$. Even the cases where $\PP=\PP_{0,b}$ is uniform on $\SB_d(b)$ or where $\PP=\mu_\ms$ is multivariate normal pose significant challenges.
\fin
\end{remark}

